\documentclass[abstract=true, DIV=12]{scrartcl}

\usepackage{preamble}
\title{A rational model for the fiberwise~THH~transfer~I: Sullivan~algebras}
\author{Florian Naef \and Robin Stoll}
\date{\today}

\begin{document}

\maketitle

\begin{abstract}
  Given a map $f$ of fibrations over a space $B$ such that the fiber of $f$ is simply connected and finitely dominated, we prove that its fiberwise THH transfer, considered as a map of parametrized spectra over $B$, is rationally modeled by the Hochschild homology transfer of a Sullivan model of $f$.
  The proof goes in two steps.
  Firstly, we use the machinery of higher categorical traces to show that the fiberwise THH transfer can be computed internally to parametrized spectra.
  Secondly, we model the resulting description rationally using work of Braunack-Mayer, who proved that parametrized spectra can be modeled by modules over Sullivan algebras.
  In Part~II, we will use our result to obtain a rational model of the Becker--Gottlieb transfer, and for applications to manifold topology.
\end{abstract}

\setcounter{tocdepth}{\subsectiontocdepth}
\tableofcontents

\section{Introduction}

For a fibration $f \colon X \to Y$ whose fibers are equivalent to finite CW-complexes (or, more generally, are finitely dominated), there is a ``wrong-way'' map for the suspension spectra of the free loop spaces
\[ f^*  \colon  \SS[\L Y]  \longto  \SS[\L X] \]
dubbed the \emph{free loop transfer} by Lind--Malkiewich \cite{LM}.
It is obtained by observing that $\SS[\L X]$ is equivalent to $\THH(X)$, i.e.\ topological Hochschild homology of the full subcategory of compact objects in the stable $\infty$-category of parametrized spectra over $X$ (i.e.\ functors from $X$ considered as an $\infty$-groupoid to spectra).
Pulling back parametrized spectra along a map $f$ as above preserves compact objects, and hence induces a map
\[ f^* \colon \THH(Y) \longto \THH(X) \]
called the \emph{THH transfer}.
This turns out to be related to many other constructions; for example, Lind--Malkiewich proved that one can recover the Becker--Gottlieb transfer $\SS[Y] \to \SS[X]$ from the THH transfer.
We will return to this point in \cite{II}.

Given another fibration $Y \to B$, we can form the THH transfer of $f$ fiberwise over $B$ to obtain the \emph{fiberwise THH transfer}: a map
\[ f^* \colon \THH_B(Y) \longto \THH_B(X) \]
of parametrized spectra over $B$; here $\THH_B(E)$ denotes fiberwise THH of a space $E$ over $B$.
When $Y = B$, it is also called the \emph{fiberwise THH Euler characteristic}.
Via the Dennis trace, it is related to the fiberwise \emph{A-theory} Euler characteristic (i.e.\ the analogous construction for algebraic K-theory instead of THH), which plays a central role in the topological Riemann--Roch theorem of Dwyer--Weiss--Williams \cite{DWW}.
Following a suggestion of Manuel Krannich, this relationship in fact served as one of our main motivations, and we heavily exploit it in \cite{II}.
Lastly let us mention that, by work of Naef--Safronov \cite{NS}, the THH Euler characteristic measures the failure of homotopy invariance of the Goresky--Hingston coproduct in string topology.

\subsection*{Main result}

Our main result is that a rational model for the fiberwise THH transfer is given by the classical transfer of Hochschild homology; see below for its definition.

\begin{introtheorem}[see \cref{thm:transfer_model_cdga}] \label{intro:thm:cdga}
  Let $f \colon X \to Y$ and $p \colon Y \to B$ be fibrations between nilpotent spaces of finite rational type that are respectively modeled by cofibrations $\phi \colon R \to S$ and $\iota \colon \k \to R$ of cdgas.
  Assume that the fiber of $f$ is simply connected and finitely dominated, and that the fiber of $p$ is simply connected.
  Then the fiberwise THH transfer $f^* \colon \THH_B(Y) \to \THH_B(X)$ is modeled by the Hochschild homology transfer $\phi^* \colon \HH_\k(S) \to \HH_\k(R)$.
\end{introtheorem}

Recall that by work of Sullivan \cite{Sul}, there is a contravariant equivalence between the rational homotopy category of nilpotent spaces of finite rational type, and the homotopy category of connected commutative differential graded algebras over $\QQ$ (\emph{cdgas} for short) of finite type.
A cdga $\k$ is a \emph{model} for a space $B$ if they correspond to each other under this equivalence.
In this situation, Braunack--Mayer \cite{Bra} furthermore showed that the rational homotopy category of nilpotent, bounded below parametrized spectra over $B$ of finite rational type is equivalent to the homotopy category of $\k$-modules of finite type (previous results in this direction were obtained by Félix--Murillo--Tanré \cite{FMT}).
This is the sense in which the map $\HH_\k(S) \to \HH_\k(R)$ in the homotopy category of $\k$-modules models the fiberwise THH transfer (we denote by $\HH_\k(A)$ the Hochschild complex of a $\k$-algebra $A$).

Note that, in the non-parametrized setting, it is well-known that $\THH(X)$ is modeled by the Hochschild homology $\HH_\QQ(S)$ of a cdga model $S$ of $X$ (see e.g.\ \cite{Lod15}) and it is folklore that the THH transfer is modeled by the Hochschild homology transfer.
Hence our main contribution is to upgrade this to a parametrized result.

\paragraph{The Hochschild homology transfer.}
A map $\phi \colon R \to S$ of $\k$-algebras induces a transfer map on Hochschild homology when $S$ is perfect (i.e.\ derived dualizable) as an $R$-module.
In this case, there is a canonical coevaluation map $\coev \colon S \to S \dertensor_R \dual{S}_R$ in the homotopy category of $S$-bimodules, where $\dual{S}_R$ denotes the derived $R$-dual of $S$.
The Hochschild homology transfer $\phi^*$ is then defined to be the following composite
\[ \HH_\k(S)  \xlongto{\coev_*}  \HH_\k(S, S \dertensor_R \dual{S}_R)  \eq  \HH_\k(R, \dual{S}_R \dertensor_S S)  \xlongto{\ev_*}  \HH_\k(R) \]
(see e.g.\ Keller \cite[§5]{Kel21}).
The middle equivalence can be visualized using the following picture
\[
\begin{tikzcd}[column sep = 10]
  S \vphantom{\dual{S}_R} \rar[start anchor = north, end anchor = north, bend left = 90, dash]{\dertensor_S} & \lar[start anchor = south, end anchor = south, bend left = 90, dash]{\dertensor_R} \dual{S}_R
\end{tikzcd}
\qquad = \qquad
\begin{tikzcd}[column sep = 10]
  \dual{S}_R \rar[start anchor = north, end anchor = north, bend left = 90, dash]{\dertensor_R} & \lar[start anchor = south, end anchor = south, bend left = 90, dash]{\dertensor_S} S \vphantom{\dual{S}_R}
\end{tikzcd}
\]
where we recall that the Hochschild complex of an $S$-bimodule is obtained by derived tensoring the left and right $S$-module structures together.

There is an alternative definition of the Hochschild homology transfer that is less explicit, but more similar to the definition of the THH transfer.
Namely, the construction of the Hochschild complex can be extended to categories enriched in $\k$-modules such that the Hochschild complex of the category $\mathrm{Perf}_A$ of perfect modules over a $\k$-algebra $A$ is equivalent to $\HH_\k(A)$ (see McCarthy \cite{McC} or Keller \cite{Kel99}).
Restriction along a map of $\k$-algebras $\phi \colon R \to S$ as above induces a $\k$-linear functor $\phi^* \colon \mathrm{Perf}_S \to \mathrm{Perf}_R$.
This induces a transfer map $\phi^* \colon \HH_\k(S) \to \HH_\k(R)$, and it is folklore that this agrees with the construction above.
Using this definition, it is possible to give a more direct proof of the non-parametrized version of \cref{intro:thm:cdga} (i.e.\ $B = *$ and $\k = \QQ$).
However, this does not lead to a simpler proof of the parametrized result; one would still have to compare the $X$-parametrized Hochschild homology transfer with the transfer in $\HH_\k$, which is the majority of the work.

\paragraph{Equivariant models.}
The condition of \cref{intro:thm:cdga} that the spaces involved are nilpotent can be inconvenient; for example it is usually not fulfilled in the case of the universal fibration $F / \aut(F) \to \B \aut(F)$ with fiber $F$.
However, if $E \to B$ is a fibration of connected spaces with simply connected fiber, one can always find a (discrete) group $G$ and a map $B \to \B G$ such that the (homotopy) fibers of $E$ and $B$ over $\B G$ are nilpotent; for example one can choose $G$ to be $\pi_1(B)$.
In that case the map $E \to B$ can be modeled by a map of cdgas with an action of $G$ (assuming that the universal coverings of $E$ and $B$ are of finite rational type); see \cite{BS,GHT} for elaborations on this idea.
An explicit equivariant model of this kind for the classifying space $\B \aut(F)$ was constructed by Berglund--Zeman \cite{BZ}, and this was extended to a model for the classifying space of block diffeomorphisms of a manifold of dimension $\ge 6$ (with spherical boundary) by Berglund--Stoll \cite{BS24}.

Following these ideas, we extend the work of Braunack--Mayer on models for parametrized spectra to the $G$-equivariant context (see \cref{sec:eq}), and use this to prove a version of \cref{intro:thm:cdga} for maps of spaces over $\B G$ that are fiberwise nilpotent; see \cref{thm:transfer_model_cdga}.

\paragraph{The proof.}
We use work on higher categorical traces by Hoyois--Scherotzke--Sibilla \cite{HSS} and Carmeli--Cnossen--Ramzi--Yanovski \cite{CCRY} to show that the fiberwise THH transfer can be computed internally to parametrized spectra; see \cref{lemma:transfer_diag}.
We then use the work of Braunack--Mayer to model the resulting description algebraically, thus obtaining \cref{intro:thm:cdga}.
This involves producing rational models for various constructions of the six-functor calculus on parametrized spectra, in particular the dualizing spectrum of Klein \cite{Kle}, which might be of independent interest (see \cref{sec:model_Wirthmueller,sec:model_dualizing}).

Due to our reliance on higher categorical methods and results, most of this paper is written in the language of $\infty$-categories.
To facilitate our use of equivariant models, in \cref{sec:sections} we deduce from work of Harpaz \cite{Har} and Dugger \cite{Dug} that a groupoidal limit of (not necessarily simplicial) model categories is a model for the limit of the underlying $\infty$-categories.

\paragraph{Further directions.}
Keller \cite{Kel98} has also constructed a transfer for cyclic homology of dgas.
It seems likely to the authors that it is a rational model for the fiberwise transfer of rational negative cyclic homology $\HCm(\blank; \QQ)$, i.e.\ the homotopy fixed points of the $\Sphere 1$-action on $\THH(\blank) \tensor \H\QQ$.\footnote{We would like to thank Samuel Muñoz-Echániz and Manuel Krannich for explaining to us that this is not equivalent to the rationalization of $\TCm$ (i.e.\ the homotopy fixed points of the $\Sphere 1$-action on $\THH$), which we had implicitly assumed in a previous version of this introduction.}
We thus formulate the following.

\begin{introconjecture} \label{intro:conj:HC}
  In the situation of \cref{intro:thm:cdga}, the fiberwise transfer of rational negative cyclic homology $f^* \colon \HCm_B(Y; \QQ) \to \HCm_B(X; \QQ)$ is modeled by the transfer of cyclic homology $\phi^* \colon \HC_\k(S) \to \HC_\k(R)$.
\end{introconjecture}

By work of Goodwillie \cite{Goo}, rational negative cyclic homology is rationally closely related to algebraic K-theory.
In particular \cref{intro:conj:HC} would provide a rational model for a close approximation of the fiberwise A-theory transfer.

In a different direction, it should be straightforward to generalize our methods to obtain a rational model for the fiberwise Reidemeister trace of a fiberwise self-map $g$, i.e.\ the analog of the THH Euler characteristic for THH twisted by $g$.
This is an obstruction for $g$ to be fiberwise homotopic to a map without fixed points; in fact Klein--Williams \cite{KW} showed that it is a complete obstruction under certain assumptions on the base and fiber (also see Ponto \cite{Pon}).

\paragraph{Applications.}
In \cite{II}, we provide an explicit formula for the Hochschild homology transfer in terms of $\Ainf$-algebras.
Combined with \cref{intro:thm:cdga}, this yields an explicit rational model for the fiberwise THH transfer.
We then use it for the following applications:
\begin{itemize}
  \item We obtain a rational model for the Becker--Gottlieb transfer, using that it can be expressed in terms of the THH transfer by a result of Lind--Malkiewich \cite{LM}.
  In \cite{Transfer}, we use this rational model to prove that the Becker--Gottlieb transfer is not functorial.
  \item We consider the characteristic classes constructed by Berglund \cite{Ber} for fibrations with fiber a Poincaré complex (which generalize classes found by Berglund--Madsen \cite{BM}); they are defined via the Lie graph complex, and we prove that the classes corresponding to non-trivalent graphs with exactly one loop vanish when evaluated on fiber bundles with fiber a compact simply connected topological manifold.
  \item We provide a rational model for the space of fiberwise THH-simple structures, which is a step towards obtaining rational models for the classifying spaces of diffeomorphisms and homeomorphisms of a compact simply connected manifold in the rational concordance stable range.
\end{itemize}

\subsection*{Acknowledgments}

First and foremost, we would like to thank Nils Prigge, who was involved in the beginning of this project but declined to be an author.
Secondly, we would like to thank Manuel Krannich for suggesting to us that a rational model for the fiberwise A-theory transfer might allow to prove vanishing of graph characteristic classes, which started this project.
Moreover we would like to thank Fabian Hebestreit for repeatedly explaining various parts of the six functor calculus on parametrized spectra to us, and Maxime Ramzi for answering our questions about higher categorical traces.
We would furthermore like to thank Alexander Berglund, Cary Malkiewich, Samuel Muñoz-Echániz, Thomas Nikolaus, Oscar Randal-Williams, George Raptis, Wolfgang Steimle, Jan Steinebrunner, Ferdinand Wagner, and Kelly Wang for helpful discussions.

While working on this project, Robin Stoll was supported by a postdoctoral scholarship of the Knut and Alice Wallenberg foundation.

\section{Preliminaries}

In this section, we fix our notation, recall basic constructions and facts we will need throughout the main body of the paper, and prove some elementary lemmas we did not find in the literature.

\subsection{Model categories, \texorpdfstring{$\infty$}{infinity}-categories, animas, and spectra}

We will assume familiarity with model categories and the language of higher category theory as covered in \cite[§1 and Appendix A]{LurHTT}, and the basics of higher algebra, for example as covered in \cite{Gep}.
In particular the term \emph{$\infty$-category} will mean quasi-category, though we try to work model independently where possible.
We will ignore set-theoretic issues throughout, passing to a higher universe where necessary.

\begin{notation}
  We write $\Catinf$ for the $\infty$-category of $\infty$-categories and $\Catinfmon$ for the $\infty$-category of symmetric monoidal $\infty$-categories and (strong) symmetric monoidal functors.
  We furthermore write $\PrL \subset \Catinf$ for the subcategory consisting of the presentable $\infty$-categories and the left adjoint functors, and $\PrLst \subset \PrL$ for the full subcategory spanned by the stable presentable $\infty$-categories.
\end{notation}

Recall that $\Catinfmon$ is a subcategory of $\Fun(\Finpt, \Catinf)$, where $\Finpt$ denotes the category of finite pointed sets.

\subsubsection*{2-categories}

\begin{notation}
  By \emph{$2$-category} we will mean a bicategory; a \emph{pseudofunctor} is a homomorphism of $2$-categories, and a \emph{pseudonatural transformation} is a transformation of pseudofunctors.
  A pseudofunctor is \emph{normal} if it strictly preserves identity $1$-morphisms; it is \emph{strict} if it additionally strictly preserves composition.
  A $2$-category is \emph{strict} if all associators and unitors are identities.
  A \emph{$(2,1)$-category} is a $2$-category where all $2$-morphisms are isomorphisms.
  We write $\twoCat$ for the strict $2$-category of categories.
\end{notation}

Recall that, by a result of Power \cite[§4.2]{Pow}, any pseudofunctor from a $2$-category to $\twoCat$ can be replaced by a strict functor up to pseudonatural equivalence.
The same result holds for strict $2$-categories of categories equipped with some extra structure, by transporting this structure along equivalences of categories.

\begin{notation}
  For a category $\cat C$, we denote by $\Nerve(\cat C)$ its nerve considered as a simplicial set and $\infty$-category.
  We will occasionally implicitly consider a category to be an $\infty$-category via its nerve.
  For a $2$-category $\cat C$, we denote by $\Duskin(\cat C)$ its Duskin nerve (see e.g.\ \kerodon{009T}).
\end{notation}

Recall that the Duskin nerve is functorial in normal pseudofunctors.
Furthermore recall that there is a canonical isomorphism $\Duskin(\cat C) \iso \Nerve(\cat C)$ if $\cat C$ is a $1$-category, and that $\Duskin(\cat C)$ is an $\infty$-category if (and only if) $\cat C$ is a $(2,1)$-category (see e.g.\ \kerodon{00AC}).

\subsubsection*{Localizations of categories}

\begin{notation}
  Given an $\infty$-category $\cat C$ and a class of morphisms $W$ of $\cat C$ that contains all equivalences, we denote by $\loc {\cat C} W$ the localization of $\cat C$ at $W$ (see \cite[§1.3.4]{LurHA}).
\end{notation}

An explicit model for the localization $\loc {\cat C} W$ is provided by a fibrant replacement of the pair $(\cat C, W)$ in the model category $\sSetmark$ of marked simplicial sets (see \cite[§3.1]{LurHTT}).
For $\cat C$ an (ordinary) category, this construction is pseudofunctorial in a sense we will now explain.

\begin{notation}
  We denote by $\RelCat$ the strict $(2,1)$-category of pairs $(\cat C, W_{\cat C})$ of a category $\cat C$ and a class of morphisms $W_{\cat C}$ of $\cat C$ that contains all isomorphisms, functors $F \colon \cat C \to \cat D$ such that $F(W_{\cat C}) \subseteq W_{\cat D}$, and natural isomorphisms.
\end{notation}

Considering $\RelCat$ as a simplicial category by applying the nerve to each hom-category, the nerve induces a simplicial functor $\RelCat \to \sSetmark$ (here we use the simplicial enrichment of $\sSetmark$ given by $\Mapmarkueq$, which makes it into a simplicial model category).
Composing with a simplicial fibrant replacement functor of $\sSetmark$ and passing to homotopy coherent nerves, we obtain a functor
\begin{equation} \label{eq:loc}
  \Duskin(\RelCat)  \longto  \Catinf, \qquad (\cat C, W)  \longmapsto  \loc {\cat C} {W}
\end{equation}
(where we use \kerodon{00KY} to identify the Duskin nerve with the homotopy coherent nerve of the underlying simplicial category).

\subsubsection*{The underlying \texorpdfstring{$\infty$}{infinity}-category of a model category}

\begin{definition}
  Let $\cat M$ be a model category.
  We denote by $\fibobj{\cat M}$ resp.\ $\cofibobj{\cat M}$ the full subcategories spanned by the fibrant resp.\ cofibrant objects, and set $\bifibobj{\cat M} \defeq \fibobj{\cat M} \intersect \cofibobj{\cat M}$.
  Given any subcategory $\cat M'$ of $\cat M$, we denote by $\We{\cat M} \subset \cat M'$ the wide subcategory consisting of the weak equivalences.
  We write
  \[ \Underlying(\cat M)  \defeq  \loc {\cat M} {\We{\cat M}}  \qquad  \fUnderlying(\cat M)  \defeq  \loc {\fibobj{\cat M}} {\We{\cat M}}  \qquad  \cUnderlying(\cat M)  \defeq  \loc {\cofibobj{\cat M}} {\We{\cat M}} \]
  and call $\Underlying(\cat M)$ the \emph{underlying $\infty$-category of $\cat M$}.
\end{definition}

Recall that the homotopy category of $\Underlying(\cat M)$ is naturally equivalent to the homotopy category of $\cat M$ (see e.g.\ \kerodon{01MW}).
Furthermore note that the canonical inclusions
\[
\begin{tikzcd}
  (\bifibobj{\cat M}, \We{\cat M}) \rar \dar & (\fibobj{\cat M}, \We{\cat M}) \dar \\
  (\cofibobj{\cat M}, \We{\cat M}) \rar & (\cat M, \We{\cat M})
\end{tikzcd}
\]
induce equivalences upon passing to localizations, with quasi-inverses provided by (co)fibrant replacement functors.
In particular we have canonical equivalences
\begin{equation} \label{eq:underlying}
  \cUnderlying(\cat M) \eq \Underlying(\cat M) \eq \fUnderlying(\cat M)
\end{equation}
of $\infty$-categories.

\begin{notation}
  We denote by $\ModCatC$ the strict $(2,1)$-category of model categories, functors that preserve cofibrant objects and weak equivalences between cofibrant objects, and natural isomorphisms.
  We denote by $\ModCatL \subset \ModCatC$ the wide $(2,1)$-subcategory spanned by the left Quillen functors (and again all natural isomorphisms).
  Dually we define strict $(2,1)$-categories $\ModCatR \subset \ModCatF$.
  Lastly, we write $\ModCatmon$ for the strict $(2,1)$-category of symmetric monoidal model categories whose unit is cofibrant, (strong) symmetric monoidal left Quillen functors, and monoidal natural isomorphisms (see e.g.\ \cite[Definition~4.2.16]{Hov}).
\end{notation}

Note that there is a strict pseudofunctor $\ModCatC \to \RelCat$ given by $\cat M \mapsto (\cofibobj{\cat M}, \We{\cat M})$, and dually a strict pseudofunctor $\ModCatF \to \RelCat$.
In particular we obtain functors
\[ \cUnderlying  \colon  \Duskin(\ModCatC)  \longto  \Catinf  \qquad \text{and} \qquad  \fUnderlying  \colon  \Duskin(\ModCatF)  \longto  \Catinf \]
by composing with the localization functor \eqref{eq:loc}.
We furthermore recall the following from \cite[Example~4.1.7.6]{LurHA}.

\begin{definition}
  Let $\cat M$ be a symmetric monoidal model category with cofibrant unit.
  The \emph{underlying symmetric monoidal $\infty$-category} of $\cat M$ is the composite
  \[ \monUnderlying(\cat M)  \colon  \Nerve(\Finpt)  \xlongto{\cat M^\tensor}  \Duskin(\ModCatC)  \xlongto{\cUnderlying}  \Catinf\]
  where $\cat M^\tensor$ is the pseudofunctor $\Finpt \to \ModCatC$ that sends $n$ to $\cat M^{\times n}$.
\end{definition}

Note that $\monUnderlying(\cat M)$ is indeed a symmetric monoidal $\infty$-category since the localization functor \eqref{eq:loc} preserves finite products (see \cite[Proposition~4.1.7.2]{LurHA}).
In particular this construction yields a functor
\[ \monUnderlying  \colon  \Duskin(\ModCatmon)  \longto  \Catinfmon \]
sending $\cat M$ to its underlying symmetric monoidal $\infty$-category.

Using the equivalences of \eqref{eq:underlying}, left and right Quillen functors also induce maps on $\Underlying$ as follows.

\begin{notation}
  Given a left Quillen functor $L \colon \cat M \to \cat N$, we write $\Lder L$ for the composite
  \[ \Underlying(\cat M)  \xlongto{q}  \cUnderlying(\cat M)  \xlongto{\cUnderlying(L)}  \cUnderlying(\cat N)  \longto  \Underlying(\cat N) \]
  where $q$ is a cofibrant replacement functor.
  Dually we define $\Rder R$ for a right Quillen functor $R$.
  We will also write $\Lder L$ and $\Rder R$ for the induced functors on homotopy categories.
\end{notation}

By \cite[Proposition~1.5.1]{Hin}, given a Quillen adjunction $L \dashv R$, there is an induced adjunction $\Lder L \dashv \Rder R$ of $\infty$-categories.
If $L \dashv R$ is a Quillen equivalence, then the induced adjunction is an equivalence.

\subsubsection*{Natural transformations of Quillen functors and their mates}

The following composite, where $\hocat$ is (the restriction of) the homotopy category functor of \kerodon{025K},
\[ \Duskin(\ModCatL)  \xlongto{\cUnderlying}  \Catinf  \xlongto{\hocat}  \Duskin(\twoCat) \]
is, via the equivalences $\hocat {\cUnderlying(\cat M)} \eq \hocat {\Underlying(\cat M)} \eq \Hocat(\cat M)$, equivalent to the pseudofunctor $\Hocat \colon \ModCatL \to \twoCat$ of Hovey \cite[Theorem~1.4.3]{Hov}.
It sends a left Quillen functor $L$ to $\Lder L$, a natural isomorphism $L_1 \after L_2 \iso L$ to the natural equivalence
\[ \Lder L_1 \after \Lder L_2  =  L_1 q L_2 q  \xlongto{\eq}  L_1 L_2 q  \iso  L q  =  \Lder L \]
and has the unitor equivalence $\Lder \id = q \to \id$.
Everything works dually for $\ModCatR$ in place of $\ModCatL$.

In particular, a natural transformation $\alpha \colon L_2 \after L_1 \to L_2' \after L_1'$ of left Quillen functors induces a natural transformation
\[ \Lder L_2 \after \Lder L_1  \xlongto{\eq}  \Lder (L_2 \after L_1)  \xlongto{\alpha}  \Lder (L_2' \after L_1')  \xlongfrom{\eq}  \Lder L_2' \after \Lder L_1' \]
on homotopy categories, and we would like to be able to compute its mate (see e.g.\ \cite[§2.2]{KS} for an overview of the mate correspondence).
This is the content of the following result of Shulman \cite{Shu}.

\begin{proposition}[Shulman] \label{prop:Shulman}
  Let the following be natural transformations that are mates of each other
  \[
  \begin{tikzcd}
    \cat A \rar{L_1} \dar[swap]{L_1'} & \cat B \dar{L_2} \dlar[Rightarrow, shorten = 10][swap]{\alpha}  &[15]  \cat A \rar{L_1} \drar[Rightarrow, shorten = 10]{\beta} & \cat B  &[15]  \cat A & \lar[swap]{R_1} \cat B \\
    \cat C \rar[swap]{L_2'} & \cat D  &  \cat C \rar[swap]{L_2'} \uar{R_1'} & \cat D \uar[swap]{R_2}  &  \cat C \uar{R_1'} \urar[Rightarrow, shorten = 10]{\gamma} & \lar{R_2'} \cat D \uar[swap]{R_2}
  \end{tikzcd}
  \]
  where $\cat A$, $\cat B$, $\cat C$, and $\cat D$ are model categories, and $L_1 \dashv R_1$, $L_2 \dashv R_2$, $L_1' \dashv R_1'$, and $L_2' \dashv R_2'$ are Quillen adjunctions.
  Then the mate of the transformation $\Lder L_2 \after \Lder L_1 \to \Lder L_2' \after \Lder L_1'$ induced by $\alpha$, and the mate of the transformation $\Rder R_1' \after \Rder R_2' \to \Rder R_1 \after \Rder R_2$ induced by $\gamma$, are both represented by the zig-zag of natural transformations of functors $\Hocat(\cat C) \to \Hocat(\cat B)$ given by
  \[ L_1 q R_1' r  \xlongfrom{\eq}  L_1 q R_1' q r  \longto  L_1 R_1' q r  \xlongto{\beta}  R_2 L_2' q r  \longto  R_2 r L_2' q r  \xlongfrom{\eq}  R_2 r L_2' q \]
  where $r$ is a fibrant replacement functor and $q$ is a cofibrant replacement functor that preserves fibrant objects (which is automatic if it is obtained from a functorial factorization).
  
  In particular the derived unit and the derived counit of a Quillen adjunction $L \dashv R$ are represented by
  \[ \id  \xlongfrom{\eq} q \xlongto{\eta}  R L q  \longto  R r L q  \qquad \text{and} \qquad  L q R r  \longto  L R r  \xlongto{\epsilon}  r  \xlongfrom{\eq}  \id \]
  respectively.
\end{proposition}

\begin{proof}
  This follows from \cite[Theorem~7.6 and Proposition~6.9]{Shu}.
\end{proof}

We also note the following general fact about mates for later use.

\begin{lemma} \label{lemma:mates}
  In a $2$-category, let $l \dashv r$ and $l' \dashv r'$ be two adjunctions and let
  \[
  \begin{tikzcd}
    A \rar{l} \dar[swap]{a} & B \dar{b} & A \dar[swap]{a} \drar[Rightarrow, shorten = 10]{\beta} & \lar[swap]{r} B \dar{b} \\
    C \rar{l'} \urar[Rightarrow, shorten = 10]{\alpha} & D & C & \lar[swap]{r'} D
  \end{tikzcd}
  \]
  be two $2$-morphisms that are mates of each other.
  Then the following two diagrams commute
  \[
  \begin{tikzcd}
    a \rar{\eta'} \dar[swap]{\eta} & r' l' a \dar{\alpha} & l' a r \dar[swap]{\beta} \rar{\alpha} & b l r \dar{\epsilon} \\
    a r l \rar{\beta} & r' b l & l' r' b \rar{\epsilon'} & b
  \end{tikzcd}
  \]
  where $\eta$, $\eta'$, $\epsilon$, and $\epsilon'$ are the (co)units of the adjunctions $l \dashv r$ and $l' \dashv r'$.
\end{lemma}

\begin{proof}
  Chasing through the definitions, this follows from the triangle identities.
\end{proof}

\subsubsection*{Animas}

\begin{notation}
  We write $\sSet$ for the category of simplicial sets, equipped with the Quillen model structure.
  We write $\An \defeq \Underlying(\sSet)$ and call its objects \emph{animas}.
\end{notation}

Recall that $\An$ is equivalent to the $\infty$-category of $\infty$-groupoids.
Furthermore recall that $\An$ is the $\infty$-category freely generated under colimits by a single object.
The following is a version of the resulting universal property, which we record for later use.

\begin{observation} \label{obs:assembly}
  For a cocomplete $\infty$-category $\cat C$, the functor $\ev_{*} \colon \Fun(\An, \cat C) \to \cat C$ has a fully faithful left adjoint $\iota_!$ given by left Kan extension along $\iota \colon \set{*} \to \An$.
  In particular the map
  \[ \ev_*  \colon  \Map \bigl( \iota_! X, F \bigr)  \xlongto{\eq}  \Map \bigl( X, F(*) \bigr) \]
  is an equivalence for all $X \in \An$ and $F \colon \An \to \cat C$.
  By \cite[Lemma~5.1.5.5]{LurHA} the essential image of $\iota_!$ consists precisely of the colimit-preserving functors.
  Since the unit of the adjunction $\iota_! \dashv \ev_*$ is an equivalence, a natural transformation $F \to G$ of colimit-preserving functors $\An \to \cat C$ is thus an equivalence if and only if $F(*) \to G(*)$ is an equivalence.
\end{observation}

We now recall the (un)straightening equivalence, specialized to the case of functors to $\An$ (i.e.\ $\infty$-groupoids).
To this end note that, for an $\infty$-category $\cat C$ that admits pullbacks, the functor $\ev_1 \colon \cat C^{[1]} \to \cat C$ is a cartesian fibration and as such is classified by a functor $\opcat{\cat C} \to \Catinf$ that sends an object $C \in \cat C$ to $\overcat{\cat C}{C}$ and a map $f \colon C \to D$ to the functor $f \colon \overcat{\cat C}{D} \to \overcat{\cat C}{C}$ given by pullback along $f$ (cf.\ \kerodon{05SA}).
For an anima $X$, unstraightening yields an equivalence of $\infty$-categories
\begin{equation} \label{eq:straighten}
  \Fun(X, \An)  \xlongto{\eq}  \overcat{\An}{X}
\end{equation}
that sends $F \colon X \to \An$ to the pullback of the forgetful functor $\Anpt \to \An$ along $F$ (cf.\ \cite[§3.3.2]{LurHTT} and \cite[Theorem~3.3.17]{Lan}); it is a natural transformation of functors $\opcat{\An} \to \Catinf$ by \cite[Corollary~A.32]{GHN}.

\begin{observation} \label{obs:pi_1_action}
  By the unstraightening equivalence \eqref{eq:straighten} there is, for any map of animas $f \colon E \to B$, an essentially unique pullback square
  \[
  \begin{tikzcd}
    E \rar \dar[swap]{f} & \Anpt \dar \\
    B \rar{F} & \An
  \end{tikzcd}
  \]
  in $\Catinf$.
  Note that the fiber $F_b \defeq \fib_b f$ is canonically equivalent to $F(b)$ for every $b \in B$.
  Passing to homotopy categories, we obtain actions of $\pi_1(B, b)$ on $F_b \in \hocat \An$ and of $\pi_1(E, e)$ on $(F_{f(e)}, e) \in \hocat \Anpt$.\footnote{Note that our definition of these actions agrees with the classical construction for a Hurewicz fibration $f \colon E \to B$ of topological spaces (see e.g.\ \cite[§1.5]{MP}); this can be seen by considering \eqref{eq:straighten} for the walking isomorphism $\Interval$ (for the pointed action we pass to the undercategories $\undercat{\Fun(\Interval, \An)}{\const_*}$ and $\undercat{(\Anover{\Interval})}{\id_I}$).}
  In particular $\pi_1(E, e)$ acts on $\pi_n(F_{f(e)}, e)$ for every $n$.
\end{observation}

\subsubsection*{Spectra}

We recall the following notation from \cite[§1.4]{LurHA}.

\begin{notation}
  For an $\infty$-category $\cat C$ that admits finite limits, we denote by $\Spob(\cat C)$ the stable $\infty$-category of spectrum objects in $\cat C$ and by $\Loopsinf \colon \Spob(\cat C) \to \cat C$ the infinite loop functor.
  When $\cat C$ is presentable, we write $\Suspinf \colon \cat C \to \Spob(\cat C)$ for the left adjoint of $\Loopsinf$.
  We write $\Sp \defeq \Spob(\An)$ for the $\infty$-category of spectra and furthermore set $\SS[\blank] \defeq \Suspinf \colon \An \to \Sp$ and $\SS \defeq \SS[*]$.
\end{notation}

We will need the following version of the universal property of $\Spob(\blank)$.
In the case where $\cat D$ is presentable, it is a consequence of \cite[Corollary~1.4.4.5]{LurHA}.

\begin{lemma} \label{lemma:Spobj}
  Let $\cat C$ be a presentable $\infty$-category and $\cat D$ a stable $\infty$-category.
  Then the following functor is fully faithful
  \[ (\Suspinf)^*  \colon  \FunL \bigl( \Spob(\cat C), \cat D \bigr)  \longto  \FunL ( \cat C, \cat D ) \]
  where $\FunL(\cat A, \cat B) \subseteq \Fun(\cat A, \cat B)$ denotes the full subcategory of left adjoint functors.
\end{lemma}

\begin{proof}
  By taking adjoints, the assertion is equivalent to the upper horizontal functor in the following diagram being fully faithful
  \[
  \begin{tikzcd}
    \FunR \bigl( \cat D, \Spob(\cat C) \bigr) \rar{(\Loopsinf)_*} \dar &[10] \FunR ( \cat D, \cat C ) \dar \\
    \Excred \bigl( \cat D, \Spob(\cat C) \bigr) \rar{(\Loopsinf)_*} & \Excred ( \cat D, \cat C )
  \end{tikzcd}
  \]
  where $\FunR \subseteq \Fun$ denotes the full subcategory of right adjoint functors and $\Excred \subseteq \Fun$ denotes the full subcategory of reduced excisive functors, i.e.\ those that send pushouts to pullbacks and preserve the terminal object.
  Note that the vertical fully faithful inclusions exist by the assumption that $\cat D$ is stable.
  The bottom horizontal functor is an equivalence by \cite[Proposition~1.4.2.22]{LurHA}.
  This implies the claim.
\end{proof}

\begin{lemma} \label{cor:Sp}
  Let $\cat D$ be a stable $\infty$-category.
  Then the following functor is fully faithful
  \[ \ev_\SS  \colon  \FunL ( \Sp, \cat D )  \longto  \cat D \]
  where $\FunL(\cat A, \cat B) \subseteq \Fun(\cat A, \cat B)$ denotes the full subcategory of left adjoint functors.
\end{lemma}

\begin{proof}
  The functor $\ev_* \colon \FunL ( \An, \cat D ) \to \cat D$ is fully faithful by \cite[Lemma~5.1.5.5~(1) and Proposition~4.3.2.15]{LurHTT}.
  Then the claim follows from \cref{lemma:Spobj} applied to $\cat C = \An$.
\end{proof}

\subsection{Differential graded algebra}

In this subsection, we fix our conventions for homological algebra, and recall elementary properties of the homotopy theory of differential graded algebras and their modules.

\begin{convention}
  The base field is $\QQ$.
  We use cohomological grading conventions, i.e.\ differentials have degree $1$ and we write degrees as superscripts.
\end{convention}

\begin{notation}
  We write $\Ch$ for the category of unbounded cochain complexes, equipped with its usual symmetric monoidal structure.
  We denote by $\shift[n]$ a cohomological shift by $-n$, i.e.\ for a cochain complex $C$, we set $(\shift[n] C)^p \defeq C^{p+ n}$.
  We write $\shift C \defeq \shift[1] C$.
\end{notation}

Since we will mainly require homological shifts by $1$ (i.e.\ cohomological shifts by $-1$), we chose this convention for $\shift[n]$ to keep our notation light.

\subsubsection*{Algebras}

\begin{definition}
  A \emph{dga} is a monoid in $\Ch$.
  For a dga $R$, we denote by $\opmod{R}$ the same underlying cochain complex with the opposite multiplication, i.e.\ we define $\opmod{a} \cdot \opmod{b} \defeq (-1)^{\deg a \deg b} \opmod{(b \cdot a)}$.
\end{definition}

\begin{definition}
  A \emph{cdga} is a commutative monoid in $\Ch$.
  We denote by $\CDGA$ the category of non-negatively graded cdgas, equipped with the projective model structure, i.e.\ the weak equivalences are the quasi-isomorphisms and the fibrations are the degreewise surjections.\footnote{This projective model structure exists by \cite[Theorem~4.3]{BG}.}
  A \emph{cofibration of cdgas} will refer to a map of non-negatively graded cdgas that is a cofibration in $\CDGA$; a \emph{cofibrant cdga} will refer to a non-negatively graded cdga that is cofibrant in $\CDGA$.
  We will sometimes implicitly consider a $\QQ$-algebra to be a dga concentrated in degree $0$.
\end{definition}

We remark that the model category $\CDGA$ is combinatorial: it is cofibrantly generated by \cite[§1.2]{Hes} and, using finitely generated cdgas, it is easy to see that it is locally presentable as well (see e.g.\ \cite[Lemma~5.4]{Sci}).

\begin{definition}
  A non-negatively graded cdga $R$ is \emph{homologically connected} if $\Coho 0 (R) \iso \QQ$; it is \emph{connected} if $R^0 \iso \QQ$.
\end{definition}

\begin{definition}
  A map of non-negatively graded cdgas $R \to S$ is \emph{quasi-free} if the underlying graded algebra of $S$ is free over $R$, i.e.\ if there is an isomorphism $S \iso R \tensor \SA V$ of graded algebras under $R$ for some graded vector space $V$.
  A quasi-free map of cdgas is \emph{semifree} if the isomorphism can be chosen such that $V$ admits a basis $(v_\alpha)_{\alpha \in I}$ for some well-ordered set $I$ such that $d(v_\alpha) \in R \tensor \SA V_{< \alpha}$, where $V_{< \alpha} \subset V$ is the subspace spanned by those $v_\beta$ with $\beta < \alpha$.
  A semifree map of cdgas is \emph{minimal} if there exists such a basis such that $\deg{v_\alpha} \le \deg{v_\beta}$ whenever $\alpha \le \beta$.
  A cdga $R$ is \emph{quasi-free}, \emph{semifree}, or \emph{minimal} when the map $\QQ \to R$ is so.
  A \emph{minimal model} of a cdga $R$ is a quasi-isomorphism of cdgas $M \to R$ such that $M$ is minimal.
\end{definition}

Note that a semifree map of cdgas is a cofibration (see e.g.\ \cite[Remark~3.4]{Bra}).
Recall that a homologically connected cdga always admits a minimal model (see e.g.\ \cite[Proposition~7.7]{BG}), which is automatically connected.

\begin{observation} \label{obs:cofib_cdgas_retract}
  Given a cofibration $\phi \colon \k \to R$ of cdgas, there exists a semifree map $\psi \colon \k \to S$ of cdgas such that $\phi$ is a retract of $\psi$ in the under category $\undercat{(\CDGA)}{\k}$.
  This follows from the facts that the semifree maps of cdgas are precisely the relative cell complexes (see e.g.\ \cite[Remark~3.3]{Bra}), and that $\undercat{(\CDGA)}{\k}$ is again cofibrantly generated (see e.g.\ \cite[Theorem~15.3.6]{MP}).
\end{observation}

\subsubsection*{Modules}

\begin{definition}
  Let $R$ be a dga.
  An \emph{$R$-module} is a left module over $R$ in $\Ch$.
  We denote by $\Mod{R}$ the category of $R$-modules, equipped with the projective model structure, i.e.\ the weak equivalences are the quasi-isomorphisms and the fibrations are the degreewise surjections.\footnote{This projective model structure exists by \cite[Theorem~3.3]{BMR}.}
  When $R$ is commutative, we equip $\Mod{R}$ with the closed symmetric monoidal structure given by $\tensor_R$.
\end{definition}

Note that for us a $\QQ$-module is by definition a cochain complex.
By \cite[Theorem~3.3 and Remark~6.19]{BMR}, the model category $\Mod{R}$ is combinatorial, proper, enriched over $\Mod{\QQ}$, and, if $R$ is commutative, also monoidal.
In the latter case it is in particular enriched over itself, and thus the internal hom-functor $\Hom_R(\blank, N) \colon \opcat{\Mod{R}} \to \Mod{R}$ is right Quillen, and so is $\Hom_R(M, \blank) \colon \Mod{R} \to \Mod{R}$ when $M$ is cofibrant (see e.g.\ \cite[§16.4]{MP}).

\begin{notation}
  For a cdga $R$ and an $R$-module $M$, we write $\dual M_R$ for $\Hom_R(M, R)$ as an $R$-module.
  Note that, if $R \to S$ is a map of cdgas, then $\dual S_R$ canonically lifts to an $S$-module.
\end{notation}

\begin{definition}
  Let $R$ be a dga.
  An $R$-module $M$ is \emph{quasi-free} if its underlying graded module over the underlying graded algebra of $R$ is free, i.e.\ if there is an isomorphism of graded $R$-modules $M \iso R \tensor V$ for some graded vector space $V$.
  A quasi-free $R$-module is \emph{semifree} if the isomorphism can be chosen such that $V$ admits a basis $(v_\alpha)_{\alpha \in I}$ for some well-ordered set $I$ such that $d(v_\alpha) \in R \tensor V_{< \alpha}$, where $V_{< \alpha} \subset V$ is the subspace spanned by those $v_\beta$ with $\beta < \alpha$.
  A semifree $R$-module is \emph{minimal} if there exists such a basis such that $\deg{v_\alpha} \le \deg{v_\beta}$ whenever $\alpha \le \beta$.
  A \emph{minimal model} of an $R$-module $N$ is a quasi-isomorphism of $R$-modules $M \to N$ such that $M$ is minimal.
\end{definition}

Assuming that $R$ is a non-negatively graded cdga, recall from \cite[§3.2]{Bra} that semifree $R$-modules are cofibrant, that an $R$-module whose homology is bounded below admits a minimal model, and that a minimal model is unique up to isomorphism if it exists.
Also note that, if $\phi \colon R \to S$ is a semifree map of cdgas, then $S$ is semifree as an $R$-module; hence, if $\phi$ is a cofibration of cdgas, then $S$ is cofibrant as an $R$-module (using \cref{obs:cofib_cdgas_retract}).
We furthermore record the following lemma for later use.

\begin{lemma} \label{lemma:cofibrant_flat}
  Let $R$ be a dga, $f \colon M \to M'$ a quasi-isomorphism of $R$-modules, and $g \colon N \to N'$ a quasi-isomorphism of $\opmod{R}$-modules.
  If $M$ and $M'$ are cofibrant, then $f \tensor g \colon M \tensor_R N \to M' \tensor_R N'$ is a quasi-isomorphism.
  In particular, when $R$ is a cdga and $M$ is a cofibrant $R$-module, the functor $M \tensor_R \blank \colon \Mod{R} \to \Mod{R}$ preserves quasi-isomorphisms.
\end{lemma}

\begin{proof}
  We will prove that $M \tensor_R \blank$ preserves quasi-isomorphism; the general statement follows by choosing a cofibrant replacement $\ol N \to N$.
  Since the model category $\Mod{R}$ is compactly generated by \cite[Theorem~3.3]{BMR}, the cofibrant object $M$ is a retract of a sequential cell complex of $R$-modules; since retracts of quasi-isomorphisms are quasi-isomorphisms we can assume without loss of generality that $M$ is a sequential cell complex.
  It thus has an exhaustive, bounded below filtration such that each filtration quotient is isomorphic to $R^{\oplus I}$ for some set $I$.
  Since $R^{\oplus I} \tensor_R \blank$ clearly preserves quasi-isomorphisms, this implies the claim.
\end{proof}

\subsubsection*{Hochschild homology}

We now recall the bar construction, and the complexes defining Hochschild homology and cyclic homology of an algebra.
For more background, see for example \cite[Chapters~1 and 2]{Lod}.

\begin{definition} \label{def:bar}
  Let $\k$ be a cdga and $\k \to R$ a map of dgas.
  The \emph{two-sided bar construction} $\B_\k(R, R, R)$ of $R$ as a $\k$-algebra is the $(R \tensor_\k \opmod{R})$-module obtained as the total complex of
  \[ \cdots  \xlongto{d_3}  R \tensor_\k R^{\tensor_\k 2} \tensor_\k R  \xlongto{d_2}  R \tensor_\k R \tensor_\k R  \xlongto{d_1}  R \tensor_\k R \]
  where $d_n \defeq \sum_{i = 0}^{n} (-1)^i {\id^{\tensor i}} \tensor \mu_R \tensor \id^{\tensor n - i}$ with $\mu_R$ the multiplication of $R$.
  The augmentation $d_0 = \mu_R \colon R \tensor_\k R \to R$ yields a canonical map $\epsilon \colon \B_\k(R, R, R) \to R$ of $(R \tensor_\k \opmod{R})$-modules.
\end{definition}

\begin{definition} \label{def:HH}
  Let $\k$ be a cdga and $\k \to R$ a map of dgas.
  For an $(R \tensor_\k \opmod{R})$-module $M$, we write
  \[ \HH_\k(R, M)  \defeq  \B_\k(R, R, R) \tensor_{R \tensor_\k \opmod{R}} M \]
  for the \emph{Hochschild complex} of $R$ with coefficients in $M$.
  Note that there is a canonical isomorphism $\HH_\k(R, M) \iso \bigoplus_{n \ge 0} M \tensor_\k (\shift R)^{\tensor_\k n}$ of graded $\k$-modules.
  We write $\HH_\k(R) \defeq \HH_\k(R, R)$ and define the \emph{Connes complex} of $R$ to be the quotient
  \[ \HC_\k(R)  \defeq  \shift[-1] \bigoplus_{n \ge 0} (\shift R)^{\tensor_\k n+1}_{\Cyclic {n+1}}  \longtwoheadleftarrow  \shift[-1] \bigoplus_{n \ge 0} (\shift R)^{\tensor_\k n+1}  \iso  \shift[-1] \HH_\k(R, \shift R)  \iso  \HH_\k(R) \]
  where the cyclic group $\Cyclic{n+1}$ acts by cyclic permutations.
\end{definition}

The defining property of the bar complex is that it is a cofibrant resolution of $R$ as an $(R \tensor_\k \opmod{R})$-module, as the following lemma shows.

\begin{lemma} \label{lemma:bar_complex}
  Let $\k$ be a cdga and $\k \to R$ a map of dgas.
  Then $\epsilon \colon \B_\k(R, R, R) \to R$ is a quasi-isomorphism.
  If $R$ is cofibrant as a $\k$-module, then $\B_\k(R, R, R)$ is cofibrant as an $(R \tensor_\k \opmod{R})$-module.
\end{lemma}

\begin{proof}
  The maps
  \[ s_n  \colon  R^{\tensor_\k n + 2}  \longto  R^{\tensor_\k n + 3},  \quad  x_0 \tensor \dots \tensor x_{n+1}  \longmapsto  1 \tensor x_0 \tensor \dots \tensor x_{n+1} \]
  fulfill $d_{n+1} s_n + s_{n-1} d_n = \id$ for all $n \ge 0$.
  Since the $s_n$ are furthermore cochain maps, this implies that they exhibit $s_{-1}$ as a homotopy inverse of $\epsilon$.
  
  By construction, the $(R \tensor_\k \opmod{R})$-module $\B_\k(R, R, R)$ is isomorphic to a sequential colimit of maps obtained as pushouts along maps of the form $M_n \to \Cone(M_n)$ with $M_n$ isomorphic to $(R \tensor_\k \opmod{R}) \tensor_\k (\shift R)^{\tensor_\k n}$.
  The $\k$-module $(\shift R)^{\tensor_\k n}$ is cofibrant by assumption, and hence $M_n$ is a cofibrant $(R \tensor_\k \opmod{R})$-module.
  The map $\iota_n \colon M_n \to \Cone(M_n)$ is obtained by taking the tensor product of the cofibration $\QQ \to \Cone(\QQ)$ of $\QQ$-modules with $M_n$; since $\Mod{R \tensor_\k \opmod{R}}$ is a $\Mod{\QQ}$-enriched model category, this implies that $\iota_n$ is a cofibration.
  This completes the proof.
\end{proof}

\subsection{Homology localization}

In this subsection, we recall the definition and basic properties of homology localizations of animas, with a particular focus on rationalization.
Recall the following definition from \cite[§5.5.4]{LurHTT}.

\begin{definition}
  Let $\cat C$ be an $\infty$-category and $S$ a class of morphisms of $\cat C$.
  We say that an object $L$ of $\cat C$ is \emph{$S$-local} if the map of animas
  \[ f^*  \colon  \Map_{\cat C}(Y, L)  \longto  \Map_{\cat C}(X, L) \]
  is an equivalence for every map $f \colon X \to Y$ in $S$.
\end{definition}

\begin{definition}
  Let $A$ be an abelian group.
  We say that a map of animas is an \emph{$A$-equivalence} if it induces an isomorphism on homology with coefficients in $A$.
  An anima is \emph{$A$-local} if it is local with respect to the class of $A$-equivalences.
  We denote by $\An^A \subseteq \An$ the full subcategory spanned by the $A$-local animas.
\end{definition}

\begin{lemma} \label{lemma:localization}
  Let $A$ be an abelian group.
  Then the inclusion $\An^A \subseteq \An$ has a left adjoint $(\blank)_A$; for any anima $X$, the unit $\eta \colon X \to X_A$ is an $A$-equivalence to an $A$-local anima, and it is uniquely characterized by this property up to contractible choice.
  Furthermore, a map $f \colon X \to Y$ of animas is an $A$-equivalence if and only if $f_A \colon X_A \to Y_A$ is an equivalence.
\end{lemma}

\begin{proof}
  The class of $A$-equivalences is of small generation by \cite[Proposition~5.5.4.16]{LurHTT} applied to the functor $\mathrm{H}A \tensor \blank \colon \An \to \Sp$.
  Then the claim follows from \cite[(Proof of) Proposition~5.5.4.15]{LurHTT}.
\end{proof}

Classically, being $A$-local is often defined in terms of the homotopy category (see e.g.\ \cite[§5.2]{MP}).
The following lemma shows that this is equivalent to the definition we use.

\begin{lemma}
  Let $R$ be a ring.
  Then an anima is $R$-local if and only if it is $R$-local in the homotopy category of $\An$.
\end{lemma}

\begin{proof}
  The ``only if''-direction is clear.
  For the ``if''-direction, let $L$ be an $R$-local object of $\hocat \An$ and $f \colon X \to Y$ an $R$-equivalence.
  Note that, for every anima $Z$, the map $Z \times f$ is again an $R$-equivalence by the Künneth spectral sequence.
  Thus the maps
  \begin{gather*}
    f^*  \colon  \pi_0 \Map_\An(Z \times Y, L)  \longto  \pi_0 \Map_\An(Z \times X, L)
  \shortintertext{and hence}
    f^*  \colon  [Z, \Map_\An(Y, L)]  \longto  [Z, \Map_\An(X, L)]
  \end{gather*}
  are bijections.
  Thus $f^* \colon \Map_\An(Y, L) \to \Map_\An(X, L)$ is an equivalence by the Yoneda lemma in $\hocat \An$.
\end{proof}

For the rest of this subsection we specialize to the case of $\QQ$-localization (though everything that follows also works for an arbitrary subring of $\QQ$).

\begin{definition}
  We say that a $\QQ$-equivalence is a \emph{rational equivalence} and that a $\QQ$-local anima is \emph{rational}.
  We call the functor $(\blank)_\QQ$ of \cref{lemma:localization} the \emph{rationalization functor}.
  We say that a map of animas $X \to Y$ is a \emph{rationalization} if it is a rational equivalence and $Y$ is rational.
\end{definition}

Note that the unit map $X \to \rat X$ is a rationalization, and that any rationalization is of this form up to equivalence.

\subsubsection*{Nilpotent animas and rationalization}

We will now recall that the rationalization functor is particularly well-behaved for nilpotent animas.
We begin with the notion of a nilpotent action on a group; the definition we use is taken from Bousfield--Kan \cite[Ch.~II, 4.1]{BK}.

\begin{definition}
  Let $K$ be a group that acts on a group $G$, i.e.\ it comes equipped with a group homomorphism $K \to \Aut_\Grp(G)$.
  We say that this action is \emph{nilpotent} if $G$ has a finite normal series of subgroups (i.e.\ each $A_i$ is normal in $A_{i+1}$)
  \[ 1 = A_0 \subseteq A_1 \subseteq \dots \subseteq A_n = G \]
  such that each $A_i$ is fixed setwise by $K$ and each quotient $\quot {A_{i+1}} {A_i}$ is abelian with trivial $K$-action.
  A group $G$ is \emph{nilpotent} if the conjugation action of $G$ on itself is nilpotent.
\end{definition}

Recall that the class of nilpotent $K$-actions is closed under taking subgroups, quotients, and extensions, see \cite[Ch.~II, Lemma~4.2]{BK}.
Also note that a group $G$ is nilpotent if and only if it admits a central series of finite length.
The following lemma shows that the definition of a nilpotent action we use is equivalent to the stronger definition of Hilton \cite{Hil75} (which is also used by May--Ponto \cite[§3.1]{MP}) when $G$ is nilpotent.\footnote{This is stated in \cite[§0]{HRS} without proof.}

\begin{lemma} \label{lemma:nilpotent_defs}
  Let $G$ be a nilpotent group.
  Then an action of a group $K$ on $G$ is nilpotent if and only if $G$ admits a central series of subgroups (i.e.\ each commutator $[G, A_{i + 1}]$ is contained in $A_i$)
  \[ 1 = A_0 \subseteq A_1 \subseteq \dots \subseteq A_n = G \]
  such that each $A_i$ is fixed setwise by $K$ and each quotient $\quot {A_{i+1}} {A_i}$ has trivial $K$-action.
\end{lemma}

\begin{proof}
  For the ``if''-direction note that $[G, A_{i + 1}] \subseteq A_i$ implies that $A_i \subseteq A_{i+1}$ is normal and that $\quot {A_{i+1}} {A_i}$ is abelian.
  For the ``only if''-direction, let
  \[ G = \Gamma^0 G \supseteq \Gamma^1 G \supseteq \dots \supseteq \Gamma^n G = 1 \]
  be the lower central series of $G$, i.e.\ we set $\Gamma^{i+1} G \defeq [\Gamma^i G, G]$.
  The group $K$ acts on the abelian group $\quot {\Gamma^i G} {\Gamma^{i+1} G}$, and this action is nilpotent since the class of nilpotent $K$-actions is closed under taking quotients and subgroups.
  Thus there is a normal series of subgroups that are fixed setwise by $K$
  \[ \quot {\Gamma^i G} {\Gamma^{i+1} G} = F^i_0 \supseteq F^i_1 \supseteq \dots \supseteq F^i_{m_i} = 1 \]
  such that the $K$-action on the filtration quotients is trivial.
  Then, denoting the quotient map $\Gamma^i G \to \quot {\Gamma^i G} {\Gamma^{i+1} G}$ by $p_i$, the filtration
  \[ G = \inv{p_0}(F^0_0) \supseteq \dots \supseteq \inv{p_0}(F^0_{m_0}) = \Gamma^1 G = \inv{p_1}(F^1_0) \supseteq \dots \supseteq \inv{p_{n-1}}(F^{n-1}_{m_{n-1}}) = 1 \]
  is a central series of subgroups that are fixed setwise by $K$ such that the $K$-action on the filtration quotients is trivial.
\end{proof}

We now recall the notion of a nilpotent anima and a nilpotent map, as well as basic properties we will use throughout.

\begin{definition}
  We say that a map of animas $f \colon X \to Y$ is \emph{nilpotent} if $Y$ and $F \defeq \fib f$ are connected and for some $x \in F$ the action of $\pi_1(X, x)$ on $\pi_n(F, x)$ (cf.\ \cref{obs:pi_1_action}) is nilpotent for all $n \ge 1$.
  We say that a connected anima $X$ is \emph{nilpotent} if the map $X \to *$ is nilpotent.
\end{definition}

Recall that a map of connected animas $f \colon X \to Y$ is nilpotent if and only if its fiber $F$ is nilpotent and the action of $\pi_1(Y)$ on $\Ho n(F; \ZZ)$ is nilpotent for all $n \ge 0$ (see \cite[Corollary~2.2]{Hil}).
Further recall that a nilpotent anima $X$ is rational if and only if $\Ho n (X; \ZZ)$ is uniquely divisible (i.e.\ a rational vector space) for all $n \ge 1$ (see e.g.\ \cite[Theorem~6.1.1]{MP}).
On nilpotent animas, a model for the rationalization is given by the Bousfield--Kan $\QQ$-completion, see \cite[Ch.~V, Proposition~4.2]{BK}.
The $\QQ$-completion of a nilpotent map of animas is again nilpotent by \cite[Ch.~II, Lemma~4.8]{BK}; in particular the rationalization of a nilpotent anima is again nilpotent.
Lastly recall that for a nilpotent anima $X$, there are canonical isomorphisms $\pi_n(\rat X) \iso \pi_n(X) \tensor \QQ$ and $\Ho n (\rat X; \ZZ) \iso \Ho n (X; \QQ)$ for $n \ge 1$ (see e.g.\ \cite[Ch.~V, Proposition~3.1]{BK}); for $n = 1$, the expression $\pi_1(X) \tensor \QQ$ denotes the Malcev completion of the nilpotent group $\pi_1(X)$.

We conclude this subsection by recording the following two lemmas we will need later.

\begin{lemma} \label{lemma:nilpotent_composite}
  Let $f \colon X \to Y$ and $g \colon Y \to Z$ be two maps of connected animas with connected fibers.
  If any two of $f$, $g$, and $g \after f$ are nilpotent, then so is the third.
  In particular a map $f \colon X \to Y$ of nilpotent animas is nilpotent if and only if its fiber is connected.
\end{lemma}

\begin{proof}
  This is \cite[Ch.~II, Proposition~4.4]{BK}.
\end{proof}

\begin{lemma} \label{lemma:rat_pullback}
  Let $f \colon X \to B$ and $g \colon Y \to B$ be pointed maps of nilpotent animas.
  Then the basepoint component of $X \times_B Y$ is nilpotent, and the canonical map $X \times_B Y \to \rat X \times_{\rat B} \rat Y$ is a rationalization when restricted to the respective basepoint components.
  When the fiber of $f$ or $g$ is connected, then $X \times_B Y$ and $\rat X \times_{\rat B} \rat Y$ are connected and hence the rationalization functor preserves the pullback $X \times_B Y$.
\end{lemma}

\begin{proof}
  This is \cite[Proposition~6.2.5]{MP}.
\end{proof}

\subsection{Sullivan's approach to rational homotopy theory}

In this subsection, we summarize the approach of Sullivan \cite{Sul} to rational homotopy theory, following Bousfield--Gugenheim \cite{BG} (also see Braunack--Mayer \cite[§3]{Bra}, whose notation we will largely follow).
First recall the Sullivan--de Rham adjunction of \cite[§8.1]{BG}
\begin{equation} \label{eq:Sullivan}
\begin{tikzcd}
  \sSet \rar[yshift = 7]{\APL}[below, name = T]{} &[20] \lar[yshift = -7]{\Real}[above, name = B]{} \opcat{\CDGA}
  \ar[from = T, to = B, phantom, "\vertdashv"]
\end{tikzcd}
\end{equation}
where $\APL(X)$ denotes piecewise linear de Rham forms on a simplicial set $X$, and $\Real(A)$ denotes the spatial realization of a cdga $A$.
Recall that there is a natural quasi-isomorphism $\APL(X) \to \Cochains * (X; \QQ)$ that induces an isomorphism $\Coho * (\APL(X)) \iso \Coho * (X; \QQ)$ of graded algebras (see e.g.\ \cite[Theorem~2.2 and Corollary~3.4]{BG}).
We furthermore recall the following terminology.

\begin{definition}
  An anima $X$ is \emph{of finite rational type} if $\Ho n (X; \QQ)$ is finite-dimensional for all $n \ge 0$.
  We write $\An^\QQ_{\nil,\ft} \subset \An$ for the full subcategory spanned by the rational nilpotent animas of finite rational type.
\end{definition}


\begin{definition}
  A cdga $A$ is of \emph{finite homotopical type} if it admits a minimal model that is finite-dimensional in each degree.
  We write $\CDGA[\ge 1, \fht] \subset \CDGA$ for the full subcategory spanned by the homologically connected cdgas of finite homotopical type.
\end{definition}

\begin{proposition}[Sullivan, Bousfield--Gugenheim] \label{prop:Sullivan}
  The Sullivan--de Rham adjunction \eqref{eq:Sullivan} is a Quillen adjunction and its derived functors restrict to an adjoint equivalence
  \[ \opcat{\Underlying(\CDGA)_{\ge 1, \fht}}  \eq  \An^\QQ_{\nil,\ft} \]
  of $\infty$-categories.
\end{proposition}

\begin{proof}
  In the following we equip $\opcat{\CDGA}$ with the opposite model structure.
  By \cite[Lemma~2.7]{BG}, the functor $\APL$ sends cofibrations of simplicial sets (i.e.\ degreewise injections) to degreewise surjections of cdgas (which are the cofibrations in the opposite model structure).
  By \cite[Lemma~8.2]{BG}, the functor $\Real$ sends cofibrations of cdgas (which are the fibrations in the opposite model structure) to Kan fibrations of simplicial sets.
  Hence \eqref{eq:Sullivan} is a Quillen adjunction.
  
  We thus obtain an adjunction of $\infty$-categories
  \[
  \begin{tikzcd}
    \An = \Underlying(\sSet) \rar[yshift = 7]{\Lder \APL}[below, name = T]{} &[20] \lar[yshift = -7]{\Rder \Real}[above, name = B]{} \Underlying(\opcat{\CDGA}) = \opcat{\Underlying(\CDGA)}
    \ar[from = T, to = B, phantom, "\vertdashv"]
  \end{tikzcd}
  \]
  (note that every object of $\sSet$ is cofibrant, so that $\Rder \Real = \Underlying(\Real)$).
  By \cite[Theorem~10.1]{BG}, these derived functors restrict to functors (which are thus again adjoints) as follows
  \[
  \begin{tikzcd}
    \An^\QQ_{\nil,\ft} \rar[yshift = 4]{\Lder \APL} &[20] \lar[yshift = -4]{\Rder \Real} \opcat{\Underlying(\CDGA)_{\ge 1, \fht}}
  \end{tikzcd}
  \]
  and, when restricted to these full subcategories, the natural transformations
  \[ \id  \xlongto{\eta}  \Real \after \APL  \longto  \Real \after \Lder \APL  \qquad \text{and} \qquad  \Lder \APL \after \Real  \xlongfrom{\eq}  \APL \after \Real  \xlongto{\epsilon}  \id \]
  are pointwise weak equivalences.
  Thus they induce homotopies $\id \eq \Rder \Real \after \Lder \APL$ and $\Lder \APL \after \Rder \Real \eq \id$, which completes the proof.
\end{proof}

Note in particular that, by \cite[Proposition~1.3.4.23]{LurHA}, homotopy colimits in the model category $\CDGA$, as long as they are again homologically connected and of finite homotopical type, represent limits in $\An^\QQ_{\nil,\ft}$.

\begin{definition} \label{def:model}
  Let $\cat I$ be a category.
  We say that a diagram $A \colon \cat I \to \opcat{\CDGA[\ge 1, \fht]}$ \emph{models} a diagram $X \colon \cat I \to \An_{\nil,\ft}$ if it comes equipped with an equivalence $\Rder \Real \after u \after A \eq \rat X$ of functors $\cat I \to \An^\QQ_{\nil,\ft}$, where $u \colon \opcat{\CDGA[\ge 1, \fht]} \to \opcat{\Underlying(\CDGA)_{\ge 1, \fht}}$ is the canonical functor.
\end{definition}

\begin{remark}
  By \cite[Theorem~11.2]{BG}, the unit $\id \to \Rder \Real \after \Lder \APL$ is a rationalization when restricted to $\An_{\nil,\ft}$.
  In particular there is a natural equivalence $\Lder \APL (\rat Z) \eq \Lder \APL(Z)$ for $Z \in \An_{\nil,\ft}$.
  Hence, in \cref{def:model}, we could have equivalently asked for an equivalence $u \after A \eq \Lder \APL \after X$.
\end{remark}

Note that, if an anima $X$ is modeled by a cdga $A$, there is an induced isomorphism $\Coho * (X; \QQ) \iso \Coho * (A)$ of graded algebras.
We conclude this subsection by recording the following two lemmas we will need later.

\begin{lemma} \label{lemma:augmented}
  A homologically connected cofibrant cdga $R$ of finite homotopical type admits an augmentation $R \to \QQ$.
\end{lemma}

\begin{proof}
  Since $R$ is cofibrant, it follows from \cref{prop:Sullivan} that homotopy classes of maps $* \to \Rder \Real (R) \neq \emptyset$ are in bijection with homotopy classes of maps $R \to \QQ$.
\end{proof}

\begin{lemma} \label{lemma:pullback_model}
  Let the following be a pullback diagram in $\An$ and a homotopy pushout diagram in $\CDGA$, respectively,
  \[
  \begin{tikzcd}
    X \times_Z Y \rar \dar & Y \dar{g}  &  P & \lar B \\
    X \rar{f} & Z  &  A \uar & \lar C \uar
  \end{tikzcd}
  \]
  such that the maps $X \to Z \from Y$ lie in $\An_{\nil,\ft}$ and are modeled by the maps $A \to C \from B$, which lie in $\CDGA[\ge 1, \fht]$.
  If $f$ or $g$ has connected fibers, then the left-hand square lies in $\An_{\nil,\ft}$ and it is canonically modeled by the right-hand square, which lies in $\CDGA[\ge 1, \fht]$.
\end{lemma}

\begin{proof}
  By \cref{lemma:rat_pullback} the pullback $X \times_Z Y$ is nilpotent and the rationalization of the left-hand square is a pullback square in $\An^\QQ_\nil$.
  By \cref{lemma:nilpotent_composite} one of the maps $f$ and $g$ is nilpotent and hence the Eilenberg--Moore spectral sequence (with coefficients in $\QQ$) for the left-hand pullback square converges to $\Ho * (X \times_Z Y; \QQ)$ by \cite[Theorem~3.1]{Shi}; this implies that this homology is of finite type.
  Hence the rationalization of the left-hand square is also a pullback square in $\An^\QQ_{\nil,\ft}$ and it is enough to show that $P$ lies in $\CDGA[\ge 1, \fht]$.
  
  To this end, first note that up to quasi-isomorphism we can replace $C$ by a minimal (and hence connected and cofibrant) cdga.
  Without loss of generality we can assume that $f$ is surjective on $\pi_1$; then the map $C \to A$ is injective on $\Coho 1$ and we can replace it by a cofibration whose target is connected by \cite[Remark~7.9]{BG}.
  Similarly we can replace $C \to B$ by an arbitrary cofibration, so that its target is cofibrant and we can take $P$ to be the actual pushout of the square.
  \Cref{lemma:augmented} yields an augmentation of $B$, which induces an augmentation of $P$ such that the right-hand pushout square consists of maps of augmented cdgas.
  Now recall from \cite[§9.2]{BG} that a non-negatively graded augmented cdga $\epsilon \colon R \to \QQ$ is of finite homotopical type if and only if its homotopy groups $\pi^n(R) \defeq \Coho n (\quot {\ol R} {(\ol R \cdot \ol R)})$ are all finite-dimensional (here $\ol R \defeq \ker \epsilon$ is the augmentation ideal).
  By the long exact Mayer--Vietoris sequence of \cite[Proposition~6.14]{BG}, which relates the homotopy groups of $A$, $B$, $C$, and $P$, we deduce that $P$ is of finite homotopical type.
  To see that $P$ is homologically connected, first note that the map $B^0 \to P^0$ is an isomorphism since $A$ and $C$ are both connected.
  Furthermore $P^1$ is the pushout of $A^0 \tensor C^1 \tensor B^0 \to A^0 \tensor B^1$ along the injection $A^0 \tensor C^1 \tensor B^0 \to A^1 \tensor B^0$; hence $B^1 \iso A^0 \tensor B^1 \to P^1$ is injective.
  This implies that
  \[ \Coho 0 (B) = \ker (d \colon B^0 \to B^1) \iso \ker (d \colon P^0 \to P^1) = \Coho 0 (P) \]
  which completes the proof.
\end{proof}

\subsection{Parametrized spectra} \label{sec:param_spectra}

In this subsection, we recall the notion of a parametrized spectrum and parts of the associated six-functor formalism, which is a case of a so-called Wirthmüller context.

\begin{definition}
  Let $X$ be an anima.
  A \emph{parametrized spectrum over $X$}, or \emph{$X$-spectrum} for short, is a functor $X \to \Sp$.
  We write $\Sp[X] \defeq \Fun(X, \Sp)$ for the $\infty$-category of $X$-spectra, equipped with the pointwise symmetric monoidal structure $\tensor_X$ (we sometimes drop the subscript if there is no risk of confusion); we write $\SS_X \defeq \const_\SS$ for the monoidal unit.
  We furthermore write $\SS_X[\blank]$ and $\Suspinf[X]$ for the functor $(\Suspinf)_* \colon \Anover{X} \eq \Fun(X, \An) \to \Sp[X]$, and $\Loopsinf[X]$ for the functor $(\Loopsinf)_* \colon \Sp[X] \to \Fun(X, \An) \eq \Anover{X}$.
  For a map $f \colon X \to Y$ of animas, we denote by $f^* \colon \Sp[Y] \to \Sp[X]$ the precomposition functor, and by $f_! \dashv f^* \dashv f_*$ its adjoints.
  Unless stated otherwise, we will consider $\Fun(\blank, \Sp)$ to be contravariantly functorial using the maps $f^*$ and $\Sp[\blank]$ to be covariantly functorial using the maps $f_!$.
  By an abuse of notation, given an anima $E$ over $X$, we sometimes denote its structure map $E \to X$ by $E$.
\end{definition}

Recall that $\Sp[\blank] \colon \An \to \PrLst$ is the unique colimit-preserving functor that sends the point to $\Sp$ (see e.g.\ \cite[Remark~4.16]{CCRY}).

\begin{lemma} \label{lemma:susp_map}
  For an anima $X$ over $B$, there is a canonical equivalence $\SS_B[X] \eq X_! X^* (\SS_B)$.
  For a map $f \colon X \to Y$ of animas over $B$, the induced map $\SS_B[X] \to \SS_B[Y]$ is canonically homotopic to the composite
  \[ \SS_B[X]  \eq  X_! X^* (\SS_B)  \eq  Y_! f_! f^* Y^* (\SS_B)  \xlongto{\epsilon}  Y_! Y^* (\SS_B)  \eq  \SS_B[Y] \]
  where $\epsilon$ is the counit of the adjunction $f_! \dashv f^*$.
\end{lemma}

\begin{proof}
  We begin by observing that the diagram
  \[
  \begin{tikzcd}
    \Anover{B} \rar{X^*} \dar[swap]{\eq} & \Anover{X} \rar{X_!} \dar{\eq} & \Anover{B} \dar{\eq} \\
    \Fun(B, \An) \rar{X^*} \dar[swap]{\Suspinf[B]} & \Fun(X, \An) \rar{X_!} \dar{\Suspinf[X]} & \Fun(B, \An) \dar{\Suspinf[B]} \\
    \Fun(B, \Sp) \rar{X^*} & \Fun(X, \Sp) \rar{X_!} & \Fun(B, \Sp)
  \end{tikzcd}
  \]
  commutes.
  This implies the first claim since $X_! \colon \Anover{X} \to \Anover{B}$ is given by composing with $X \to B$, so that $B \in \Anover{B}$ is mapped to $X \in \Anover{B}$ by the top row.
  For the second claim, we furthermore observe that \cref{lemma:mates} implies that the diagram
  \[
  \begin{tikzcd}
    f_! \Suspinf[X] f^* \rar{\eq} \dar[swap]{\eq} & f_! f^* \Suspinf[Y] \dar{\epsilon} \\
    \Suspinf[Y] f_! f^* \rar{\epsilon} & \Suspinf[Y]
  \end{tikzcd}
  \]
  commutes.
  Evaluating on $Y \in \Anover{Y}$ and applying $Y_!$ completes the proof, noting that the counit of the adjunction $f_! \dashv f^*$ at $Y \in \Anover{Y}$ is given by the map $X \to Y$ of animas over $Y$.
\end{proof}

\subsubsection*{The Wirthmüller context}

Given a map $f \colon X \to Y$ of animas, note that $f^* \colon \Sp[Y] \to \Sp[X]$ is strong symmetric monoidal.
In particular there is, for $P \in \Sp[X]$ and $Q \in \Sp[Y]$, a natural map
\[ f_! \bigl( f^*(P) \tensor_Y Q \bigr)  \xlongto{\eta}  f_! \bigl( f^*(P) \tensor_Y f^* f_! (Q) \bigr)  \eq  f_! f^* \bigl( P \tensor_X f_! (Q) \bigr)  \xlongto{\epsilon}  P \tensor_X f_! (Q) \]
which is an equivalence known as the \emph{projection formula} (see e.g.\ \cite[Proposition~6.8]{ABG}).

Given a pullback square of animas
\begin{equation} \label{eq:pullback}
\begin{tikzcd}
  A \rar{f} \dar[swap]{a} & B \dar{b} \\
  C \rar{g} & D
\end{tikzcd}
\end{equation}
the mates of the canonical equivalence $f^* b^* \eq a^* g^*$
\[ f_! a^*  \xlongto{\eq}  b^* g_!  \qquad \text{and} \qquad  b^* g_*  \xlongto{\eq}  f_* a^* \]
are equivalences themselves, known as the \emph{pull--push formulas} (see e.g.\ \cite[Proposition~4.3.3]{HL} and its adjoint).

When $X$ is a compact anima, there is an equivalence $X_*(\blank) \eq X_!(\blank \tensor \DS{X})$ for a certain $X$-spectrum $\DS{X}$ (this was first studied by Klein \cite{Kle}).
Interpreting $X_*$ as taking cohomology and $X_!$ as taking homology of $X$ with local coefficients, this equivalence can be viewed as a general form of twisted Poincaré duality.
We now recall an explicit and relative version of this.

\begin{definition} \label{def:DS}
  Let $f \colon X \to Y$ be a map of animas.
  We define the \emph{dualizing spectrum} of $f$ to be
  \[ \DS f  \defeq  (\pr_1)_* \Delta_! (\SS_X)  \in  \Sp[X] \]
  where $\Delta \colon X \to X \times_Y X$ is the diagonal and $\pr_1 \colon X \times_Y X \to X$ the first projection.
\end{definition}

\begin{lemma} \label{lemma:f_*}
  Let $f \colon X \to Y$ be a map of animas with compact fibers.
  Then there is a natural equivalence of functors $\Sp[X] \to \Sp[Y]$
  \[ f_!(\blank \tensor_X \DS{f})  \xlongto{\eq}  f_*(\blank) \]
  that is adjoint to the composite
  \begin{equation} \label{eq:f_*}
    \begin{aligned}
      f^* f_! \bigl( \blank \tensor_X (\pr_1)_* \Delta_! (\SS_X) \bigr)
      &\eq (\pr_2)_! \pr_1^* \bigl( \blank \tensor_X (\pr_1)_* \Delta_! (\SS_X) \bigr) \\
      &\eq (\pr_2)_! \bigl( \pr_1^* (\blank) \tensor_{X \times_Y X} \pr_1^* (\pr_1)_* \Delta_! (\SS_X) \bigr) \\
      &\xto{\epsilon} (\pr_2)_! \bigl( \pr_1^* (\blank) \tensor_{X \times_Y X} \Delta_! (\SS_X) \bigr) \\
      &\eq (\pr_2)_! \Delta_! \bigl( \Delta^* \pr_1^* (\blank) \tensor_X \SS_X \bigr) \\
      &\eq \id(\blank)
    \end{aligned}
  \end{equation}
  of the pull--push formula, the counit $\epsilon$ of the adjunction $\pr_1^* \dashv (\pr_1)_*$, and the projection formula.
\end{lemma}

\begin{proof}
  This follows for example from \cite[Lemma~3.6, Corollary~3.14, and Example~3.12]{Cno}.
\end{proof}

In the rest of this subsection, we collect various relationships between these equivalences that we will need throughout the paper.

\begin{observation} \label{lemma:DS_pullback}
  Given the pullback square \eqref{eq:pullback}, there is an equivalence
  \[ a^*(\DS{g}) = a^* (\pr_1)_* \Delta_! (\SS_C) \eq (\pr_1)_* (a \times a)^* \Delta_! (\SS_C) \eq (\pr_1)_* \Delta_! a^* (\SS_C) \eq (\pr_1)_* \Delta_! (\SS_A) = \DS{f} \]
  using the two pullback squares
  \[
  \begin{tikzcd}
    A \rar{\Delta} \dar[swap]{a} & A \times_B A \rar{\pr_1} \dar{a \times a} & A \dar{a} \\
    C \rar{\Delta} & C \times_D C \rar{\pr_1} & C
  \end{tikzcd}
  \]
  and the associated pull--push formulas.
\end{observation}

\begin{lemma} \label{lemma:norm_pull-push}
  Given the pullback square \eqref{eq:pullback}, the following pasting of the equivalences of \cref{lemma:f_*,lemma:DS_pullback} and the pull--push formula for $f_!$ and $g_!$
  \[
  \begin{tikzcd}
    \Sp[A] \ar{dd}[swap]{f_*} \drar[near end]{\tensor \DS f} &[-20] & &[-20] \ar{lll}[swap]{a^*} \dlar[swap, near end]{\tensor \DS g} \Sp[C] \ar{dd}{g_*} \\
    & \dlar[near start]{f_!} \Sp[A] & \lar[swap]{a^*} \Sp[C] \drar[swap, near start]{g_!} & \\
    \Sp[B] & & & \ar{lll}[swap]{b^*} \Sp[D]
  \end{tikzcd}
  \]
  is the equivalence from the pull--push formula for $f_*$ and $g_*$.
\end{lemma}

\begin{proof}
  This follows from a diagram chase using \cref{lemma:mates} and \cref{lemma:projection_pullback} below.
\end{proof}

\begin{lemma} \label{lemma:projection_pullback}
  Given the pullback square \eqref{eq:pullback}, let $P \in \Sp[D]$.
  Then the following cube commutes
  \[
  \begin{tikzcd}
     & \Sp[A] \ar{rr}{f_!} \ar[from=dd]{}[near start]{a^*} & & \Sp[B] \\
    \Sp[A] \ar[crossing over]{rr}[near end]{f_!} \urar[near end]{\tensor a^* g^* P} & & \Sp[B] \urar[swap, near start]{\tensor b^* P} & \\
     & \Sp[C] \ar{rr}[near start]{g_!} & & \Sp[D] \ar{uu}[swap]{b^*} \\
    \Sp[C] \ar{rr}{g_!} \ar{uu}{a^*} \urar[near end]{\tensor g^* P} & & \Sp[D] \ar[crossing over]{uu}[swap, near end]{b^*} \urar[swap, near start]{\tensor P}
  \end{tikzcd}
  \]
  where the faces are given by the monoidality of $a^*$ and $b^*$ and the respective pull--push and projection formulas (note that $a^* g^* P \eq f^* b^* P$).
\end{lemma}

\begin{proof}
  This follows from a diagram chase using \cref{lemma:mates}.
\end{proof}

\subsection{Rational parametrized spectra}

In this subsection, we recall the definition of rational parametrized spectra and the associated rationalization functor.
We furthermore prove that the resulting $\infty$-categories are invariant under rational equivalences of the bases (under certain nilpotence conditions).

\begin{notation}
  Let $X$ be an anima, $x \colon * \to X$ a map, $E$ an $X$-spectrum, and $n \in \ZZ$.
  Then we write $\pi_n(E, x) \defeq \pi_n(x^* E)$ for the fiberwise homotopy groups of $E$ at the point $x$.
\end{notation}

\begin{definition}
  Let $X$ be an anima.
  We write $\rat {(\blank)} \colon \Sp[X] \to \Sp[X]$ for the functor $\blank \tensor_X X^*(\H \QQ)$.
  A map of $X$-spectra $f \colon E \to E'$ is a \emph{rational equivalence} if $\rat f$ is an equivalence.
  An $X$-spectrum $E$ is \emph{rational} if $\pi_n(E, x)$ is uniquely divisible (i.e.\ a rational vector space) for all $n \in \ZZ$ and $x \in X$.
  We denote by $\Sp[X]^\QQ \subseteq \Sp[X]$ the full subcategory spanned by the rational $X$-spectra.
\end{definition}

For a spectrum $A$, the canonical map $\pi_*(A) \tensor \QQ \to \pi_*(\rat A)$ is an isomorphism.%
\footnote{The collection of those $A$ for which the map is an isomorphism contains $\SS$ and is closed under taking retracts, (co)fibers, and filtered colimits; by \cite[Proposition~7.2.4.2]{LurHA} this collections thus contains all objects of $\Sp$.}
This implies that the essential image of $\rat {(\blank)}$ is precisely $\Sp[X]^\QQ$.
Moreover a map of $X$-spectra $f \colon E \to E'$ is a rational equivalence if and only if $\pi_*(f, x) \tensor \QQ$ is an isomorphism for all $x \in X$.
In particular the unit map $\SS \to \H \QQ$ is a rational equivalence, so that $X^*(\H \QQ)$ is an idempotent object of $\Sp[X]$ in the sense of \cite[Definition~4.8.2.1]{LurHA}.
By \cite[Proposition~4.8.2.7]{LurHA}, this implies that $\rat {(\blank)} \colon \Sp[X] \to \Sp[X]^\QQ$ is a symmetric monoidal localization functor.

As in the case of animas, the rational homotopy theory of parametrized spectra is better behaved on a certain class of nilpotent objects.
The following definition is taken from \cite[§2.2]{Bra}.

\begin{definition} \label{def:nilpotent_spectrum}
  Let $X$ be an anima.
  Then an $X$-spectrum $E$ is \emph{nilpotent} if the action of $\pi_1(X, x)$ on $\pi_n(E, x)$ is nilpotent for each $n \in \ZZ$ and $x \in X$.
  We denote by $\Sp[X]_\nil \subseteq \Sp[X]$ the full subcategory spanned by the nilpotent $X$-spectra.
\end{definition}

\subsubsection*{Invariance under base change}

We will now prove that the $\infty$-category of nilpotent rational parametrized spectra is invariant under rational equivalences of the bases.
To this end, we adapt the argument for group actions of Hilton--Mislin--Roitberg \cite[Ch.~I, Theorem~4.8]{HMR}.
We begin with an unstable version of the desired statement.

\begin{definition} \label{def:fiberwise}
  Let $B$ be an anima, and $X$ an anima over $B$.
  We say that $X$ has a property \emph{fiberwise} if each fiber of the structure map $X \to B$ has the property.
  We write $\Anover[\QQ]{B}$, $\Anover[\ft]{B}$, $\Anover[\ge 1]{B}$, and $\Anover[\fnil]{B}$ for the full subcategories of $\Anover{B}$ spanned by those animas over $B$ that are fiberwise rational, fiberwise of finite rational type, fiberwise connected, and fiberwise nilpotent, respectively.
  When $B$ is connected, we furthermore write $\Anover[\nil]{B}$ for the full subcategory spanned by those objects $f \colon X \to B$ such that $f$ is nilpotent.
  Multiple superscripts indicate the intersection of the respective full subcategories.
\end{definition}

\begin{notation}
  For an anima $B$, we denote by $\Retr{B} \defeq \undercat {(\Anover{B})} {\id_B}$ the $\infty$-category of \emph{retractive animas over $B$}.
  As in \cref{def:fiberwise}, we have full subcategories $\Retr[\QQ]{B}$, $\Retr[\ft]{B}$, $\Retr[\ge 1]{B}$, and $\Retr[\fnil]{B}$, as well as $\Retr[\nil]{B}$ when $B$ is connected.
  Multiple superscripts indicate the intersection of the respective full subcategories.
\end{notation}

\begin{lemma} \label{lemma:param_An_rational_eq}
  Let $f \colon X \to Y$ be a rational equivalence between nilpotent animas.
  Then the respective pullback functors $f^*$ restrict to equivalences $\Anover[\QQ,\nil]{Y} \eq \Anover[\QQ,\nil]{X}$ and $\Retr[\QQ,\nil]{Y} \eq \Retr[\QQ,\nil]{X}$.
\end{lemma}

\begin{proof}
  We will prove the version for retractive animas; the proof of the other case is identical.
  First note that $f^*$ indeed restricts to a functor $\Retr[\QQ,\nil]{Y} \to \Retr[\QQ,\nil]{X}$.
  We will prove that it is an equivalence in the case of the rationalization map $r \colon X \to \rat X$.
  This implies the general case since $\rat f$ is an equivalence.
  
  To this end, consider the functor $q \colon \Retr{X} \to \Retr{\rat X}$ induced by rationalization, i.e.\ it sends an object $a \colon A \to X$ to $\rat a \colon \rat A \to \rat X$.
  Note that, given a nilpotent map $a \colon A \to X$ with $X$ nilpotent, the anima $A$ is nilpotent by \cref{lemma:nilpotent_composite}.
  Moreover the map $\rat a$ is nilpotent and the induced map from the fiber of $a$ to the fiber of $\rat a$ is a rationalization by \cref{lemma:rat_pullback}.
  In particular $q$ restricts to a functor $\Retr[\QQ,\nil]{X} \to \Retr[\QQ,\nil]{\rat X}$.
  We claim that it is quasi-inverse to $r^*$.
  
  To this end, let $a \colon A \to X$ be an object of $\Retr[\QQ,\nil]{X}$.
  Then we obtain a map of fiber sequences
  \[
  \begin{tikzcd}
    \fib a \rar \dar[swap]{\eq} & A \rar{a} \dar & X \dar{r} \\
    \fib \rat a \rar & \rat A \rar{\rat a} & \rat X
  \end{tikzcd}
  \]
  where the induced map on the fibers is an equivalence since $\fib a$ is rational.
  Hence the right-hand square is a pullback, which implies that the canonical natural transformation $\id \to r^* q$ is an equivalence on $\Retr[\QQ,\nil]{X}$.
  
  Now let $b \colon B \to \rat X$ be an object of $\Retr[\QQ,\nil]{\rat X}$ and consider the map of fiber sequences
  \[
  \begin{tikzcd}
    \fib b \rar \dar[equal] & P \rar \dar[swap]{p} & X \dar{r} \\
    \fib b \rar & B \rar{b} & \rat X
  \end{tikzcd}
  \]
  where we define $P$ to be the pullback of the lower right-hand corner.
  Note that $B$ is nilpotent; thus it is rational since both $\rat X$ and the fiber of $b$ are.
  Hence we obtain a map $\rat p \colon \rat P \to \rat B \eq B$, i.e.\ a natural transformation $\epsilon \colon q r^* \to \id$ on $\Retr[\QQ,\nil]{\rat X}$.
  Since $r$ is a rational equivalence and all animas involved are nilpotent, it follows from \cref{lemma:rat_pullback} that $\epsilon$ is an equivalence.
\end{proof}

Our goal is now to deduce a stable version of the preceding lemma.
To this end, we first prove that (rational) nilpotent parametrized spectra over $X$ arise as the stabilization of the $\infty$-category of (fiberwise rational) nilpotent retractive animas over $X$.

\begin{lemma} \label{lemma:param_spectrum_ob}
  Let $X$ be a connected anima.
  Then there are natural equivalences
  \[ \Sp[X]_\nil \eq \Spob(\Retr[\nil]{X})  \qquad \text{and} \qquad  \Sp[X]^\QQ_\nil \eq \Spob(\Retr[\QQ,\nil]{X}) \]
  of contravariant functors from animas to $\infty$-categories.
\end{lemma}

\begin{proof}
  Note that $\Sp = \Spob(\An) \eq \Spob(\Anpt)$ by \cite[Remark~1.4.2.18]{LurHA}; we will begin by proving that this is equivalent to $\Spob(\Anpt^{\ge 1})$, where $\Anpt^{\ge 1} \subset \Anpt$ denotes the full subcategory spanned by the connected pointed animas.
  More precisely it is given by the pullback
  \[
  \begin{tikzcd}
    \Anpt^{\ge 1} \ar[r] \ar[d, "\iota"'] & \{ * \} \ar[d] \\
    \Anpt \ar[r, "\pi_0"] & \Setpt
  \end{tikzcd}
  \]
  in $\Catinf$.
  Since $\pi_0 = \tau_{\leq 0}$ is a left adjoint by \cite[Proposition~5.5.6.18]{LurHTT}, this is also a pullback diagram in $\PrL$ by \cite[Proposition~5.5.3.13]{LurHTT}.
  In particular the functor $\iota$ is left adjoint; since it also preserves the terminal object, it thus commutes with suspensions.
  Since the suspension functor of $\Anpt$ lands in $\Anpt^{\ge 1}$, we obtain a commutative diagram
  \begin{equation} \label{eq:Ange1_Susp}
  \begin{tikzcd}
    \Anpt^{\ge 1} \ar[r, "\Susp"] \ar[d, "\iota"'] & \Anpt^{\ge 1} \ar[d, "\iota"] \\
    \Anpt \ar[r, "\Susp"] \ar[ur] & \Anpt
  \end{tikzcd}
\end{equation}
  in $\PrL$.
  By \cite[Example~4.8.1.23]{LurHA}, there is an equivalence $\Spob(\cat C) \eq \Sp \otimes \cat C$ for a presentable $\infty$-category $\cat C$.
  We claim that, if $\cat C$ is furthermore pointed, we have $\Sp \tensor \Susp_{\cat C} \eq \Susp_{\Spob(\cat C)}$.
  Indeed, using \cref{lemma:Spobj}, this follows from the equivalences
  \[ ( \Sp \tensor \Susp_{\cat C} ) \circ \Suspinf[\cat C]  \eq  {\Suspinf[\cat C]} \circ \Susp_{\cat C}  \eq  {\Susp_{\Spob(\cat C)}} \circ \Suspinf[\cat C] \]
  where the first one follows from $\Suspinf[\cat C] \eq {\Suspinf[\An]} \tensor \cat C$ (see \cite[Proof of Proposition~4.8.2.18]{LurHA}), and the second one from $\Suspinf \colon \cat C \to \Spob(\cat C)$ being a left adjoint, which also preserves the terminal object since $\cat C$ is pointed.
  Applying $\Sp \otimes \blank$ to the diagram \eqref{eq:Ange1_Susp} and using that $\Sp \tensor \Susp$ is an equivalence by the claim, we conclude that $\Spob(\Anpt^{\ge 1}) \eq \Sp$.
  
  Now note that
  \[ \Sp[X]  =  \Fun(X, \Sp)  \eq  \Fun \bigl( X, \Spob(\Anpt^{\ge 1}) \bigr)  \eq  \Spob \bigl( \Fun(X, \Anpt^{\ge 1}) \bigr) \]
  using \cite[Remark~1.4.2.9]{LurHA}.
  By unstraightening as in \eqref{eq:straighten}, we furthermore have equivalences
  \[ \Fun(X, \Anpt^{\ge 1}) \eq \undercat{\Fun(X, \An^{\ge 1})}{\const_*} \eq \undercat{(\Anover[\ge 1]{X})}{\id_X} = \Retr[\ge 1]{X} \]
  that are natural in $X$.
  
  We claim that the full subcategories $\Retr[\nil]{X} \subset \Retr[\ge 1]{X}$ and $\Retr[\QQ,\nil]{X} \subset \Retr[\nil]{X}$ are closed under finite limits; by \cite[Corollaries~4.4.2.4 and 4.4.2.5]{LurHTT} it is enough to show that they contain the terminal object and are closed under pullbacks.
  First note that $\Retr[\QQ,\nil]{X}$ does in fact contain the terminal object $\id_X$.
  By \cite[Corollary~3.8]{HRS}, a map $f \colon E \to X$ under $\id_X$ is nilpotent if and only if $F \defeq \fib f$ is nilpotent and the action of $\pi_1(X)$ on $\pi_n(F)$, induced by the corresponding functor $X \to \Anpt$, is nilpotent for all $n \ge 1$.
  Hence $\Retr[\nil]{X}$ is equivalent to the full subcategory $\Fun(X, \Anpt^\nil)^\nil \subset \Fun(X, \Anpt^{\ge 1})$ spanned by those functors $F$ such that, for all $x \in X$, the anima $F(x)$ is nilpotent and the induced action of $\pi_1(X, x)$ on $\pi_n(F(x))$ is nilpotent for all $n \ge 1$.
  Now let the following diagram be a map of horizontal fiber sequences in $\Fun(X, \Anpt)$
  \begin{equation} \label{eq:pullback_nilpotent}
    \begin{tikzcd}
      F \rar \dar[equal] & P \rar \dar & A \dar{a} \\
      F \rar & B \rar & C &
    \end{tikzcd}
  \end{equation}
  such that the right-hand square is a pullback and $A(x)$, $B(x)$, and $C(x)$ are connected for all $x \in X$.
  If $A(x)$, $B(x)$, and $C(x)$ are nilpotent, then by \cref{lemma:rat_pullback} so is $\tau_{\ge 1} P(x)$, where $\tau_{\ge 1} \colon \Anpt \to \Anpt^{\ge 1}$ denotes the right adjoint of the fully faithful inclusion $\iota \colon \Anpt^{\ge 1} \to \Anpt$, i.e.\ the functor given by taking the connected component of the basepoint.
  Since $\tau_{\ge 1}$ preserves limits, this implies that $\Fun(X, \Anpt^\nil) \subset \Fun(X, \Anpt^{\ge 1})$ is closed under pullbacks.
  If we furthermore assume that $A(x)$, $B(x)$, and $C(x)$ are rational, then so is $\tau_{\ge 1} P(x)$ by the same lemma.
  Hence $\Fun(X, \Anpt^{\QQ,\nil}) \subset \Fun(X, \Anpt^\nil)$ is closed under pullbacks.
  Now assume that in the situation of \eqref{eq:pullback_nilpotent} the action of $\pi_1(X, x)$ on the homotopy groups of $B(x)$ and $C(x)$ is nilpotent.
  Since the class of nilpotent actions is closed under taking subgroups, quotients, and extensions, this implies that the action of $\pi_1(X, x)$ on the homotopy groups of $\tau_{\ge 1} F(x)$ is nilpotent; if additionally the action on the homotopy groups of $A(x)$ is nilpotent, then the same argument implies that this also holds for $\tau_{\ge 1} P(x)$.
  Hence $\Fun(X, \Anpt^{\ge 1})^\nil \subseteq \Fun(X, \Anpt^{\ge 1})$ is closed under pullbacks, completing the proof of the claim.
  
  By the claim, the inclusion of $\Retr[\nil]{X} \eq \Fun(X, \Anpt^\nil)^\nil$ into $\Retr[\ge 1]{X} \eq \Fun(X, \Anpt^{\ge 1})$ preserves finite limits.
  In particular we can apply $\Spob(\blank)$ to obtain the commutative diagram
  \[
  \begin{tikzcd}[column sep = 35]
    \Spob \bigl( \Fun(X, \Anpt^\nil)^\nil \bigr) \rar \dar & \Spob \bigl( \Fun(X, \Anpt^{\ge 1}) \bigr) \dar & \lar{(\tau_{\ge 1})_*}[swap]{\eq} \Spob \bigl( \Fun(X, \Anpt) \bigr) \dar{\Loopsinf[X]} \\
    \Fun(X, \Anpt^\nil)^\nil \rar & \Fun(X, \Anpt^{\ge 1}) & \lar[swap]{\tau_{\ge 1}} \Fun(X, \Anpt)
  \end{tikzcd}
  \]
  where all vertical maps are given by the corresponding $\Loopsinf$ functor.
  Note that $\tau_{\ge 1} \after \Loopsinf[X]$ lands in $\Fun(X, \Anpt^\nil)$ since its value at any $x \in X$ is an H-space and hence simple (see e.g.\ \cite[Corollary~1.4.5]{MP}).
  Moreover an object $M \in \Spob ( \Fun(X, \Anpt) ) \eq \Sp[X]$ lies in $\Sp[X]_\nil$ if and only if $\tau_{\ge 1} \Loopsinf[X] (\Susp[n] M)$ lies in $\Fun(X, \Anpt^{\ge 1})^\nil$ for all $n \in \ZZ$.
  Hence the fully faithful functor $\Spob(\Retr[\nil]{X}) \to \Spob(\Retr[\ge 1]{X})$ restricts to an equivalence $\Spob(\Retr[\nil]{X}) \eq \Sp[X]_\nil$.
  By a similar argument the fully faithful functor $\Spob(\Retr[\QQ,\nil]{X}) \to \Spob(\Retr[\ge 1]{X})$ restricts to an equivalence $\Spob(\Retr[\QQ,\nil]{X}) \eq \Sp[X]^\QQ_\nil$.
\end{proof}

\begin{proposition} \label{lemma:param_Sp_rational_eq}
  Let $f \colon X \to Y$ be a rational equivalence between nilpotent animas.
  Then the functor $f^*$ restricts to an equivalence
  \[ \Sp[Y]^\QQ_\nil  \xlongto{\eq}  \Sp[X]^\QQ_\nil \]
  of $\infty$-categories.
\end{proposition}

\begin{proof}
  This follows from \cref{lemma:param_spectrum_ob,lemma:param_An_rational_eq}.
\end{proof}

\section{The fiberwise THH transfer as a trace} \label{sec:trace}

In this section, we begin by recalling the definition of topological Hochschild homology (THH for short) as a higher categorical trace, following Hoyois--Scherotzke--Sibilla \cite{HSS}.
We then use work of Carmeli--Cnossen--Ramzi--Yanovski \cite{CCRY} to prove that fiberwise THH, and its transfer map, can be computed internally to parametrized spectra.

\subsection{Dualizable objects and the trace}

As a preparation for the rest of this section, we recall here the notion of a dualizable object of a symmetric monoidal ($\infty$-)category, as well as the definition of the trace of an endomorphism of a dualizable object.

\begin{definition} \label{def:dualizable}
  An object $C$ of a symmetric monoidal $\infty$-category $\cat C$ is \emph{dualizable} if there exists an object $\dual C \in \cat C$ and maps $\coev_C \colon \unit \to C \tensor \dual C$ and $\ev_C \colon \dual C \tensor C \to \unit$ such that the \emph{triangle identities} hold, i.e.\ both of the composites
  \[ C  \xlongto{{\coev_C} \tensor \id}  C \tensor \dual C \tensor C  \xlongto{{\id} \tensor \ev_C}  C  \qquad \text{and} \qquad  \dual C  \xlongto{{\id} \tensor \coev_C}  \dual C \tensor C \tensor \dual C  \xlongto{{\ev_C} \tensor \id}  \dual C \]
  are homotopic to the identity.
\end{definition}

By the $1$-dimensional cobordism hypothesis (see Harpaz \cite{Har1}), given a dualizable object $C$ of $\cat C$, the space of choices for $\dual C$, $\ev_C$, $\coev_C$, and the coherence data connecting them, is contractible.

\begin{definition}
  Let $f \colon X \to Y$ be a map between dualizable objects of a symmetric monoidal $\infty$-category.
  Then the composite
  \[ \dual{f}  \colon  \dual{Y}  \xlongto{{\id} \tensor \coev_X}  \dual{Y} \tensor X \tensor \dual{X}  \xlongto{{\id} \tensor f \tensor \id}  \dual{Y} \tensor Y \tensor \dual{X}  \xlongto{\ev_Y \tensor \id}  \dual{X} \]
  is the \emph{dual morphism} of $f$.
\end{definition}

\begin{observation} \label{obs:dual_morphism}
  Let $f \colon X \to Y$ be a map between dualizable objects of a symmetric monoidal $\infty$-category.
  Then the two diagrams
  \[
  \begin{tikzcd}
    \unit \rar{\coev} \dar[swap]{\coev} & Y \tensor \dual Y \dar{{\id} \tensor \dual f} & \dual Y \tensor X \rar{{\id} \tensor f} \dar[swap]{\dual f \tensor \id} & \dual Y \tensor Y \dar{\ev} \\
    X \tensor \dual X \rar{f \tensor \id} & Y \tensor \dual X & \dual X \tensor X \rar{\ev} & \unit
  \end{tikzcd}
  \]
  canonically commute using the triangle identities.
\end{observation}

\begin{definition} \label{def:trace}
  The \emph{trace} of an endomorphism $f \colon C \to C$ of a dualizable object in a symmetric monoidal $\infty$-category $\cat C$ is defined to be the composite endomorphism
  \[ \tr(f)  \colon  \unit  \xlongto{\coev}  C \tensor \dual C  \xlongto{f \tensor \id}  C \tensor \dual C  \eq  \dual C \tensor C  \xlongto{\ev}  \unit \]
  of the unit of $\cat C$.
  The \emph{dimension} $\dim(C)$ of a dualizable object $C \in \cat C$ is $\tr(\id_C)$.
\end{definition}

We will recall in the next subsection how to promote $\dim(\blank)$ to a functor.

\subsection{THH as a higher categorical trace}

We will require some basics of $(\infty, 2)$-category theory, for which we follow Gaitsgory--Rozenblyum \cite{GR} (also see Loubaton--Ruit \cite{LR25}, who explain that all assertions claimed in \cite[Ch.~10, §0.4]{GR} have now been proved).
In particular, for us an $(\infty,2)$-category is a complete Segal object $\opcat{\Delta} \to \Catinf$ and a $2$-functor $\cat C \to \cat D$ between $(\infty, 2)$-categories is a natural transformation of functors $\opcat{\Delta} \to \Catinf$.
A natural transformation between such $2$-functors is a $2$-functor $\cat C \times [1] \to \cat D$; this yields in particular a notion of adjunctions between $(\infty, 2)$-categories.

\begin{notation}
  For an $(\infty, 2)$-category $\cat C$, we write $\maxsub{1} \cat C$ for its underlying $\infty$-category, i.e.\ the maximal sub-$\infty$-category.
  Given another $(\infty, 2)$-category $\cat D$, we write $\Fun(\cat C, \cat D)$ for the $(\infty, 2)$-category of $2$-functors $\cat C \to \cat D$ (see \cite[Ch.~10, §2.5]{GR}).
\end{notation}

\begin{definition}
  Let $\cat C$ be an $(\infty, 2)$-category.
  We say that $\cat C$ \emph{admits limits} if the diagonal functor $\cat C \to \Fun(\cat I, \cat C)$ admits a right adjoint $\lim_{\cat I}$ for every small $(\infty, 1)$-category $\cat I$.
\end{definition}

For an $(\infty, 1)$-category $\cat I$, we have $\maxsub{1} \Fun(\cat I, \cat C) \eq \Fun(\cat I, \maxsub{1} \cat C)$.
In particular limits in $\cat D$ can be computed in $\maxsub{1} \cat D$.

\begin{notation}
  We write $\twoCatinf$ for the $(\infty, 2)$-category of $\infty$-categories (see \cite[Ch.~10, §2.4]{GR}), and $\catMod{\Sp} = \twoPrLst$ for the symmetric monoidal $(\infty, 2)$-category of presentable stable $\infty$-categories and left adjoint functors (see \cite[Ch.~10, §2.4 and Ch.~1, 6.1.8]{GR} and \cite[§4.4]{HSS}).
\end{notation}

Recall that the underlying $\infty$-categories of $\twoCatinf$ and $\catMod{\Sp}$ are $\Catinf$ and $\PrLst$, respectively.
Furthermore note that the $(\infty, 2)$-categories $\twoCatinf$ and $\catMod{\Sp}$ admit limits (e.g.\ by combining \cite[Theorem~2.6]{DM} and \cite[Proposition~4.1.2]{AGH}).

\begin{notation}
  Given an object $C$ of an $(\infty, 2)$-category $\cat C$, we denote by $\Loops[C] \cat C$ the $\infty$-category of endomorphisms of $C$.
  When $\cat C$ is monoidal with unit $\unit$, we abbreviate $\Loops \cat C \defeq \Loops[\unit] \cat C$.
\end{notation}

Note that $\Loops \catMod{\Sp}$ is the $\infty$-category of left adjoint functors $\Sp \to \Sp$, which is equivalent to $\Sp$ itself via evaluation at the sphere spectrum $\SS$ (see \cite[Corollary~1.4.4.6]{LurHA}).

\begin{notation}
  For a symmetric monoidal $(\infty, 2)$-category $\cat C$, we denote by $\dbl{\cat C}$ the $\infty$-category of dualizable objects in $\cat C$ as in \cite[Definition~2.18]{CCRY}.
  We write $\dim \colon \dbl{\cat C} \to \Loops \cat C$ for the dimension functor of \cite[Definition~2.24]{CCRY}.
  We furthermore write
  \[ \Dim  \colon  \dbl{(\catMod{\Sp})}  \xlongto{\dim}  \Loops \catMod{\Sp}  \eq  \Sp \]
  (cf.\ \cite[Definition~4.1]{CCRY}).
\end{notation}

Recall that the forgetful functor $\dbl{\cat C} \to \cat C$ identifies $\dbl{\cat C}$ with the subcategory of $\maxsub{1} \cat C$ given by the dualizable objects and the left adjoint morphisms, see \cite[Corollary~2.21]{CCRY}.
By construction, the functor $\dim$ sends a dualizable object $C$ of $\cat C$ to the dimension of $C$ (in the sense of \cref{def:trace}) considered as an object of $\Loops \cat C$.

Writing $\THH(\cat D)$ for the topological Hochschild homology of a stable $\infty$-category $\cat D$ (see e.g.\ \cite[Definition~4.3]{HNS}), there is an equivalence $\THH(\cat D) \eq \Dim(\Ind \cat D)$ by \cite[Proposition~4.24]{HSS} (see also \cite[Remark~4.38]{HNS}).
For an anima $X$ we in particular have $\THH(\cptobj{\Sp[X]}) \eq \Dim(\Sp[X])$, where we note that $\Sp[\blank] \colon \An \to \maxsub{1} \catMod{\Sp}$ factors through $\dbl{(\catMod{\Sp})}$ by \cite[Proposition~4.11 and Corollary~3.12]{CCRY}.
This justifies the following definition.

\begin{definition}
  We write
  \[ \THH  \colon  \An  \xlongto{\Sp[\blank]}  \dbl{(\catMod{\Sp})}  \xlongto{\Dim}  \Sp \]
  for the \emph{topological Hochschild homology} of animas.
\end{definition}

Recall that there is a natural equivalence $\THH(X) \eq \SS[\L X]$, where $\L X \defeq \Map(\Sphere 1, X)$ denotes the free loop space of $X$ (see e.g.\ \cite[Theorem~4.40]{CCRY} or \cref{lemma:THH_loops} below).
In particular this yields a canonical equivalence $\THH(*) \eq \SS$.

\begin{definition}
  For an anima $B$, we write
  \begin{align*}
    \Sp[\blank]_B  &\colon  \Anover{B}  \eq  \Fun(B, \An)  \xlongto{\Sp[\blank]_*}  \Fun \bigl( B, \dbl{(\catMod{\Sp})} \bigr) \\
    \THH_B  &\colon  \Anover{B}  \xlongto{\Sp[\blank]_B}  \Fun \bigl( B, \dbl{(\catMod{\Sp})} \bigr)  \xlongto{\Dim_*}  \Fun(B, \Sp) = \Sp[B]
  \end{align*}
  and call the latter \emph{fiberwise topological Hochschild homology}.
\end{definition}

We now define the fiberwise THH transfer, which is a wrong way map on $\THH_B$.
To this end, we use that $\THH_B$ factors through the larger category $\Fun ( B, \dbl{(\catMod{\Sp})} )$, and that the pointwise right adjoint of $\Sp[f]_B$ sometimes defines a map in that $\infty$-category.
More specifically, we observe the following.

\begin{observation}
  For a map of animas $f \colon X \to Y$ with compact fibers, the right adjoint $f_*$ of the functor $f^* \colon \Sp[Y] \to \Sp[X]$ preserves colimits by \cref{lemma:f_*} and is thus a left adjoint.
  Hence $f^*$ is an internal left adjoint in the $(\infty, 2)$-category $\catMod{\Sp}$.
  
  Let now $f \colon X \to Y$ be a map of animas over an anima $B$.
  Its image $\Sp[f]_B$ in $\Fun(B, \catMod{\Sp})$ is pointwise an internal left adjoint and hence an internal left adjoint by a result of Haugseng \cite[Theorem~4.6]{Hau}.
  The internal right adjoint of $\Sp[f]_B$ is pointwise given by pulling back and thus again an internal left adjoint when the fibers of $f$ are compact.
\end{observation}

\begin{definition} \label{def:fiberwise_astrology}
  Let $f \colon X \to Y$ be a map of animas over $B$.
  We write $f_! \defeq \Sp[f]_B$ as a map of $\Fun ( B, \catMod{\Sp} )$ and $f^* \colon \Sp[Y]_B \to \Sp[X]_B$ for its internal right adjoint.
  If $f$ has compact fibers, we define the \emph{fiberwise THH transfer} of $f$ to be the image of $f^*$ under $\Dim_* \colon \Fun(B, \dbl{(\catMod{\Sp})}) \to \Fun(B, \Sp)$, which we also denote by $f^* \colon \THH_B(Y) \to \THH_B(X)$.
  In the case that $Y = B$, we call the resulting map the \emph{fiberwise THH Euler characteristic} and denote it by $\chi \colon \SS_B \eq \THH_B(B) \to \THH_B(X)$.
\end{definition}

\subsection{Fiberwise THH via parametrized \texorpdfstring{$\infty$}{infinity}-categories}

In this subsection, we prove that fiberwise THH and the fiberwise transfer can also be computed as a trace on the $(\infty, 2)$-category $\Fun ( B, \catMod{\Sp} )$.
The following is a variant of \cite[Lemma~5.9]{CCRY}.

\begin{lemma} \label{lemma:parametrized_trace}
  Let $B$ be an anima and $\cat C$ a symmetric monoidal $(\infty, 2)$-category.
  Then, in the following diagram, there are essentially unique vertical maps
  \[
  \begin{tikzcd}
    \Fun (B, \cat C) \dar[equal] & \lar \dbl { \Fun (B, \cat C) }  \ar[r, "\dim"] \ar[d, "\eq"'] &  \Loops \Fun(B, \cat C)  \ar[d, "\eq"] \\
    \Fun (B, \cat C) & \lar \Fun \bigl( B, \dbl{\cat C} \bigr)  \ar[r, "\dim_*"] & \Fun(B, \Loops \cat C)
  \end{tikzcd}
  \]
  that are natural in $B$ and restrict to the identity for $B = *$.
  Furthermore the vertical maps are equivalences and the diagram commutes.
\end{lemma}

\begin{proof}
  First note that $\Fun(\blank, \cat C)$ preserves limits as a functor in $\opcat{\An}$ (since the inclusion of $\An$ into the $\infty$-category of $(\infty, 2)$-categories is left adjoint and hence preserves colimits), and that $\dbl{(\blank)}$ and $\Loops$ preserve limits by (the proof of) \cite[Lemma~2.26]{CCRY}.
  Then the claim follows from \cref{obs:assembly}.
\end{proof}

To apply this lemma, we require an explicit description of the functor $\Loops \Fun(B, \cat C) \to \Fun(B, \Loops \cat C)$.
The following lemma provides such a description (specialized to $\cat C = \catMod{\Sp}$, though the argument works more generally).

\begin{lemma} \label{lemma:loop_fun_eval}
  Let $B$ be an anima.
  Then the composite
  \begin{equation} \label{eq:loop_fun_eval}
    \Loops \Fun(B, \catMod{\Sp})  \xlongto{\Loops {\lim}}  \Loops[\Fun(B, \Sp)] \catMod{\Sp}  \xlongto{\ev_{\SS_B}}  \Fun(B, \Sp)
  \end{equation}
  is homotopic to the map $\Loops \Fun(B, \catMod{\Sp}) \to \Fun(B, \Loops \catMod{\Sp})$ of \cref{lemma:parametrized_trace}.
\end{lemma}

\begin{proof}
  Throughout this proof, for an $(\infty, 2)$-category $\cat C$ and two objects $X, Y \in \cat C$, we write $\cat C(X, Y)$ for the $\infty$-category of morphisms from $X$ to $Y$.
  As explained in \cite[Ch.~11, §5.3]{GR}, this assignment can be made fully functorial in $X$ and $Y$, i.e.\ it lifts to a $2$-functor $\oneopcat{\cat C} \times \cat C \to \twoCatinf$, where $\oneopcat{\cat C}$ denotes $\cat C$ with the direction of the $1$-morphisms inverted (cf.\ \cite[Ch.~10, 2.1.6]{GR}).
  
  For $B = *$ the composite \eqref{eq:loop_fun_eval} is the canonical equivalence $\ev_\SS \colon \Loops \catMod{\Sp} \to \Sp$.
  Hence the claim follows once we show that the composite can be lifted to a natural transformation in $B \in \opcat{\An}$.
  To this end, we denote by $B^* \colon \catMod{\Sp} \to \Fun(B, \catMod{\Sp})$ the functor given by pulling back along the constant map $B \to *$ and by $\lim_B$ its right adjoint.
  This adjunction yields, for $X \in \catMod{\Sp}$ and $Y \in \Fun(B, \catMod{\Sp})$, an equivalence natural in $X$ and $Y$
  \[ \Fun(B, \catMod{\Sp})(B^* X, Y)  \xlongto{\lim_B}  \catMod{\Sp}(\lim_B B^* X,  \lim_B Y)  \xlongto{\eta^*}  \catMod{\Sp}(X, \lim_B Y) \]
  where $\eta \colon X \to \lim_B B^* X$ is the unit of the adjunction (see e.g.\ \cite[Proposition~4.1.2]{AGH}).
  Specializing to $Y = B^* \Sp$, we claim that the resulting equivalence
  \begin{equation} \label{eq:limcstadjunction}
    \Fun(B, \catMod{\Sp})(B^* (\blank), B^* \Sp) \eq \catMod{\Sp}(\blank, \lim_B B^* \Sp)
  \end{equation}
  can be lifted to a natural equivalence of functors $\opcat{\An} \to \Fun(\oneopcat{\catMod{\Sp}}, \Catinf)$.
  First note that the left-hand side is indeed such a functor, and that it lands in the representable functors (using the equivalence pointwise).
  By the fully faithfulness of the $(\infty, 2)$-categorical Yoneda embedding (see \cite[Ch.~11, Proposition~5.3.7]{GR}) we thus obtain a functor $\opcat{\An} \to \catMod{\Sp}$ that is pointwise given by $B \mapsto \lim_B B^* \Sp$ and that makes \eqref{eq:limcstadjunction} natural.
  Applying \eqref{eq:limcstadjunction} to $\Sp$, we obtain an equivalence
  \[ \Fun(B, \Sp)  \eq  \Fun(B, \Loops \catMod{\Sp})  \eq  \Loops \Fun(B, \catMod{\Sp})  \eq  \lim_B B^* \Sp \]
  that is natural in $B$.
  Lastly we claim that the composite
  \[ \Sp  \xlongto{\eta}  \lim_B B^* \Sp  \eq  \Fun(B, \Sp) \]
  is given by pulling back along $B \to *$.
  Since this is evidently true for $B = *$, it is enough to prove that $\eta$ is natural in $B$.
  This follows from the fact that it corresponds to $\id_{B^*\Sp}$ under \eqref{eq:limcstadjunction}.
\end{proof}

\subsection{Fiberwise THH via parametrized spectra}

We now prove the main result of this section, providing a formula for fiberwise THH and its transfer internally to parametrized spectra.
Recall the dualizing spectrum $\DS f$ of a map $f$ of animas from \cref{def:DS}.

\begin{proposition} \label{lemma:transfer_diag}
  Let $X$ be an anima over $B$.
  Then $\THH_B(X)$ is given by evaluating the following composite at $\SS_B$
  \[ \Sp[B]  \xlongto{X^*}  \Sp[X]  \xlongto{\Delta_!}  \Sp[X \times_B X]  \xlongto{\Delta^*}  \Sp[X]  \xlongto{X_!}  \Sp[B] \]
  where $X \colon X \to B$ denotes the structure map and $\Delta \colon X \to X \times_B X$ the diagonal.
  Furthermore, let $f \colon X \to Y$ be a map of animas over $B$ such that $f$ has compact fibers.
  Then the fiberwise THH transfer $f^* \colon \THH_B(Y) \to \THH_B(X)$ is given by evaluating the following pasting at $\SS_B$
  \begin{equation} \label{eq:transfer_diag}
  \begin{tikzcd}[sep = 17]
      \Sp[B] \rar{Y^*} \dar[swap]{X^*} &[-10] \Sp[Y] \rar{\Delta_!} & \Sp[Y \times_B Y] \dar[swap]{({\id} \times f)^*} \rar[equals] &[-10] \Sp[Y \times_B Y] \ar{dd}{(f \times \id)^*} \ar[bend left = 15]{ddrr}{\id}[swap, name=U]{} & &[-10] \\
      \Sp[X] \dar[swap]{\Delta_!} \ar[phantom]{rr}{\text{\eqref{eq:transfer_diag_1}}} & & \Sp[Y \times_B X] \dar[swap]{\tensor \pr_2^*(\DS{f})} & & & \\
      \Sp[X \times_B X] \ar{rr}{(f \times \id)_*} \ar[bend right = 15]{ddrr}[name=C]{}[swap]{\id} & & \Sp[Y \times_B X] \ar{dd}[swap]{(f \times \id)^*} \ar[to=C, Rightarrow, shorten = 1em, "\epsilon"'] & \Sp[X \times_B Y] \dar{({\id} \times f)^*} \ar{rr}[swap]{(f \times \id)_*} \ar[from=U, Rightarrow, shorten = 1em, "\eta"'] & & \Sp[Y \times_B Y] \dar{\Delta^*} \\
       & & & \Sp[X \times_B X] \dar{\tensor \pr_2^*(\DS{f})} \ar[phantom]{rr}{\text{\eqref{eq:transfer_diag_2}}} & & \Sp[Y] \dar{Y_!} \\
       & & \Sp[X \times_B X] \rar[equals] & \Sp[X \times_B X] \rar{\Delta^*} & \Sp[X] \rar{X_!} & \Sp[B]
  \end{tikzcd}
  \end{equation}
  where $\eta$ and $\epsilon$ are the respective (co)unit.
  Using \cref{lemma:f_*,lemma:DS_pullback}, the projection formula, and the pull--push formula, the upper left-hand square commutes via the composite equivalence
  \begin{align}
    ({\id_Y} \times f)^* \Delta_! Y^*(\blank) \tensor \pr_2^*(\DS f)
    &\eq (f, \id_X)_! f^* Y^* (\blank) \tensor \pr_2^*(\DS f) \nonumber \\
    &\eq (f, \id_X)_! \bigl( f^* Y^* (\blank) \tensor (f, \id_X)^* \pr_2^*(\DS f) \bigr) \nonumber \\
    &\eq (f \times \id_X)_! \Delta_! \bigl( X^* (\blank) \tensor \Delta^* \pr_1^*(\DS f) \bigr) \label{eq:transfer_diag_1} \\
    &\eq (f \times \id_X)_! \bigl( \Delta_! X^* (\blank) \tensor \DS{f \times \id} \bigr) \nonumber \\
    &\eq (f \times \id_X)_* \Delta_! X^* (\blank) \nonumber
  \intertext{the lower right-hand square commutes via the composite equivalence}
    Y_! \Delta^* (f \times \id_Y)_* (\blank) \nonumber 
    &\eq Y_! \Delta^* (f \times \id_Y)_! (\blank \tensor \DS {f \times \id}) \nonumber \\
    &\eq Y_! f_! (\id_X, f)^* (\blank \tensor \pr_1^*(\DS f)) \nonumber \\
    &\eq X_! \bigl( (\id_X, f)^* (\blank) \tensor (\id_X, f)^* \pr_1^*(\DS f) \bigr) \label{eq:transfer_diag_2} \\
    &\eq X_! \bigl( \Delta^* ({\id_X} \times f)^* (\blank) \tensor \Delta^* \pr_2^*(\DS f) \bigr) \nonumber \\
    &\eq X_! \Delta^* \bigl( ({\id_X} \times f)^* (\blank) \tensor \pr_2^*(\DS f) \bigr) \nonumber
  \end{align}
  and the middle square commutes via the obvious equivalence.
\end{proposition}

\begin{proof}
  By \cref{lemma:parametrized_trace,lemma:loop_fun_eval} there is a commutative diagram
  \[
  \begin{tikzcd}
    \Fun ( B, \maxsub{1} \catMod{\Sp} ) & \lar[hook'] \dbl { \Fun (B, \catMod{\Sp}) } \rar{\dim} \dar[swap]{\eq} & \Loops \Fun(B, \catMod{\Sp}) \rar{\Loops {\lim}} \dar{\eq} & \Loops[{\Sp[B]}] \catMod{\Sp} \dar{\ev_{\SS_B}} \\
    \Anover{B} \uar{\Sp[\blank]_B} \rar{\Sp[\blank]_B} & \ular[hook'] \Fun ( B, \dbl{\catMod{\Sp}} ) \rar{\Dim_*} & \Fun(B, \Sp) \rar[equal] & \Sp[B]
  \end{tikzcd}
  \]
  in which $\THH_B$ is the bottom horizontal composite.
  Hence, to compute $\THH_B(X)$, we can take the dimension of $\Sp[X]_B$, apply $\Loops{\lim}$, and then evaluate at $\SS_B$.
  
  By \cite[Proposition~4.11~(1)]{CCRY} the functor $\Sp[\blank] \colon \An \to \catMod{\Sp}$ is a so-called symmetric monoidal bivariant theory; by the same argument as in \cite[Proof of Lemma~3.16]{CCRY}, the functor $\Sp[\blank]_B \colon \Anover{B} \to \Fun(B, \catMod{\Sp})$ is thus a symmetric monoidal bivariant theory as well.
  By \cite[(Proof of) Corollary~3.12]{CCRY}, this implies that $\Sp[X]_B$ is self-dual with (co)evaluation given by
  \begin{align*}
    \coev  &\colon  \Sp[B]_B  \xlongto{X^*}  \Sp[X]_B  \xlongto{\Delta_!}  \Sp[X \times_B X]_B  \eq  \Sp[X]_B \tensor \Sp[X]_B \\
    \ev  &\colon  \Sp[X]_B \tensor \Sp[X]_B  \eq  \Sp[X \times_B X]_B  \xlongto{\Delta^*}  \Sp[X]_B  \xlongto{X_!}  \Sp[B]_B
  \end{align*}
  with notation as in \cref{def:fiberwise_astrology} (and using that $\Sp[\blank]_B$ is in particular strong symmetric monoidal).
  The triangle identities are provided by the following commutative diagram (resp.\ the analogous one with the order of all products inverted)
  \[
  \begin{tikzcd}[column sep = 40, row sep = 30]
    \Sp[X]_B \dar[swap]{(X \times \id)^*} \rar{\id} & \Sp[X]_B \drar{\Delta_!} \rar{\id} & \Sp[X]_B \\
    \Sp[X \times_B X]_B \rar{(\Delta \times \id)_!} \urar{\Delta^*} & \Sp[X \times_B X \times_B X]_B \rar{({\id} \times \Delta)^*} & \Sp[X \times_B X]_B \uar[swap]{({\id} \times X)_!}
  \end{tikzcd}
  \]
  where the middle square commutes via the pull--push formula.
  
  Hence the dimension of $\Sp[X]_B$ is the composite
  \[ \Sp[B]_B  \xlongto{X^*}  \Sp[X]_B  \xlongto{\Delta_!}  \Sp[X \times_B X]_B  \xlongto{\Delta^*}  \Sp[X]_B  \xlongto{X_!}  \Sp[B]_B \]
  and we now want to evaluate the functor ${\ev_{\SS_B}} \after \Loops{\lim_B}$ on this composite.
  Since the functor $\Sp[\blank] \colon \An \to \maxsub{1} \catMod{\Sp}$ preserves colimits and $\colim_B \eq \lim_B$ as functors $\Fun(B, \maxsub{1} \catMod{\Sp}) \to \maxsub{1} \catMod{\Sp}$ (this follows from the corresponding fact for $\PrL$, see \cite[Example~4.3.11]{HL} or \cite[Proposition~2.26]{Ben}, combined with \cite[Corollary~4.2.3.3]{LurHA}), we furthermore have a commutative diagram
  \[
  \begin{tikzcd}
    \overcat \An B \rar{\eq} \drar[swap]{\pr} & \Fun(B, \An) \dar{\colim} \rar{\Sp[\blank]_*} &[25] \Fun(B, \maxsub{1} \catMod{\Sp}) \dar{\lim} \\
    & \An \rar{\Sp[\blank]} & \maxsub{1} \catMod{\Sp}
  \end{tikzcd}
  \]
  i.e.\ a natural equivalence ${\lim_B} \after \Sp[\blank]_B \eq \Sp[\blank]$.
  Noting that $\lim_B$ is a $2$-functor and hence preserves internal adjunctions completes the proof of the first claim.
  
  For the second claim, recall that the fiberwise THH transfer $f^* \colon \THH_B(Y) \to \THH_B(X)$ is given by $\Dim_*(f^*)$; we will denote the internal right adjoint of $f^*$ in $\Fun(B, \catMod{\Sp})$ by $f_* \colon \Sp[X]_B \to \Sp[Y]_B$.
  By \cite[Lemma~2.25]{CCRY} the dimension of $f^*$ is given by the pasting
  \begin{equation} \label{eq:transfer_diag_pre}
  \begin{tikzcd}[sep = 15]
    \Sp[B]_B \rar{Y^*} \dar[swap]{X^*} &[-12] \Sp[Y]_B \rar{\Delta_!} & \Sp[Y \times_B Y]_B \ar{dd}[swap]{{\id} \tensor \dual{f_*}} \rar[equals] & \Sp[Y \times_B Y]_B \ar{dd}{f^* \tensor \id} \ar[bend left = 15]{ddrr}{\id}[swap, name=U]{} & &[-12] \\
    \Sp[X]_B \dar[swap]{\Delta_!} & & & & & \\
    \Sp[X \times_B X]_B \ar{rr}{f_* \tensor \id} \ar[bend right = 15]{ddrr}[name=C]{}[swap]{\id} & & \Sp[Y \times_B X]_B \ar{dd}[swap]{f^* \tensor \id} \ar[to=C, Rightarrow, shorten = 1em, "\epsilon"'] & \Sp[X \times_B Y]_B \ar{dd}{{\id} \tensor \dual{f_*}} \ar{rr}[swap]{f_* \tensor \id} \ar[from=U, Rightarrow, shorten = 1em, "\eta"'] & & \Sp[Y \times_B Y]_B \dar{\Delta^*} \\
    & & & & & \Sp[Y]_B \dar{Y_!} \\
    & & \Sp[X \times_B X]_B \rar[equals] & \Sp[X \times_B X]_B \rar{\Delta^*} & \Sp[X]_B \rar{X_!} & \Sp[B]_B
  \end{tikzcd}
  \end{equation}
  where we implicitly identify $\Sp[\blank]_B \tensor \Sp[\blank]_B \eq \Sp[\blank \times_B \blank]_B$ and use \cref{obs:dual_morphism} for the upper left-hand and lower right-hand squares.
  Under this identification, the map $f^* \tensor \id$ corresponds to $(f \times \id)^*$ since $\blank \tensor \Sp[Y]_B$ is a $2$-functor and hence preserves internal adjunctions; analogously ${\id} \tensor f^*$ corresponds to $({\id} \times f)^*$, $f_* \tensor \id$ to $(f \times \id)_*$, and ${\id} \tensor f_*$ to $({\id} \times f)_*$.
  Expanding its definition using similar arguments, we see that $\dual{f_*}$ is given by the composite
  \[
  \begin{tikzcd}
    \Sp[Y]_B \rar{\pr_1^*} & \Sp[Y \times_B X]_B \rar{({\id} \times \Delta)_!} &[10] \Sp[Y \times_B X \times_B X]_B \dar{({\id} \times f \times \id)_*} \\
    \Sp[X]_B & \lar[swap]{(\pr_2)_!} \Sp[Y \times_B X]_B & \lar[swap]{(\Delta \times \id)^*} \Sp[Y \times_B Y \times_B X]_B
  \end{tikzcd}
  \]
  which is preserved when applying the $2$-functor $\lim_B$.
  Hence we have
  \begin{align*}
    \lim_B(\dual{f_*})(\blank) &\eq (\pr_2)_! (\Delta \times \id)^* ({\id} \times f \times \id)_* ({\id} \times \Delta)_! \pr_1^* (\blank) \\
    &\eq (\pr_2)_! (\Delta \times \id)^* ({\id} \times f \times \id)_! \bigl( ({\id} \times \Delta)_! \pr_1^* (\blank) \tensor \DS{{\id} \times f \times \id} \bigr) \\
    &\eq (\pr_2)_! (\Delta \times \id)^* ({\id} \times f \times \id)_! ({\id} \times \Delta)_! \bigl( \pr_1^*(\blank) \tensor ({\id} \times \Delta)^* \pr_2^*(\DS f) \bigr) \\
    &\eq (\pr_2)_! (f, \id)_! (f, \id)^* \bigl( \pr_1^*(\blank) \tensor \pr_2^*(\DS f) \bigr) \\
    &\eq (f, \id)^* \pr_1^*(\blank) \tensor (f, \id)^* \pr_2^*(\DS f) \\
    &\eq f^*(\blank) \tensor \DS{f}
  \end{align*}
  where we use \cref{lemma:f_*}, the projection formula, the pull--push formula applied to the two pullback squares
  \[
  \begin{tikzcd}
    X \rar{\Delta} \dar[swap]{(f, \id)} & X \times_B X \dar[swap]{(f, \id) \times \id} \rar{f \times \id} &[15] Y \times_B X \dar{\Delta \times \id} \\
    Y \times_B X \rar{{\id} \times \Delta} & Y \times_B X \times_B X \rar{{\id} \times f \times \id} & Y \times_B Y \times_B X
  \end{tikzcd}
  \]
  as well as \cref{lemma:DS_pullback}.
  
  By construction, the upper left-hand square of \eqref{eq:transfer_diag_pre} commutes via the pasting
  \[
  \begin{tikzcd}
    \Sp[B]_B \rar{Y^*} \dar[swap]{X^*} &[20] \Sp[Y]_B \rar{\Delta_!} \dar{({\id} \times X)^*} \drar[phantom, start anchor = center, end anchor = center]{\textnormal{(a)}} &[20] \Sp[Y \times_B Y]_B  \dar{({\id} \times {\id} \times X)^*} \\
    \Sp[X]_B \rar{(Y \times \id)^*} \dar[swap]{\Delta_!} \drar[phantom, start anchor = center, end anchor = center]{\textnormal{(a)}} & \Sp[Y \times_B X]_B \rar{(\Delta \times \id)_!} \dar{({\id} \times \Delta)_!} & \Sp[Y \times_B Y \times_B X]_B \dar{({\id} \times {\id} \times \Delta)_!} \\
    \Sp[X \times_B X]_B \dar[swap]{(f \times \id)_*} \rar{(Y \times {\id} \times \id)^*} \drar[phantom, start anchor = center, end anchor = center]{\textnormal{(a)}} & \Sp[Y \times_B X \times_B X]_B \rar{(\Delta \times {\id} \times \id)_!} \dar{({\id} \times f \times \id)_*} \drar[phantom, start anchor = center, end anchor = center]{\textnormal{(b)}} & \Sp[Y \times_B Y \times_B X \times_B X]_B \dar{({\id} \times {\id} \times f \times \id)_*} \\
    \Sp[Y \times_B X]_B \drar[swap]{\id} \rar{(Y \times {\id} \times \id)^*} & \Sp[Y \times_B Y \times_B X]_B \rar{(\Delta \times {\id} \times \id)_!} \dar{(\Delta \times \id)^*} \drar[phantom, start anchor = center, end anchor = center]{\textnormal{(a)}} & \Sp[Y \times_B Y \times_B Y \times_B X]_B \dar{({\id} \times \Delta \times \id)^*} \\
    & \Sp[Y \times_B X]_B \drar[swap]{\id} \rar{(\Delta \times \id)_!} & \Sp[Y \times_B Y \times_B X]_B \dar{({\id} \times Y \times \id)_!} \\
    & & \Sp[Y \times_B X]_B 
  \end{tikzcd}
  \]
  where the unlabeled squares and triangles commute via the transformations obtained from the functoriality of $\Sp[\blank]_B$ (and taking adjoints), the squares labeled (a) commute via the mates of these, and the square (b) in turn commutes via the mate of that mate.
  To see that these are the correct transformations, we note that, for an adjunction $l \dashv r$ and a morphism $g$ in a monoidal $2$-category, the canonical transformations making the following two squares commute
  \[
  \begin{tikzcd}
    C \tensor D \rar{l \tensor \id} \dar[swap]{{\id \tensor g}} & C' \tensor D \dar{{\id \tensor g}} & & C \tensor D \dar[swap]{{\id \tensor g}} & \lar[swap]{r \tensor \id} C' \tensor D \dar{{\id \tensor g}} \\
    C \tensor D' \rar{l \tensor \id} & C' \tensor D' & & C \tensor D' & \lar[swap]{r \tensor \id} C' \tensor D'
  \end{tikzcd}
  \]
  are mates of each other.
  Similar arguments yield pastings making the other two squares of \eqref{eq:transfer_diag_pre} commute.
  The claims about the transformation making the squares of \eqref{eq:transfer_diag} commute then follow from diagram chases (using \cref{lemma:norm_pull-push,lemma:projection_pullback} and that mates compose).
\end{proof}

To conclude this section, we provide explicit descriptions for fiberwise versions of constructions related to THH.
We begin with the equivalence $\THH(X) \eq \SS[\L X]$, adapting the proof of \cite[Theorem~4.40]{CCRY}.

\begin{definition}
  Let $\cat C$ be an $\infty$-category that admits finite limits.
  We write $\L \colon \cat C \to \cat C$ for the \emph{free loop space} functor, defined to be the pullback
  \begin{equation} \label{eq:L}
  \begin{tikzcd}
    \L X \rar{e} \dar[swap]{e} & X \dar{\Delta} \\
    X \rar{\Delta} & X \times X
  \end{tikzcd}
  \end{equation}
  where $\Delta \colon X \to X \times X$ is the diagonal.
  We write $L_B$ for the free loop space functor of $\Anover{B}$ (note that in this case, the product $X \times X$ above is a fiber product over $B$).
\end{definition}

\begin{lemma} \label{lemma:THH_loops}
  Let $B$ be an anima.
  Then the composite
  \[ \Anover{B}  \xlongto{\L_B}  \Anover{B}  \xlongto{\SS_B[\blank]}  \Sp[B] \]
  is equivalent to $\THH_B$, contravariantly natural in $B$.
  For $X \in \Anover{B}$, this equivalence is given by
  \[ \THH_B(X)  \eq  X_! \Delta^* \Delta_! X^* (\SS_B)  \eq  X_! e_! e^* X^* (\SS_B)  \eq  (\L_B X)_! (\L_B X)^* (\SS_B)  =  \SS_B[\L_B X] \]
  where we use \cref{lemma:transfer_diag} and the pull--push formula associated to the pullback square \eqref{eq:L}.
\end{lemma}

\begin{proof}
  Recall from the proof of \cref{lemma:transfer_diag} that $\Sp[\blank]_B \colon \Anover{B} \to \Fun(B, \catMod{\Sp})$ is a symmetric monoidal bivariant theory.
  By a result of Macpherson \cite[Theorem~4.4.6]{Macp} (see also \cite[Proposition~3.6]{CCRY}), this implies that it extends uniquely to a strong symmetric monoidal functor $\Span(\Anover{B}) \to \Fun(B, \catMod{\Sp})$ out of the $(\infty, 2)$-category of spans in $\Anover{B}$.
  This yields the following commutative diagram
  \[
  \begin{tikzcd}
    \Anover{B} \rar \dar[swap]{\L_B} & \dbl { \Span ( \Anover{B} ) } \dar[swap]{\dim} \rar & \dbl { \Fun(B, \catMod{\Sp}) } \dar{\dim} \rar{\eq} & \Fun(B, \dbl {\catMod{\Sp}}) \dar{\Dim_*} \\
    \Anover{B} \rar{\eq} & \Loops \Span ( \Anover{B} ) \rar & \Loops \Fun(B, \catMod{\Sp}) \rar{\eq} & \Fun(B, \Sp)
  \end{tikzcd}
  \]
  where the left-hand square is the one of \cite[Theorem~3.19]{CCRY}, the middle square commutes by naturality of the dimension, and the right-hand square commutes by \cref{lemma:parametrized_trace}.
  The whole diagram is (contravariantly) functorial in $B$.
  The bottom composite $\Anover{B} \to \Sp[B]$ is, by \cite[Proof of Theorem~4.40]{CCRY}, equivalent to $\SS[\blank] \colon \An \to \Sp$ for $B = *$; hence it is equivalent to $\SS_B[\blank]$ by \cref{obs:assembly}.
  Since the upper composite is $\THH_B$, this proves that $\THH_B \eq \SS_B \after \L_B$ naturally in $B$.
  Chasing through the proof of \cite[Theorem~3.19]{CCRY}, the explicit pointwise formula follows from \cite[Proof of Lemma~3.11]{CCRY}.
\end{proof}

Note that, since the functor $\Fun(\blank, \Sp) \colon \opcat{\An} \to \Catinf$ preserves limits, \cref{obs:assembly} implies that the equivalence $\THH_B \eq \SS_B[\blank] \after \L_B$ is equivalent to using the equivalence $\THH \eq \SS[\blank] \after \L$ pointwise.
Since the functor $\SS[\blank] \colon \An \to \Sp$ preserves colimits, \cref{obs:assembly} also yields the following.

\begin{definition} \label{def:THH_assembly}
  The \emph{assembly map} is the unique natural transformation $\alpha \colon \SS[X] \to \THH(X)$ that restricts to the equivalence $\SS \eq \THH(*)$ for $X = *$.
  For an anima $B$ and $X \in \Anover{B}$, the \emph{fiberwise assembly map} is the map $\SS_B[X] \to \THH_B(X)$ given by applying $\alpha$ to $X \colon B \to \An$.
\end{definition}

The following lemma provides an explicit description of the fiberwise assembly map using the description of THH resulting from \cref{lemma:transfer_diag}.

\begin{lemma} \label{lemma:fiberwise_THH_assembly}
  Let $B$ be an anima and $X \in \overcat \An B$.
  Then the fiberwise assembly map $\SS_B[X] \to \THH_B(X)$ is equivalent to the map
  \[ \SS_B[X]  =  X_! X^* (\SS_B)  \xlongto{\eta}  X_! \Delta^* \Delta_! X^* (\SS_B)  \eq  \THH_B(X) \]
  given by the unit of the adjunction $\Delta_! \dashv \Delta^*$ associated to the diagonal map $\Delta \colon X \to X \times_B X$.
\end{lemma}

\begin{proof}
  Combining the assembly map with \cref{lemma:THH_loops}, we obtain a natural composite map $\beta \colon \SS[Y] \to \THH(Y) \eq \SS[\L Y]$ that is the identity for $Y = *$.
  The commutative diagram
  \[
  \begin{tikzcd}
    Y \rar{\id} \dar[swap]{\id} & Y \dar{\Delta} \\
    Y \rar{\Delta} & Y \times Y
  \end{tikzcd}
  \]
  induces a natural map $c \colon Y \to \L Y$ (which is the inclusion of the constant loops).
  The induced natural map $c_* \colon \SS[Y] \to \SS[\L Y]$ is also the identity for $Y = *$ and hence is equivalent to $\beta$ as natural transformations by \cref{obs:assembly}.
  This implies that the composite
  \[ \SS_B[X]  \longto  \THH_B(X)  \eq  \SS_B[\L_B X] \]
  is equivalent to the map induced by the analogously defined map $c \colon X \to \L_B X$.
  Using \cref{lemma:susp_map}, we deduce the claim via a diagram chase using the commutative diagram
  \[
  \begin{tikzcd}
    X \rar{c} \dar[swap]{\id} & \L_B X \rar{e} \dar[swap]{e} & X \dar{\Delta} \\
    X \rar{\id} & X \rar{\Delta} & X \times_B X
  \end{tikzcd}
  \]
  where the right-hand square is the pullback square defining $\L_B X$.
\end{proof}

\section{A rational model for the fiberwise THH transfer}

In this section, we construct a rational model for fiberwise THH and the fiberwise THH transfer.
This constitutes the main part of the paper.
Our overall strategy is to model our description of the fiberwise THH transfer from \cref{lemma:transfer_diag}.
To this end, we require rational models for parametrized spectra and various related constructions.
Thus we begin by recalling the former, and will then move on to providing the latter, before finally proving that the fiberwise THH transfer is modeled by the Hochschild homology transfer.

\subsection{Rational models for parametrized spectra}

In this subsection we recall work of Braunack-Mayer \cite{Bra}, who showed that there is, for a nilpotent anima $X$ modeled by a cofibrant cdga $R$, an equivalence between the rational homotopy theory of (certain) nilpotent $X$-spectra (cf.\ \cref{def:nilpotent_spectrum}) and the homotopy theory of (certain) $R$-modules.

\begin{definition}
  Let $X$ be an anima.
  We say that an $X$-spectrum $E$ is of \emph{finite rational type} if $\pi_n(E, x) \tensor \QQ$ is finite-dimensional for all $n \in \ZZ$ and $x \in X$.
  We say that an $X$-spectrum $E$ is \emph{bounded below} if there exists an integer $N$ such that $\pi_*(E, x)$ is concentrated in degrees $\ge N$ for all $x \in X$.
  We write $\Sp[X]_\ft$ and $\Sp[X]_\bbl$ for the full subcategories of $\Sp[X]$ spanned by the $X$-spectra that are of finite rational type and bounded below, respectively.
\end{definition}

\begin{definition}
  Let $R$ be a cdga.
  We say that an $R$-module $M$ is \emph{of finite homotopical type} if it admits a minimal model $R \tensor V$ such that $V$ is finite-dimensional in each degree; we say that $M$ is \emph{homotopically finite} if it admits a minimal model such that $V$ is finite-dimensional.
  We say that an $R$-module $M$ is \emph{bounded below} if its underlying cochain complex is bounded below.
  We write $(\Mod{R})_\fht$ and $(\Mod{R})_\bbl$ for the full subcategories of $\Mod{R}$ spanned by the $R$-modules that are of finite homotopical type and bounded below, respectively.
\end{definition}

Our definition of finite homotopical type is slightly weaker than the one of Braunack-Mayer \cite[Definition~4.18]{Bra}, who in addition requires $M$ to be bounded below.
However, the following observation allows us to ignore the distinction.

\begin{observation}
  Assume that the cdga $R$ is bounded below.
  Then the inclusion of the full subcategory $\Hocat(\Mod{R})_{\fht,\bbl} \subset \Hocat(\Mod{R})_\fht$, spanned by the bounded below $R$-modules of finite homotopical type, is an equivalence of categories since any minimal model is bounded below.
\end{observation}

The following result is essentially due to Braunack--Mayer \cite{Bra}.
However, he only proves that his equivalence is compatible with the monoidal structures on the level of homotopy categories (and in a weak sense).
In the proof below, we will employ an $\infty$-categorical trick to obtain the fully coherent compatibility.

\begin{proposition}[Braunack--Mayer] \label{prop:BM}
  Let $X$ be a nilpotent anima that is modeled by a homologically connected cofibrant cdga $R$ of finite homotopical type.
  Then there is a right adjoint functor $\Psi_R \colon \opcat{\cUnderlying(\Mod{R})} \to \Sp[X]$ whose restriction lifts to an equivalence
  \[ \Psi^\tensor_R  \colon  \opcat{\monUnderlying(\Mod{R})_\fht}  \xlongto{\eq}  \Sp[X]^\QQ_{\nil,\ft,\bbl} \]
  of symmetric monoidal $\infty$-categories.
  Both $\Psi_R$ and $\Psi^\tensor_R$ are moreover natural in the pair $(X, R)$: contravariant in $X$ and covariant in $R$.
  In particular, given a map $f \colon X \to Y$ of nilpotent animas modeled by a map $\phi \colon R \to S$ of homologically connected cofibrant cdgas of finite homotopical type, there is a canonical homotopy $\Psi_S \after \Lder \phi_! \eq f^* \after \Psi_R$.
\end{proposition}

\begin{proof}
  By (the adjoint of) \cite[Remark~4.7]{Bra}, there is a pseudonatural transformation $\RealMod_R \colon \opcat{(\Mod{R})} \to \Sp^{\mathbb{N}}_{\Real(R)}$ of pseudofunctors $\CDGA \to \ModCatR$, where $\Sp^{\mathbb{N}}_{X}$ is the model category of sequential spectra in retractive spaces over the simplicial set $X$.
  (Using a result of Lack--Paoli \cite[Theorem~5.4]{LP}, we will henceforth implicitly replace $\Sp^{\mathbb{N}}_{(\blank)}$ and $\Mod{(\blank)}$ by \emph{normal} pseudofunctors, i.e.\ make them strictly unital; this allows us to apply the Duskin nerve to obtain $\infty$-functors.)
  By \cite[Theorem~2.40]{Bra21}, there is furthermore a natural equivalence $\fUnderlying ( \Sp^{\mathbb{N}}_{X} ) \eq \Sp[u(X)]$ of functors $\opcat{\sSet} \to \Catinf$, where $u \colon \sSet \to \An$ is the localization.
  Combining these, we obtain natural transformations of functors $\CDGA \to \Catinf$
  \[ \opcat{\cUnderlying(\Mod{R})}  \xlongfrom{\eq}  \opcat{\cUnderlying(\Mod{q(R)})}  \longto  \Sp[u \Real q(R)] \]
  where $q \colon \CDGA \to \CDGA$ is a cofibrant replacement functor.
  Since quasi-isomorphisms of cdgas induce Quillen equivalences of module categories by a result of Schwede--Shipley \cite[Theorem~4.3]{SS} (combined with \cref{lemma:cofibrant_flat}), and Quillen equivalences induce equivalences of the underlying $\infty$-categories, the left-hand map above is an equivalence and we furthermore obtain a natural transformation
  \[ \Rder \RealMod_R  \colon  \opcat{\cUnderlying(\Mod{R})}  \longto  \Sp[\Rder \Real(R)] \]
  of functors $\Underlying(\CDGA) \to \Catinf$.
  Precomposing with the functor $\Lder \APL \colon \An \to \opcat{\Underlying(\CDGA)}$, we obtain a natural transformation of functors $\opcat{\An} \to \Catinf$
  \[ \Psi_X  \colon  \opcat{\cUnderlying(\Mod{\Lder \APL(X)})}  \longto  \Sp[\Rder \Real \Lder \APL(X)]  \xlongto{\eta_X^*}  \Sp[X] \]
  where $\eta_X$ is the unit of the adjunction $\Lder \APL \dashv \Rder \Real$ at $X$.
  
  We now want to prove that, when the domain is restricted to $\opcat{\cUnderlying(\Mod{\Lder \APL(X)})_\fht}$, the natural transformation $\rat{(\blank)} \after \Psi$ can be lifted to the $\infty$-category $\Catinfmon$ of symmetric monoidal $\infty$-categories and (strong) symmetric monoidal functors.
  First note that $\cUnderlying(\Mod{\Lder \APL(X)})_\fht$ is indeed closed under the (derived) tensor product by \cite[Lemma~5.16]{Bra} and is clearly preserved under (derived) scalar extension.
  We now begin by considering the case $X = *$.
  We have $\Lder \APL (*) \eq \QQ$ and $\Rder \Real (\QQ) \eq *$.
  Furthermore, by \cite[Theorem~7.1.2.13]{LurHA}, there is an equivalence $\monUnderlying(\Mod{\QQ}) \eq \Mod{\H \QQ}$ of symmetric monoidal $\infty$-categories; by construction, taking homology groups on the left-hand side corresponds to taking homotopy groups on the right-hand side (here we think of $\Mod{\QQ}$ as \emph{homologically} graded).
  In particular note that $\cUnderlying(\Mod{\QQ})$ and its opposite are stable.
  By \cite[Lemma~4.4]{Bra}, the left adjoint $\Phi_*$ of $\Psi_*$ maps $\SS$ to $\QQ$.
  This property determines $\Phi_*$ up to equivalence by \cref{cor:Sp}.
  By \cite[Remark~5.2.5.8 and Construction~5.2.1.9]{LurHA}, the internal hom-functor
  \[ D \defeq \Hom(\blank, \H \QQ)  \colon  \opcat{(\Mod{\H \QQ})}  \longto  \Mod{\H \QQ} \]
  is right adjoint to its opposite $\opcat{D}$.
  In particular the composite
  \[ \Sp  \xlongto{\blank \tensor \H \QQ}  \Mod{\H \QQ}  \xlongto{\opcat{D}}  \opcat{(\Mod{\H \QQ})} \]
  is left adjoint and sends $\SS$ to $\H \QQ$.
  Thus it is equivalent to $\Phi_*$, and its right adjoint
  \[ \opcat{(\Mod{\H \QQ})}  \xlongto{D}  \Mod{\H \QQ}  \longto  \Sp \]
  is equivalent to $\Psi_*$.
  By \cite[Remark~5.2.2.25, Construction~5.2.5.27, and Remark~5.2.5.8]{LurHA} the functor $D$ is lax symmetric monoidal.
  We claim that its restriction to $\opcat{(\Mod{\H \QQ})_\fht}$ is strong symmetric monoidal.
  To see this, note that a $\QQ$-module is of finite homotopical type if and only if its homology is bounded above (in homological grading) and finite-dimensional in each degree; the claim follows by checking it on homotopy groups.%
  \footnote{Note that for $M, N \in \Mod{\H \QQ}$, the maps $\pi_*(M) \tensor \pi_*(N) \to \pi_*(M \tensor N)$ and $\pi_*(D(M)) \to \dual{\pi_*(M)}$ are isomorphisms: for a fixed $N$, the collection of those $M$ for which the maps are isomorphisms contains $\H \QQ$ and is closed under taking retracts, (co)fibers, and filtered colimits; by \cite[Proposition~7.2.4.2]{LurHA} this collections thus contains all objects of $\Mod{\H \QQ}$.}
  The composite $\Mod{\H \QQ} \to \Sp \to \Sp^\QQ$ can be lifted to a symmetric monoidal equivalence by \cite[Proposition~4.8.2.10]{LurHA}, and thus the restriction of $\rat{(\blank)} \after \Psi_*$ to $\opcat{(\Mod{\H \QQ})_\fht}$ lifts to a strong symmetric monoidal functor.
  Now, for any functor $F \colon \opcat{\An} \to \Catinfmon$, there is a commutative diagram
  \[
  \begin{tikzcd}
    \Map_{\Fun(\opcat{\An}, \Catinfmon)} \bigl( F, \Sp[\blank]^\QQ \bigr) \rar \dar[swap]{\eq} & \Map_{\Fun(\opcat{\An}, \Catinf)} \bigl( F, \Sp[\blank]^\QQ \bigr) \dar{\eq} \\
    \Map_{\Catinfmon} \bigl( F(*), \Sp^\QQ \bigr) \rar & \Map_{\Catinf} \bigl( F(*), \Sp^\QQ \bigr)
  \end{tikzcd}
  \]
  where the vertical maps are equivalences by \cref{obs:assembly} since $\Sp[\blank]^\QQ = \Fun(\blank, \Sp^\QQ)$ preserves limits, both as a functor to $\Catinfmon$ and $\Catinf$.
  Applying this to $F(X) = \opcat{\monUnderlying(\Mod{\Lder \APL(X)})_\fht}$, we obtain a lift of $\rat{(\blank)} \after \Psi$ along the upper horizontal map since $\rat{(\blank)} \after \Psi_*$ admits a lift along the bottom horizontal map.
  
  Lastly we note that, by \cite[Theorem~4.20]{Bra}, the functor $\rat{(\blank)} \after \Rder \RealMod_R$ restricts to an equivalence between $\opcat{\cUnderlying(\Mod{R})_\fht}$ and $\Sp[\Rder \Real(R)]^\QQ_{\nil,\ft,\bbl}$ when $R$ is cofibrant.
  Thus, by \cref{lemma:param_Sp_rational_eq}, the analogous statement is true for $\Psi_X$ when $X$ is nilpotent and of finite rational type.
  Observing that the full subcategory $\Sp[X]^\QQ_{\nil,\ft,\bbl} \subset \Sp[X]^\QQ$ is closed under tensor products by \cite[Lemma~5.15]{Bra}, and clearly preserved under restriction along maps of animas, completes the proof.
\end{proof}

Braunack--Mayer \cite[Proposition~5.3]{Bra} also provides models for the functors $\Suspinf[X]$ and $\Loopsinf[X]$.
In the following we prove a fully coherent version of this result using a different argument.

\begin{notation}
  Let $R$ be a non-negatively graded cdga, $M$ an $R$-module, and $n \in \ZZ$.
  We write $\trunc{< n} M \subseteq M$ for the smallest $R$-submodule (and thus in particular subcomplex) that contains $M^i$ for all $i < n$, and denote the cokernel of this inclusion by $\trunc{\ge n} M$.
  Note that $\trunc{\ge n}$ is the left adjoint of the inclusion $\Mod{R}^{\ge n} \subseteq \Mod{R}$ of the full subcategory spanned by those $R$-modules whose underlying cochain complex is concentrated in degrees $\ge n$.
  We furthermore write
  \[ \SAcn_R  \colon  \Mod{R}  \xlongto{\trunc{\ge 0}}  \Mod{R}^{\ge 0}  \xlongto{\SA_R}  \undercat{(\CDGA)}{R} \]
  where $\SA_R M$ denotes the free commutative $R$-algebra on $M$.
\end{notation}

Note that $\SAcn_R$ is left adjoint to the forgetful functor $\undercat{(\CDGA)}{R} \to \Mod{R}$.
Since the latter preserves fibrations and weak equivalences, this adjunction is Quillen.

\begin{lemma}[Braunack--Mayer] \label{lemma:Loops_model}
  In the situation of \cref{prop:BM}, the following diagram of $\infty$-categories commutes naturally in the pair $(X, R)$
  \[
  \begin{tikzcd}[column sep = 45]
    \opcat{\cUnderlying(\Mod{R})} \rar{\cUnderlying(\SAcn_R)} \dar[swap]{\Psi_R} & \opcat{\cUnderlying \bigl( \undercat{(\CDGA)}{R} \bigr)} \rar{\eq} & \opcat{\overcat{\Underlying(\CDGA)}{R}} \dar{\Rder \Real} \\
    \Sp[X] \rar{\Loopsinf[X]} & \Anover{X} & \lar[swap]{\eta^*} \Anover{\Rder \Real (R)}
  \end{tikzcd}
  \]
  where $\eta \colon X \to \Rder \Real \Lder \APL (X) \eq \Rder \Real (R)$ is the unit of the adjunction $\Lder \APL \dashv \Rder \Real$.
\end{lemma}

\begin{proof}
  First note that the assignment $R \mapsto \undercat{(\CDGA)}{R}$ extends to a pseudofunctor $\CDGA \to \ModCatL$ such that a map $R \to S$ is sent to $\blank \tensor_R S$ (using \cite[Theorem~5.4]{LP}, we again implicitly replace it by a normal pseudofunctor) and that the functors $\SAcn_R \colon \Mod{R} \to \undercat{(\CDGA)}{R}$ assemble into a pseudonatural transformation.
  Since a quasi-isomorphism of cdgas induces a Quillen equivalence of categories of algebras by \cite[Theorem~4.4]{SS} (combined with \cref{lemma:cofibrant_flat}), we thus obtain $\cUnderlying(\SAcn_R)$ as a natural transformation of functors $\Underlying(\CDGA) \to \Catinf$.
  Furthermore note that, for a cofibrant object $C$ of a model category $\cat M$, the canonical functor $\undercat{\cat M}{C} \to \undercat{\Underlying(\cat M)}{C}$ induces an equivalence $\Underlying(\undercat{\cat M}{C}) \eq \undercat{\Underlying(\cat M)}{C}$ (see e.g.\ \cite[Corollary~7.6.13]{Cis}).
  It is natural with respect to the contravariant functoriality of the under category (i.e.\ restriction) and thus also natural when passing to the covariant functoriality given by the left adjoints (i.e.\ taking pushouts).
  In particular all maps in the diagram above are indeed natural in the pair $(X, R)$, and it is enough to prove that it commutes naturally for pairs of the form $(X, \Lder \APL(X))$.
  
  Since the functor $\opcat{\An} \to \Catinf$ given by $X \mapsto \Anover X \eq \Fun(X, \An)$ preserves limits, it is by \cref{obs:assembly} enough to prove the claim for the pair $(*, \QQ)$, i.e.\ that the diagram
  \[
  \begin{tikzcd}[column sep = 45]
    \opcat{\cUnderlying(\Mod{\QQ})} \rar{\cUnderlying(\SAcn_{\QQ})} \dar[swap]{\Psi_\QQ} & \opcat{\cUnderlying(\CDGA)} \dar{\Rder \Real} \\
    \Sp \rar{\Loopsinf} & \An
  \end{tikzcd}
  \]
  commutes.
  Passing to left adjoints and again using \cref{obs:assembly}, it is enough to prove that the left adjoint of $\Psi_\QQ$ maps $\SS$ to $\QQ$.
  This follows from \cite[Lemma~4.4]{Bra}.
\end{proof}

\subsection{Equivariant models} \label{sec:eq}

We would like to also be able to model animas that are not necessarily nilpotent.
To this end, we observe that for a group $G$ there is an equivalence between $\Fun(\B G, \An_\nil)$ and the full subcategory of $\Anover{\B G}$ spanned by those maps $\alpha \colon X \to \B G$ whose fiber is nilpotent; note that $X$ itself does \emph{not} need to be nilpotent.
Furthermore $\Fun(\B G, \An^\QQ_{\nil,\ft})$ can be modeled by an appropriate category of equivariant cdgas.
This allows to model non-nilpotent connected animas $X$, for example using the canonical map $X \to \B \pi_1(X)$ (however note that there exist maps of animas that cannot be realized as a map of fiberwise nilpotent animas over $\B G$ for any fixed $G$).
We will now make this precise; also see \cite{BS,GHT,BZ} for further elaborations on these ideas.

\begin{definition}
  Let $G$ be a group.
  A \emph{$G$-equivariant cdga} is a cdga equipped with a left $G$-action.
  We denote by $\CDGAeq{G} \defeq \Fun(\B G, \CDGA)$ the category of non-negatively graded $G$-equivariant cdgas, equipped with the injective model structure, i.e.\ the weak equivalences are the quasi-isomorphisms and the cofibrations are the cofibrations of the underlying cdgas.\footnote{This injective model structure exists by \cite[Proposition~A.2.8.2]{LurHTT}.}
  We write $\CDGAeq[\ge 1, \fht]{G} \subseteq \CDGAeq{G}$ for the full subcategory spanned by those $G$-equivariant cdgas whose underlying cdga lies in $\CDGA[\ge 1, \fht]$.
\end{definition}

It follows from \cite[Proposition~1.3.4.25]{LurHA} that there is a canonical equivalence
\[ \Underlying \bigl( \CDGAeq{G} \bigr) \eq \Fun \bigl( \B G, \Underlying(\CDGA) \bigr) \]
that is compatible with restriction along maps of groups $H \to G$.
In particular, the equivalence is compatible with the respective restriction functors to $\Underlying(\CDGA)$.
Combined with \cref{prop:Sullivan}, this implies that there is an adjunction, which we denote $\APL^G \dashv \Real^G$,
\[
\begin{tikzcd}
  \Anover{\B G} \eq \Fun(\B G, \An) \rar[yshift = 7]{(\Lder \APL)_*}[below, name = T]{} &[20] \lar[yshift = -7]{(\Rder \Real)_*}[above, name = B]{} \Fun \bigl( \B G, \opcat{\Underlying(\CDGA)} \bigr) \eq \opcat{ \Underlying \bigl( \CDGAeq{G} \bigr) }
  \ar[from = T, to = B, phantom, "\vertdashv"]
\end{tikzcd}
\]
(here we implicitly use the canonical equivalence $\opcat{(\B G)} \eq \B G$).
It restricts to an adjoint equivalence
\[ \Anover[\QQ,\fnil,\ft]{\B G}  \eq  \Fun \bigl( \B G, \An^\QQ_{\nil,\ft} \bigr)  \eq  \opcat{ \Underlying \bigl( \CDGAeq{G} \bigr)_{\ge 1, \fht} } \]
(see \cref{def:fiberwise} for the notation).
In particular, by \cite[Proposition~1.3.4.23]{LurHA}, homotopy colimits in the model category $\CDGAeq{G}$, as long as they are again homologically connected and of finite homotopical type, represent limits in $\Anover[\QQ,\fnil,\ft]{\B G}$.
Note that the forgetful functor $\Anover{\B G} \to \An$ preserves weakly contractible limits by \cite[Lemma~2.2.7]{GHK} (in particular it preserves pullbacks).
The following are generalizations of \cref{def:model,lemma:pullback_model}.

\begin{definition}
  Let $G$ be a group and $\cat I$ a category.
  We say that a diagram $A \colon \cat I \to \opcat{(\CDGAeq[\ge 1, \fht]{G})}$ \emph{models} a diagram $X \colon \cat I \to \Anover[\fnil,\ft]{\B G}$ if it comes equipped with an equivalence $\Real^G \after u \after A \eq \rat X$ of functors $\cat I \to \Anover[\QQ,\fnil,\ft]{\B G}$, where $u \colon \opcat{(\CDGAeq[\ge 1, \fht]{G})} \to \opcat{\Underlying(\CDGAeq{G})_{\ge 1, \fht}}$ is the canonical functor.
\end{definition}

\begin{observation} \label{obs:pullback_model_eq}
  Let the following be a pullback diagram in $\Anover{\B G}$ and a homotopy pushout diagram in $\CDGAeq{G}$, respectively,
  \[
  \begin{tikzcd}
    X \times_Z Y \rar \dar & Y \dar{g}  &  P & \lar B \\
    X \rar{f} & Z  &  A \uar & \lar C \uar
  \end{tikzcd}
  \]
  such that the maps $X \to Z \from Y$ lie in $\Anover[\fnil,\ft]{\B G}$ and are modeled by the maps $A \from C \to B$, which lie in $\CDGAeq[\ge 1, \fht]{G}$.
  If $f$ or $g$ has connected fibers, then, by (the proof of) \cref{lemma:pullback_model}, the left-hand square is canonically modeled by the right-hand square.
\end{observation}

\begin{remark} \label{rem:eq_models}
  In this paper, we will be mostly concerned with maps of connected animas $f \colon X \to Y$ whose fiber is simply connected; in particular $f$ induces an isomorphism on $\pi_1$.
  If we furthermore assume that the universal coverings of $X$ and $Y$ are of finite rational type, then we can consider $f$ as a map of fiberwise simply connected animas over $\B \pi_1(Y)$ of fiberwise finite rational type.
  In this situation, the map $f$ thus always admits a model in $\pi_1(Y)$-equivariant cdgas.
\end{remark}

\subsubsection*{Equivariant models for parametrized spectra}

Given a $G$-equivariant cdga that models an anima $X$ over $\B G$, we would also like to be able to model $X$-spectra.
This is achieved by a notion of $G$-equivariant modules, which we now define.

\begin{definition}
  Let $G$ be a group and $R$ a $G$-equivariant cdga.
  A \emph{$G$-equivariant $R$-module} is a module $M$ over the underlying cdga of $R$ together with a left $G$-action on the underlying cochain complex of $M$ such that the module structure map $R \tensor M \to M$ is $G$-equivariant (where $G$ acts diagonally on the tensor product).
  A \emph{map of $G$-equivariant $R$-modules} is a map of the underlying $R$-modules that is $G$-equivariant.
  We denote by $\Modeq{G}{R}$ the category of $G$-equivariant $R$-modules, equipped with the injective model structure, i.e.\ the weak equivalences are the quasi-isomorphisms and the cofibrations are the cofibrations of $R$-modules.%
  \footnote{This injective model structure exists by \cite[Theorem~2.30]{Bar}, noting that $\Modeq{G}{R}$ is equivalent to the category of left sections of the left Quillen presheaf $\B G \to \twoCat$ given by sending the unique object to $\Mod{R}$ and a morphism $g \in G$ to scalar extension along acting by $g$.}
  We write $(\Modeq{G}{R})_\fht \subseteq \Modeq{G}{R}$ the full subcategory spanned by those $G$-equivariant $R$-modules whose underlying $R$-module is of finite homotopical type.
\end{definition}

Equivalently, a $G$-equivariant cdga is a commutative monoid in the category of $G$-equivariant cochain complexes equipped with the symmetric monoidal structure given by the tensor product with the diagonal $G$-action, and a $G$-equivariant module is a module over such an algebra.

We now provide an equivariant version of Braunack--Meyer's models for parametrized spectra.
To do this, we need the fully coherent version of his result we obtained in \cref{prop:BM}.

\begin{notation}
  Let $G$ be a group and $\alpha \colon X \to \B G$ an anima over $\B G$.
  We denote by $\Sp[X]^{G,\QQ}_{\nil,\ft,\bbl} \subseteq \Sp[X]$ the full subcategory spanned by those $X$-spectra whose restriction to the fiber of $\alpha$ lies in $\Sp[\fib \alpha]^\QQ_{\nil,\ft,\bbl}$.
\end{notation}

\begin{proposition} \label{prop:BMeq}
  Let $G$ be a group and $R$ a $G$-equivariant homologically connected cofibrant cdga of finite homotopical type that models a fiberwise nilpotent anima $X$ over $\B G$.
  Then there is a right adjoint functor $\Psi_R \colon \opcat{\cUnderlying(\Modeq{G}{R})} \to \Sp[X]$ whose restriction lifts to an equivalence
  \[ \Psi^\tensor_R  \colon  \opcat{\monUnderlying(\Modeq{G}{R})_\fht}  \xlongto{\eq}  \Sp[X]^{G,\QQ}_{\nil,\ft,\bbl} \]
  of symmetric monoidal $\infty$-categories.
  Both $\Psi_R$ and $\Psi^\tensor_R$ are moreover natural in the pair $(X, R)$: contravariant in $X$ and covariant in $R$.
  In particular, given a map $f \colon X \to Y$ of fiberwise nilpotent animas over $\B G$ modeled by a map $\phi \colon R \to S$ of $G$-equivariant homologically connected cofibrant cdgas of finite homotopical type, there is a canonical homotopy $\Psi_S \after \Lder \phi_! \eq f^* \after \Psi_R$.
\end{proposition}

\begin{proof}
  First choose a strictification $M$ of the pseudofunctor $\CDGA \to \ModCatmon$ given by $A \mapsto \Mod{A}$ and scalar extension.
  Note that this induces a pseudonatural equivalence $\monUnderlying(M(A)) \eq \monUnderlying(\Mod{A})$.
  Restricting along $R \colon \B G \to \CDGA$, we furthermore have $\Sect{\B G}(M) \eq \Modeq{G}{R}$ (cf.\ \cref{def:Sect}) and hence it follows from \cref{cor:sections_monoidal} that there is a natural equivalence of symmetric monoidal $\infty$-categories
  \[ \monUnderlying(\Modeq G R) \eq \lim_{\B G} \monUnderlying(\Mod{R}) \]
  that is compatible with the forgetful maps to $\cUnderlying(\Mod{R})$.
  Using \cref{prop:BM}, we thus obtain a natural right adjoint functor
  \[ \opcat{\cUnderlying(\Modeq G R)}  \eq  \lim_{\B G} \opcat{\cUnderlying(\Mod{R})}  \longto  \lim_{\B G} \Sp[F(\blank)]  \eq  \Sp[X] \]
  where $F \colon \B G \to \An$ is the functor corresponding to $X \in \Anover{\B G}$ (note that $\Sp[\blank] = \Fun(\blank, \Sp)$ sends colimits to limits).
  Restricting to the appropriate full subcategories, it moreover lifts to a natural equivalence
  \[ \opcat{\monUnderlying(\Modeq G R)_{\fht}}  \eq  \lim_{\B G} \opcat{\monUnderlying(\Mod{R})_{\fht}}  \xlongto{\eq}  \lim_{\B G} \Sp[F(\blank)]^\QQ_{\nil,\ft,\bbl}  \eq  \Sp[X]^{G,\QQ}_{\nil,\ft,\bbl} \]
  of symmetric monoidal $\infty$-categories.
\end{proof}

\begin{definition}
  In the situation of \cref{prop:BMeq}, let $\cat I$ be a category.
  We say that a diagram $M \colon \cat I \to \opcat{(\Modeq{G}{R})_\fht}$ \emph{models} a diagram $E \colon \cat I \to \Sp[X]^G_{\nil,\ft,\bbl}$ if it comes equipped with an equivalence $\Psi_R \after u \after M \eq \rat E$ of functors $\cat I \to \Sp[X]^{G,\QQ}_{\nil,\ft,\bbl}$, where $u$ is the composite
  \[ \opcat{(\Modeq{G}{R})_\fht}  \longto  \opcat{\Underlying(\Modeq{G}{R})_\fht}  \xlongto{q}  \opcat{\cUnderlying(\Modeq{G}{R})_\fht} \]
  with $q$ a cofibrant replacement functor.
\end{definition}

\subsubsection*{An equivariant model for parametrized suspension spectra}

We now provide equivariant models for the functors $\Loopsinf[X]$ and $\Suspinf[X]$.
To this end, we use our fully coherent version of the corresponding non-equivariant result due to Braunack--Mayer \cite[Proposition~5.3]{Bra}, which we provided in \cref{lemma:Loops_model}.

\begin{lemma} \label{lemma:Loops_model_eq}
  In the situation of \cref{prop:BMeq}, the following diagram of $\infty$-categories commutes naturally in the pair $(X, R)$
  \[
  \begin{tikzcd}[column sep = 45]
    \opcat{\cUnderlying(\Modeq{G}{R})} \rar{\cUnderlying(\SAcn_R)} \dar[swap]{\Psi_R} & \opcat{\cUnderlying \bigl( \undercat{(\CDGAeq{G})}{R} \bigr)} \rar{\eq} & \opcat{\overcat{\Underlying(\CDGAeq{G})}{R}} \dar{\Real^G} \\
    \Sp[X] \rar{\Loopsinf[X]} & \Anover{X} & \lar[swap]{\eta^*} \Anover{\Real^G (R)}
  \end{tikzcd}
  \]
  where $\eta \colon X \to \Real^G \APL^G (X) \eq \Real^G (R)$ is the unit of the adjunction $\APL^G \dashv \Real^G$.
\end{lemma}

\begin{proof}
  As in the proof of \cref{prop:BMeq}, this follows by applying $\lim_{\B G}$ to the diagram of \cref{lemma:Loops_model} and using \cref{prop:sections}.
  Note that, for an $\infty$-category $\cat C$ that admits pullbacks, the functor $\overcat{\cat C}{\blank} \colon \opcat{\cat C} \to \Catinf$ is classified by the cartesian fibration $\ev_1 \colon \cat C^{[1]} \to \cat C$ (cf.\ \kerodon{05SA}).
  By (the dual of) \cref{lemma:limit_sections}, for a functor $F \colon \cat B \to \cat C$ from an $\infty$-groupoid $\cat B$, the limit $\lim_{\opcat{\cat B}} (\overcat{\cat C}{\blank} \circ \opcat{F})$ is thus naturally equivalent to $\Fun_{\cat C}(\cat B, \cat C^{[1]})$ (note that any section of a cartesian fibration over an $\infty$-groupoid is cartesian).
  By currying, this is equivalent to $\overcat{\Fun(\cat B, \cat C)}{F}$.
  In particular there are equivalences
  \[ \lim_{\B G} \Anover{F(\blank)}  \eq  \overcat{\Fun(\B G, \An)}{F}  \eq  \overcat{(\Anover{\B G})}{X}  \eq  \Anover{X} \]
  where $F \colon \B G \to \An$ is the functor corresponding to $X \in \Anover{\B G}$ (and we implicitly identify $\opcat{\B G} \eq \B G$).
  Similarly we have $\Sect{\B G}(\undercat{(\CDGA)}{R}) \eq \undercat{(\CDGAeq{G})}{R}$.
\end{proof}

\begin{lemma} \label{lemma:Sullivan_over}
  Let $G$ be a group and $R$ a $G$-equivariant homologically connected cofibrant cdga of finite homotopical type.
  Then the adjunction $\APL^G \dashv \Real^G$ restricts to an adjoint equivalence
  \[ \opcat{\cUnderlying \bigl( \undercat{(\CDGAeq{G})}{R} \bigr)_{\fht,\ge 1}}  \eq  \overcat{(\Anover[\QQ,\fnil,\ft]{\B G})}{\Real^G (R)}^{\ge 1} \]
  where $\overcat{(\Anover[\QQ,\fnil,\ft]{\B G})}{\Real^G (R)}^{\ge 1}$ denotes the full subcategory of $\overcat{(\Anover[\QQ,\fnil,\ft]{\B G})}{\Real^G (R)} \subset \overcat{\An}{\Real^G (R)}$ spanned by those maps $X \to \Real^G (R)$ whose fiber is connected, and the subscripts $\fht$ and $\ge 1$ in the domain refer to the full subcategory spanned by those maps $\phi \colon R \to S$ of $G$-equivariant cdgas whose underlying map of cdgas admits a minimal model $R \tensor \SA V$ with $V$ positively graded and finite-dimensional in each degree.
\end{lemma}

\begin{proof}
  Since both sides are obtained as limits over $\B G$ (see the proof of \cref{lemma:Loops_model_eq}), it is enough to consider the case $G = 1$.
  In that case the claim follows from \cite[§5.1]{Bra} (using \cref{lemma:nilpotent_composite}).
\end{proof}

\begin{lemma} \label{lemma:Susp_model_eq}
  In the situation of \cref{prop:BMeq} (and using notation as in \cref{lemma:Sullivan_over}), the following diagram of $\infty$-categories commutes naturally in the pair $(X, R)$
  \begin{equation} \label{eq:Susp_model_eq}
  \begin{tikzcd}
    \Sp[X]^{G}_{\nil,\ft,\ge 1} \rar{\rat{(\blank)}} \dar[swap]{\Loopsinf[X]} & \Sp[X]^{G,\QQ}_{\nil,\ft,\ge 1} & \lar{\Psi_R}[swap]{\eq} \opcat{\Underlying(\Modeq{G}{R})_{\fht,\ge 1}} \dar{\Lder \SAcn_R} \\
    \overcat{(\Anover[\fnil,\ft]{\B G})}{X}^{\ge 1} \rar{\rat{(\blank)}} \dar[swap]{\Suspinf[X]} & \overcat{(\Anover[\QQ,\fnil,\ft]{\B G})}{\Real^G (R)}^{\ge 1} & \lar{\Real^G}[swap]{\eq} \opcat{\Underlying \bigl( \undercat{(\CDGAeq{G})}{R} \bigr)_{\fht,\ge 1}} \dar{\Rder \forget} \\
    \Sp[X]^{G}_{\nil,\ft,\bbl} \rar{\rat{(\blank)}} & \Sp[X]^{G,\QQ}_{\nil,\ft,\bbl} & \lar{\Psi_R}[swap]{\eq} \opcat{\Underlying(\Modeq{G}{R})_{\fht}}
  \end{tikzcd}
  \end{equation}
  where the subscripts $\ge 1$ in the upper row refer to the full subcategories spanned by those $X$-spectra (resp.\ $R$-modules) that are connected (resp.\ concentrated in positive degrees).
\end{lemma}

\begin{proof}
  We consider the diagram
  \begin{equation} \label{eq:Susp_model_eq_proof}
  \begin{tikzcd}
    \Sp[X]_{\ge 1} \rar{\rat{(\blank)}} \dar[swap]{\Loopsinf[X]} & \Sp[X] \dar[swap]{\Loopsinf[X]} & \lar[swap]{\eta^*} \Sp[\Real^G(R)] \dar{\Loopsinf[\Real^G(R)]} & \lar[swap]{\Psi_R} \opcat{\Underlying(\Modeq{G}{R})} \dar{\Lder \SAcn_R} \\
    \Anover{X} \rar{\rat{(\blank)}} & \Anover{X} & \lar[swap]{\eta^*} \Anover{\Real^G (R)} & \lar[swap]{\Real^G} \opcat{\Underlying \bigl( \undercat{(\CDGAeq{G})}{R} \bigr)}
  \end{tikzcd}
  \end{equation}
  where $\eta \colon X \to \Real^G (R)$ is the canonical map and the lower left-hand horizontal map is fiberwise rationalization over $X$.
  The right-hand square commutes by \cref{lemma:Loops_model_eq}, and the left-hand square commutes via the mate of the canonical equivalence making the following diagram commute
  \begin{equation} \label{eq:Loops_rat}
  \begin{tikzcd}
    \Sp[X] \dar[swap]{\Loopsinf[X]} & \lar \Sp[X]^\QQ \dar{\Loopsinf[X]} \\
    \Anover{X} & \lar \Anover[\QQ]{X}
  \end{tikzcd}
  \end{equation}
  (that this mate is an equivalence on $\Sp[X]_{\ge 1}$ follows from $\Loopsinf$ sending rational equivalences of connected spectra to rational equivalences of animas; to see this, note that the infinite loop space of a connected spectrum is nilpotent since it is an H-space and hence simple, see e.g.\ \cite[Corollary~1.4.5]{MP}).
  It follows from the proof of \cref{lemma:param_An_rational_eq} (combined with \cref{lemma:nilpotent_composite}) that $\eta^*$ restricts to an equivalence $\overcat{(\Anover[\fnil]{\B G})}{\Real^G (R)}^{\QQ,\ge 1} \to \overcat{(\Anover[\fnil]{\B G})}{X}^{\QQ,\ge 1}$ with inverse given by fiberwise rationalization over $\B G$ (and using the equivalence $\rat X \eq \Real^G(R)$).
  It follows from \cref{lemma:rat_pullback} that fiberwise rationalizing first over $X$ and then over $\B G$ is equivalent to just rationalizing over $\B G$ when restricted to $\overcat{(\Anover[\fnil]{\B G})}{X}^{\ge 1}$.
  Restricting the diagram \eqref{eq:Susp_model_eq_proof} to the appropriate subcategories hence yields the upper half of the diagram \eqref{eq:Susp_model_eq}.
  That $\Psi_R$ and $\Real^G$ restrict to equivalences as claimed follows from \cref{prop:BMeq,lemma:Sullivan_over}.
  
  To obtain the lower half of \eqref{eq:Susp_model_eq}, we consider the diagram
  \begin{equation} \label{eq:Susp_model_eq_proof_2}
  \begin{tikzcd}
    \Sp[X] \rar{\rat{(\blank)}} & \Sp[X] & \lar[swap]{\eta^*} \Sp[\Real^G(R)] \rar{\Phi_R} & \opcat{\Underlying(\Modeq{G}{R})} \\
    \Anover{X} \rar{\rat{(\blank)}} \uar{\Suspinf[X]} & \Anover{X} \uar{\Suspinf[X]} & \lar[swap]{\eta^*} \Anover{\Real^G (R)} \uar[swap]{\Suspinf[\Real^G(R)]} \rar{\APL^G} & \opcat{\Underlying \bigl( \undercat{(\CDGAeq{G})}{R} \bigr)} \uar[swap]{\Rder \forget}
  \end{tikzcd}
  \end{equation}
  where the right-hand square is obtained from the right-hand square of \eqref{eq:Susp_model_eq_proof} by passing to left adjoints, and the left-hand square is obtained from \eqref{eq:Loops_rat} by passing to left adjoints.
  As above, restricting to the appropriate full subcategories yields the desired diagram.
  To see this, note that the restriction of $\APL^G$ is inverse to the restriction of $\Real^G$ by \cref{lemma:Sullivan_over}, and that the analogous statement for $\Phi_R$ and $\Psi_R$ holds by \cite[Theorem~4.20]{Bra}.
\end{proof}

\begin{lemma} \label{lemma:Loops_Susp_adjoint}
  Let $G$ be a group and $R \to S$ a map of $G$-equivariant homologically connected cofibrant cdgas of finite homotopical type that models a map $f \colon Y \to X$ of fiberwise nilpotent animas over $\B G$.
  Assume that the fiber of $f$ is connected, and let $\alpha \colon M \to S$ be a map of $G$-equivariant $R$-modules of finite homotopical type that models a map $a \colon \Suspinf[X] Y \to P$ in $\Sp[X]^G_{\nil,\ft,\bbl}$.
  Assume that $P$ is connected and that $M$ is cofibrant and concentrated in positive degrees.
  Then the canonical map $\beta \colon \SA_R M \to S$ of $G$-equivariant cdgas models the map $b \colon Y \to \Loopsinf[X] P$ of animas over $\B G$ that is adjoint to $a$.
\end{lemma}

\begin{proof}
  We consider the commutative diagrams
  \[
  \begin{tikzcd}
    \Sp[\Real^G(R)] \dar[swap]{\Loopsinf[\Real^G(R)]} & \lar[swap]{\Psi_R} \opcat{\Underlying(\Modeq{G}{R})} \dar{\Lder \SAcn_R} & \Sp[\Real^G(R)] \rar{\Phi_R} & \opcat{\Underlying(\Modeq{G}{R})} \\
    \Anover{\Real^G (R)} & \lar[swap]{\Real^G} \opcat{\Underlying \bigl( \undercat{(\CDGAeq{G})}{R} \bigr)} & \Anover{\Real^G (R)} \uar{\Suspinf[\Real^G(R)]} \rar{\APL^G} & \opcat{\Underlying \bigl( \undercat{(\CDGAeq{G})}{R} \bigr)} \uar[swap]{\Rder \forget}
  \end{tikzcd}
  \]
  where the left-hand one is obtained from \cref{lemma:Loops_model_eq} and the right-hand one by taking left adjoints.
  By the proof of \cref{lemma:Susp_model_eq}, the fiberwise suspension spectrum $\Suspinf[\Real^G(R)] \rat Y$ of the rationalization $\rat f \colon \rat Y \to \rat X \eq \Real^G(R)$ is equivalent to the rationalization $\rat{(\Suspinf[X] Y)}$ under the equivalence $\Sp[X]^{G,\QQ}_\nil \eq \Sp[\Real^G(R)]^{G,\QQ}_\nil$ induced by the one of \cref{lemma:param_Sp_rational_eq}.
  Hence the rationalization of $a$ yields a map $a' \colon \Suspinf[\Real^G(R)] \rat Y \to \rat P \eq \Psi_R(M)$.
  We observe that the map $\APL^G(\rat Y) \to \Lder \SAcn_R (M)$ of $\opcat{\Underlying ( \undercat{(\CDGAeq{G})}{R} )}$, obtained as the adjoint of the following adjoint of $a'$
  \[ b'  \colon  \rat Y  \longto  \Loopsinf[\Real^G(R)] (\Psi_R (M))  \eq  \Real^G (\Lder \SAcn_R (M)) \]
  is homotopic to the map obtained as the adjoint of the other adjoint
  \[ \Rder \forget (\APL^G (\rat Y)) \eq \Phi_R (\Suspinf[\Real^G(R)] \rat Y) \to M \]
  of $a'$.
  Equivalently this says that the following composite (where the counit at the end is an equivalence by \cref{lemma:Sullivan_over})
  \[ S \eq \APL^G(\rat Y)  \xlongto{b'}  \APL^G ( \Loopsinf[\Real^G(R)] (\rat P) )  \eq   \APL^G ( \Real^G ( \Lder \SAcn_R (M) ) )  \xlongto{\eq}  \Lder \SAcn_R (M) \]
  is homotopic to the adjoint of the composite
  \[ S  =  \Rder \forget ( S )  \eq  \Rder \forget ( \APL^G (\rat Y) )  \eq  \Phi_R ( \Suspinf[\Real^G(R)] \rat Y )  \xlongto{a'}  \Phi_R (\rat P)  \eq  \Phi_R ( \Psi_R(M) )  \eq  M \]
  which is given by (the opposite of) $\alpha$ by assumption.
  By \cref{prop:Shulman}, this implies that $b' \colon \rat Y \to \Loopsinf[\Real^G(R)] (\rat P)$ is modeled by $\beta$.
  Lastly we note that a diagram chase shows that $b'$ is equivalent to the rationalization of $b$ (using that the natural transformations making the left-hand squares of \eqref{eq:Susp_model_eq_proof} and \eqref{eq:Susp_model_eq_proof_2} commute are mates of each other).
\end{proof}

\subsection{Models and the Wirthmüller context} \label{sec:model_Wirthmueller}

In this subsection, we provide equivariant rational models for the adjunction $f_! \dashv f^*$ of parametrized spectra, as well as related maps such as the pull--push equivalence and the projection formula (cf.\ \cref{sec:param_spectra}).
Non-equivariant versions of this are partially already contained in the work of Braunack--Mayer \cite[§5.2]{Bra}.

\begin{notation}
  Let $R$ be a cdga equipped with an action of a group $G$.
  Throughout the rest of this section, we write $q_R \colon \Modeq G {R} \to \Modeq G {R}$ for some fixed cofibrant replacement functor; we sometimes omit the subscript if there is no risk of confusion.
\end{notation}

\begin{lemma} \label{lemma:model_f_!}
  Let $G$ be a group and $\phi \colon R \to S$ a map of $G$-equivariant homologically connected cofibrant cdgas of finite homotopical type that models a map $f \colon X \to Y$ of fiberwise nilpotent animas over $\B G$.
  Assume that the fiber of $f$ is connected.
  Then $\Rder \phi^*$ and $f_!$ restrict to functors as in the diagram
  \[
  \begin{tikzcd}
    \opcat{\Underlying(\Modeq G {S})_\fht} \dar[swap]{\Rder \phi^*} \rar{\eq}[swap]{q_S} & \opcat{\cUnderlying(\Mod{S})_\fht} \rar{\eq}[swap]{\Psi_S} & \Sp[X]^{G,\QQ}_{\nil,\ft,\bbl} \dar{f_!} & \lar[swap]{\rat{(\blank)}} \Sp[X]^{G}_{\nil,\ft,\bbl} \dar{f_!} \\
    \opcat{\Underlying(\Modeq G {R})_\fht} \rar{\eq}[swap]{q_R} & \opcat{\cUnderlying(\Mod{R})_\fht} \rar{\eq}[swap]{\Psi_R} & \Sp[Y]^{G,\QQ}_{\nil,\ft,\bbl} & \lar[swap]{\rat{(\blank)}} \Sp[Y]^{G}_{\nil,\ft,\bbl}
  \end{tikzcd}
  \]
  and this diagram canonically commutes.
\end{lemma}

\begin{proof}
  Consider the left-hand diagram of horizontal fiber sequences
  \[
  \begin{tikzcd}
    \fib \alpha \rar{i} \dar[swap]{g} & X \rar{\alpha} \dar{f} & \B G \dar[equal] & \Sp[X] \rar{i^*} \dar[swap]{f_!} & \Sp[\fib \alpha] \dar{g_!} \\
    \fib \beta \rar{j} & Y \rar{\beta} & \B G & \Sp[Y] \rar{j^*} & \Sp[\fib \beta]
  \end{tikzcd}
  \]
  and note that its left-hand square is a pullback.
  Hence the right-hand diagram commutes by the pull--push formula.
  That $\Rder \phi^*$ and $g_!$ restrict as claimed is contained in \cite[Proposition~5.8]{Bra}; hence $f_!$ also restricts as claimed.
  The commutativity of the left-hand square in the statement then follows from \cref{prop:BMeq} by taking adjoints.
  Commutativity of the right-hand square follows from the projection formula.
\end{proof}

Note that $\phi^*$ preserves quasi-isomorphisms, so that it itself is its right derived functor, i.e.\ there is a canonical equivalence $\Rder \phi^* \eq \Underlying(\phi^*)$.

\begin{lemma} \label{lemma:model_composition}
  Let $G$ be a group, and let the following left-hand side be maps of $G$-equivariant homologically connected cofibrant cdgas of finite homotopical type
  \[ T \xlongfrom{\psi} S \xlongfrom{\phi} R  \qquad  X \xlongto{f} Y \xlongto{g} Z \]
  that model the two right-hand maps of fiberwise nilpotent animas over $\B G$.
  Then the canonical natural equivalence $f^* g^* \eq (g f)^*$ is modeled by the equivalence $\Lder \psi_! \Lder \phi_! \eq \Lder (\psi \phi)_!$ induced by the canonical natural isomorphism $\psi_! \phi_! \iso (\psi \phi)_!$.
\end{lemma}

\begin{proof}
  This follows from \cref{prop:BMeq}.
\end{proof}

\begin{lemma} \label{lemma:pull-push_model}
  Let $G$ be a group, and let the following left-hand diagram be a pushout square of $G$-equivariant cdgas, all of which are cofibrant, homologically connected, and of finite homotopical type,
  \[
  \begin{tikzcd}
    R \tensor_\k S & \lar[swap]{\iota_S} S & X \times_B Y \rar{\pr_Y} \dar[swap]{\pr_X} & Y \dar{g} \\
    R \uar{\iota_R} & \lar[swap]{\phi} \k \uar[swap]{\psi} & X \rar{f} & B
  \end{tikzcd}
  \]
  that models the right-hand pullback square of animas over $\B G$, all of which are fiberwise nilpotent.
  Assume that the fiber of $f$ is connected.
  Then the pull--push equivalence $g^* f_! \eq (\pr_Y)_! (\pr_X)^*$ is modeled by the natural transformation $\Lder \psi_! \Rder \phi^* \to \Rder (\iota_S)^* \Lder (\iota_R)_!$ given as
  \[ q \phi^* (\blank) \tensor_\k S \xlongfrom{\eq} q \phi^* q (\blank) \tensor_\k S \longto \phi^* q (\blank) \tensor_\k S \iso (\iota_S)^* \bigl( q (\blank) \tensor_R (R \tensor_\k S) \bigr) \]
  (which is thus a quasi-isomorphism).
\end{lemma}

\begin{proof}
  First note that the fiber of $\pr_Y$ is connected since it is a pullback of $f$.
  The pull--push equivalence $g^* f_! \eq (\pr_Y)_! (\pr_X)^*$ is defined to be the mate of the canonical equivalence $(\pr_X)^* g^* \eq (\pr_Y)^* f^*$.
  Then the claim follows from \cref{lemma:model_composition,prop:Shulman}.
\end{proof}

\begin{lemma} \label{lemma:pullback_monoidal_model}
  Let $G$ be a group and $\phi \colon R \to S$ a map of $G$-equivariant homologically connected cofibrant cdgas of finite homotopical type that models a map $f \colon X \to Y$ of fiberwise nilpotent animas over $\B G$.
  Then the canonical equivalence $f^*(\blank \tensor_Y \blank) \eq f^*(\blank) \tensor_X f^*(\blank)$ is modeled by the isomorphism
  \[ \Lder \phi_! \bigl( \blank \dertensor_R \blank \bigr)  \eq  \bigl( q(\blank) \tensor_R q(\blank) \bigr) \tensor_R S  \iso  \bigl( q(\blank) \tensor_R S \bigr) \tensor_S \bigl( q(\blank) \tensor_R S \bigr)  \eq  \Lder \phi_! (\blank) \dertensor_S \Lder \phi_! (\blank) \]
  of $G$-equivariant $S$-modules.
\end{lemma}

\begin{proof}
  This follows from \cref{prop:BMeq}.
\end{proof}

\begin{lemma} \label{lemma:projection_model}
  Let $G$ be a group and $\phi \colon R \to S$ a map of $G$-equivariant homologically connected cofibrant cdgas of finite homotopical type that models a map $f \colon X \to Y$ of fiberwise nilpotent animas over $\B G$.
  Assume that the fiber of $f$ is connected.
  Then the projection formula $f_!(f^*(\blank) \tensor_X \blank) \eq \blank \tensor_Y f_!(\blank)$ is modeled by the natural equivalence $\Rder \phi^* (\Lder \phi_!(\blank) \dertensor_S \blank) \eq \blank \dertensor_R \Rder \phi^*(\blank)$ given as the following zig-zag of maps of $G$-equivariant $R$-modules
  \[ \phi^* \bigl( ( q (\blank) \tensor_R S ) \tensor_S q(\blank) \bigr)  \iso  q(\blank) \tensor_R \phi^* q (\blank)  \xlongto{\eq}  q(\blank) \tensor_R \phi^* (\blank)  \xlongfrom{\eq}  q(\blank) \tensor_R q \phi^* (\blank) \]
  (where the indicated maps are quasi-isomorphisms by \cref{lemma:cofibrant_flat}).
\end{lemma}

\begin{proof}
  Recall that the projection formula is defined to be the composite
  \[ f_! \bigl( f^*(M) \tensor_X N \bigr)  \xlongto{\eta}  f_! \bigl( f^*(M) \tensor_X f^* f_! (N) \bigr)  \eq  f_! f^* \bigl( M \tensor_Y f_! (N) \bigr)  \xlongto{\epsilon}  M \tensor_Y f_! (N) \]
  for $M \in \Sp[Y]$ and $N \in \Sp[X]$.
  Then the claim follows from \cref{lemma:model_f_!,lemma:pullback_monoidal_model,prop:Shulman} via a diagram chase.
\end{proof}

We record the following for later use.

\begin{lemma} \label{lemma:restrict_cofibration}
  Let $G$ be a group and $\phi \colon R \to S$ a map of $G$-equivariant cdgas such that $S$ is cofibrant as an $R$-module.
  Then the functor $\phi^* \colon \Modeq{G}{S} \to \Modeq{G}{R}$ preserves cofibrations (and in particular cofibrant objects).
\end{lemma}

\begin{proof}
  Since cofibrations of $G$-equivariant modules are defined to be the cofibrations of the underlying modules, it is enough to prove this in the case $G = 1$.
  In that case, the functor $\phi^*$ is left adjoint to the functor $\Hom_R(S, \blank) \colon \Mod{R} \to \Mod{S}$, which is right Quillen when $S$ is cofibrant in $\Mod{R}$.
  Hence $\phi^*$ is left Quillen and in particular it preserves cofibrations.
\end{proof}

\subsection{Modeling the dualizing spectrum} \label{sec:model_dualizing}

In this subsection, our goal is to provide an (equivariant) rational model for the adjunction $f^* \dashv f_*$ of parametrized spectra, and the related dualizing spectrum $\DS{f}$ (cf.\ \cref{sec:param_spectra}).
The main obstacle is that the functor $f^*$ is modeled by scalar extension, which does not generally admit a left adjoint.
However, when $\phi \colon R \to S$ is a map of commutative rings and $S$ is finitely generated and projective as an $R$-module, then there is a natural isomorphism $\blank \tensor_R S \iso \Hom_R(\Hom_R(S, R), \blank )$.
In this case $\phi_!$ is thus right adjoint to the functor $\blank \tensor_S \Hom_R(S, R)$.
We begin with a homotopical version of this observation.

\begin{observation}
  Let $G$ be a group, $R$ a cdga with a $G$-action, and $M$ and $N$ two $G$-equivariant $R$-modules.
  Then the $R$-module $\Hom_R(M, N)$ becomes $G$-equivariant by equipping it with the conjugation action.
  When $R \to S$ is a map of $G$-equivariant cdgas and $M$ or $N$ is a $G$-equivariant $S$-module, then $\Hom_R(M, N)$ canonically lifts to a $G$-equivariant $S$-module.
\end{observation}

Recall that we write $\dual{M}_R \defeq \Hom_R(M, R)$ for the dual of an $R$-module $M$.

\begin{definition}
  Let $G$ be a group and $\phi \colon R \to S$ a map of $G$-equivariant cdgas.
  We write
  \begin{align*}
    \phi^\antishriek &\defeq \phi^* \bigl( \blank \tensor_S q_S (\dual{S}_R) \bigr) \colon \Modeq{G}{S} \longto \Modeq{G}{R} \\
    \phi_\antishriek &\defeq \Hom_R \bigl( q_S (\dual S_R), \blank \bigr) \colon \Modeq{G}{R} \longto \Modeq{G}{S}
  \end{align*}
  and note that these two functors form a canonical adjunction $\phi^\antishriek \dashv \phi_\antishriek$.
\end{definition}

Note that the adjunction $\phi^\antishriek \dashv \phi_\antishriek$ is Quillen when $S$ is cofibrant as an $R$-module since in that case $\phi^\antishriek$ preserves cofibrations and trivial cofibrations by \cref{lemma:restrict_cofibration}.
Furthermore note that $\phi_\antishriek$ preserves quasi-isomorphisms, so that it itself is its right derived functor, i.e.\ there is a canonical equivalence $\Rder \phi_\antishriek \eq \Underlying(\phi_\antishriek)$.

\begin{lemma} \label{lemma:scalar_ext_right_adjoint}
  Let $R$ be a cdga, and let $M$ and $N$ be $R$-modules such that $N$ is cofibrant and homotopically finite.
  Let $\epsilon \colon q' (\dual N_R) \to \dual N_R$ be any quasi-isomorphism of $R$-modules with $q' (\dual N_R)$ cofibrant.
  Then the composite
  \begin{align*}
    M \tensor_R N  &\longto  \Hom_R (\dual N_R, M)  \xlongto{\epsilon^*}  \Hom_R \bigl( q' (\dual N_R), M \bigr) \\
    m \tensor n  &\longmapsto  \bigl( \alpha \mapsto (-1)^{\deg \alpha \deg n} m \alpha(n) \bigr)
  \end{align*}
  is a quasi-isomorphism.
  In particular, for a group $G$ and a map $\phi \colon R \to S$ of $G$-equivariant cdgas such that $S$ is cofibrant and homotopically finite as an $R$-module, this yields a natural quasi-isomorphism $\phi_! \to \phi_\antishriek$ of functors $\Modeq{G}{R} \to \Modeq{G}{S}$.
\end{lemma}

\begin{proof}
  Let $\widetilde N \defeq R \tensor V$ be the minimal model of $N$.
  We obtain the commutative diagram
  \[
  \begin{tikzcd}
    M \tensor_R \widetilde N \rar{\iso} \dar[swap]{\eq} & \Hom_R (\dual {\widetilde N}_R, M) \rar{\eq} \dar & \Hom_R \bigl( q (\dual {\widetilde N}_R), M \bigr) \dar{\eq} \\
    M \tensor_R N \rar &  \Hom_R (\dual N_R, M) \rar &  \Hom_R \bigl( q (\dual N_R), M \bigr) \rar{\eq} & \Hom_R \bigl( q' (\dual N_R), M \bigr)
  \end{tikzcd}
  \]
  where the upper left-hand horizontal map is an isomorphism since $V$ is finite-dimensional, and the left-hand vertical map is a quasi-isomorphism by \cref{lemma:cofibrant_flat}.
  To see that the upper right-hand horizontal map is a quasi-isomorphism, note that $\dual {\widetilde N}_R$ is cofibrant by the proof of \cref{lemma:Hom_hf} below.
  For the bottom right-most map, we use that a dashed lift as in the diagram
  \[
  \begin{tikzcd}
     & q (\dual N_R) \dar[two heads]{\eq} \\
    q' (\dual N_R) \rar{\eq} \urar[dashed]{\eq} & \dual N_R
  \end{tikzcd}
  \]
  always exists.
\end{proof}

\begin{lemma} \label{lemma:Hom_hf}
  Let $R$ be a cdga and $M$ and $N$ two $R$-modules such that $M$ is cofibrant and homotopically finite.
  If $N$ is of finite homotopical type, then so is $\Hom_R(M, N)$.
  If $N$ is homotopically finite, then so is $\Hom_R(M, N)$.
\end{lemma}

\begin{proof}
  Let $R \tensor V \to M$ and $R \tensor W \to N$ be the respective minimal models.
  Then the two induced maps of $R$-modules
  \[ \Hom_R(M, N)  \xlongto{\eq}  \Hom_R(R \tensor V, N)  \xlongfrom{\eq}  \Hom_R(R \tensor V, R \tensor W) \]
  are quasi-isomorphisms, using that $M$ is cofibrant for the first one.
  The right-hand side is isomorphic to $\Hom(V, R \tensor W) \iso R \tensor W \tensor \dual V$ since $V$ is finite-dimensional.
  Chasing the differential of $\Hom_R(R \tensor V, R \tensor W)$ through the two isomorphisms, we claim that they in fact exhibit it as being minimal.
  To see this, choose bases of $V$ and $W$ exhibiting $M$ and $N$ as minimal and consider the dual basis of $\dual V$ equipped with the reverse order.
  Then ordering the canonical basis of $W \tensor \dual V$ first by degree and then lexicographically proves the claim.
  Observing that $W \tensor \dual V$ is of finite type when $W$ is, and finite-dimensional when $W$ is, completes the proof.
\end{proof}

\begin{lemma} \label{cor:extension_left_adjoint}
  Let $G$ be a group and $\phi \colon R \to S$ a map of $G$-equivariant cdgas such that $S$ is cofibrant and homotopically finite as an $R$-module.
  Then there is a natural equivalence $\Lder \phi_! \eq \Rder \phi_\antishriek$ of functors $\Underlying(\Modeq{G}{R}) \to \Underlying(\Modeq{G}{S})$.
  In particular $\Lder \phi^\antishriek$ is left adjoint to $\Lder \phi_!$.
\end{lemma}

\begin{proof}
  The natural isomorphism is obtained as the composite
  \[ \Lder \phi_!  =  \phi_! q  \xlongto{\eq}  \phi_\antishriek q  \xlongto{\eq}  \phi_\antishriek  =  \Rder \phi_\antishriek \]
  where the first natural quasi-isomorphism is the one of \cref{lemma:scalar_ext_right_adjoint}.
\end{proof}

We now prove that, for a map of animas $f$ modeled by a map of cdgas $\phi \colon R \to S$, the functor $f_*$ is modeled by the functor $\Lder \phi^\antishriek$ as long as the fiber of $f$ is compact.
As the following lemma shows, on the algebraic side this corresponds to the condition encountered above that $S$ is homotopically finite as an $R$-module.

\begin{lemma} \label{lemma:compact_fibs}
  Let $G$ be a group and $\phi \colon R \to S$ a map of $G$-equivariant homologically connected cofibrant cdgas of finite homotopical type that models a map $f \colon X \to Y$ of fiberwise nilpotent animas over $\B G$.
  If the fiber of $f$ is compact and connected, then $S$ is homotopically finite as an $R$-module.
\end{lemma}

\begin{proof}
  By restricting along the group homomorphism $1 \to G$, it is enough to prove the claim for $G = 1$.
  In that case, let $y \colon * \to Y$ be a map of animas and choose a map of cdgas $\epsilon \colon R \to \QQ$ that models it.
  Furthermore, let $M = R \tensor V$ be the minimal model of $S$ as an $R$-module.
  By \cref{prop:BMeq,lemma:model_f_!}, the $\QQ$-module $\Lder \epsilon_! \Rder \phi^* (S) \eq V$ is a model for $y^* f_! (\SS_X)$, which is equivalent to the suspension spectrum $\SS[F_y]$ of the fiber $F_y$ of $f$ over $y$.
  Since $F_y$ is compact, its rational cohomology is finite-dimensional, and thus $V$ is as well.
\end{proof}

\begin{lemma} \label{cor:model_f_*}
  Let $G$ be a group and $\phi \colon R \to S$ a map of $G$-equivariant homologically connected cofibrant cdgas of finite homotopical type that models a map $f \colon X \to Y$ of fiberwise nilpotent animas over $\B G$.
  Assume that the fiber of $f$ is compact and connected and that $S$ is cofibrant as an $R$-module.
  Then $\Lder \phi^\antishriek$ and $f_*$ restrict to functors as follows
  \[
  \begin{tikzcd}
    \opcat{\cUnderlying(\Modeq G {S})_\fht} \dar[swap]{\Lder \phi^\antishriek} \rar{\eq}[swap]{\Psi_S} & \Sp[X]^{G,\QQ}_{\nil,\ft,\bbl} \dar{f_*} & \lar[swap]{\rat{(\blank)}} \Sp[X]^{G}_{\nil,\ft,\bbl} \dar{f_*} \\
    \opcat{\cUnderlying(\Modeq G {R})_\fht} \rar{\eq}[swap]{\Psi_R} & \Sp[Y]^{G,\QQ}_{\nil,\ft,\bbl} & \lar[swap]{\rat{(\blank)}} \Sp[Y]^{G}_{\nil,\ft,\bbl}
  \end{tikzcd}
  \]
  and this diagram canonically commutes.
\end{lemma}

\begin{proof}
  That $f_*$ restricts as claimed and that the right-hand square commutes is the content of \cref{lemma:f_*_nilpotent} below.
  Commutativity of the left-hand square then follows from \cref{prop:BMeq,cor:extension_left_adjoint} as long as we can prove that $\Lder \phi^\antishriek$ restricts to a functor $\Hocat(\Modeq{G}{S})_\fht \to \Hocat(\Modeq{G}{R})_\fht$.
  To this end, let $M$ be a $G$-equivariant $S$-module of finite homotopical type and $N = S \tensor V$ its minimal model.
  Then $\Lder \phi^\antishriek(M)$ is quasi-isomorphic to $K \defeq \phi^* ( N \tensor_S q_S(\dual{S}_R) )$.
  By definition, the $S$-module $N$ has an exhaustive filtration with associated graded the free $S$-module $S \tensor V$.
  Hence $K$ has an exhaustive filtration with associated graded isomorphic to the $R$-module $V \tensor q_S(\dual{S}_R)$.
  Note that $\dual{S}_R$ is homotopically finite as an $R$-module by \cref{lemma:compact_fibs,lemma:Hom_hf}; let $R \tensor V'$ be its minimal model.
  In particular its homology is bounded below; thus the same is true for $V \tensor q_S(\dual{S}_R)$ and hence $K$, so that the latter admits a minimal model $R \tensor W$.
  Now let $\epsilon \colon R \to \QQ$ be an augmentation (which exists by \cref{lemma:augmented}).
  Note that $\Lder \epsilon_! (K) \eq W$; on the other hand, the $R$-module $\Lder \epsilon_! (K) = K \tensor_R \QQ$ admits an exhaustive filtration with associated graded $(V \tensor q_S(\dual{S}_R)) \tensor_R \QQ \eq V \tensor V'$.
  Since $V'$ is finite-dimensional and $V$ is finite-dimensional in each degree, the latter is also true for $V \tensor V'$ and thus $W$.
  Hence $K$ is of finite homotopical type.
\end{proof}

\begin{lemma} \label{lemma:f_*_nilpotent}
  Let $G$ be a group and $f \colon X \to Y$ a map of fiberwise nilpotent animas over $\B G$.
  If the fiber of $f$ is compact and connected, then $f_*$ restricts to functors as follows
  \[
  \begin{tikzcd}
    \Sp[X]^{G}_{\nil,\ft,\bbl} \rar{\rat{(\blank)}} \dar[swap]{f_*} & \Sp[X]^{G,\QQ}_{\nil,\ft,\bbl} \dar{f_*} \\
    \Sp[Y]^{G}_{\nil,\ft,\bbl} \rar{\rat{(\blank)}} & \Sp[Y]^{G,\QQ}_{\nil,\ft,\bbl}
  \end{tikzcd}
  \]
  and this diagram canonically commutes.
\end{lemma}

\begin{proof}
  We first prove that $f_*$ restricts to functors as claimed; for this it is, by the pull--push formula, enough to consider the case $G = 1$.
  By \cite[Theorem~2.35]{Bra} an $X$-spectrum $P$ is nilpotent if and only if each step $P_{\ge n} \to P_{\ge n-1}$ in the Postnikov tower of $P$ factors as a finite number of extensions by $X^*(\Susp[n] \H A)$ for varying abelian groups $A$.
  The $Y$-spectrum $f_* X^*(\Susp[n] \H A)$ is obtained from $f_! (\SS_X)$ by fiberwise applying $\Hom_\Sp(\blank, \Susp[n] \H A)$ (see e.g.\ \cite[Proposition~3.4]{ABG}).
  Note that $f_! (\SS_X) \eq \SS_Y[X]$, so that, writing $F$ for the fiber of $f$, the fiberwise homotopy groups of $f_* X^*(\Susp[n] \H A)$ are given by $\Coho * (F; \shift[n] A)$ with its canonical $\pi_1(Y)$-action (cf.\ \cref{obs:pi_1_action}).
  The analogous $\pi_1(Y)$-action on $\Ho * (F; \ZZ)$ is degreewise nilpotent by \cite[Corollary~2.2]{Hil} (using that $f$ is nilpotent by \cref{lemma:nilpotent_composite} since $F$ is connected).
  Hence the action on $\Coho * (F; \shift[n] A)$ is nilpotent as well, using that the class of nilpotent actions is closed under extensions and applying $\Hom(\blank, \shift[n] A)$ and $\operatorname{Ext}(\blank, \shift[n] A)$ by \cite[Ch.~I, Proposition~4.15]{HMR} (or, more precisely, the same argument applied to a contravariant functor).
  Since $F$ is compact, the cohomology $\Coho * (F; A)$ is furthermore concentrated in finitely many degrees.
  Thus $f_* X^*(\Susp[n] \H A)$ is nilpotent and concentrated in finitely many degrees.
  Since $f_*$ preserves limits, this argument implies (using \cite[Proposition~2.33 and Corollary~2.37]{Bra}) that if an $X$-spectrum $P$ is bounded below nilpotent, then so is $f_*(P)$.
  If $P$ is additionally of finite rational type, then, for each $n$, the sum of the dimensions of the vector spaces $A \tensor \QQ$ is finite; hence $f_*(P)$ is again of finite rational type.
  Similarly, if $P$ is rational, then each $A$ can be chosen to be a rational vector space, and hence $f_*(P)$ is rational as well.
  
  Lastly, note that whether the canonical map $\alpha \colon \rat{f_*(P)} \to f_*(\rat P)$ is an equivalence can be checked fiberwise, so it is again enough to prove this for $G = 1$.
  Note that $\alpha$ is an equivalence for $P = X^*(\Susp[n] \H A)$ since $\Coho * (F; A) \tensor \QQ \to \Coho * (F; A \tensor \QQ)$ is an isomorphism.
  Since $F$ is compact, for every $k$, the maps $\pi_k(f_*(P_{\ge n})) \to \pi_k(f_*(P_{\ge n-1}))$ are isomorphisms for $n$ large enough.
  Hence taking the limit of this tower commutes with rationalization, and so $\alpha$ is an equivalence for any bounded-below nilpotent $X$-spectrum $P$.
\end{proof}

We now provide rational models for the (unit and counit of the) adjunction $f^* \dashv f_*$, as well as the corresponding pull--push equivalence.

\begin{lemma} \label{lemma:f_*_adjunction_model}
  In the situation of \cref{cor:model_f_*}, the unit of the adjunction $f^* \dashv f_*$ is modeled by the map $\Lder \phi^\antishriek \Lder \phi_! \to \id$ given at $M \in \Modeq{G}{R}$ as
  \[ \phi^* \bigl( q(M) \tensor_R S \tensor_S q(\dual{S}_R) \bigr)  \longto  \phi^* \bigl( M \tensor_R S \tensor_S \dual{S}_R \bigr)  \xlongto{\ev}  M \]
  and the counit is modeled by the map $\id \to \Lder \phi_! \Lder \phi^\antishriek$ given at $N \in \Modeq{G}{S}$ by the zig-zag
  \[ N  \xlongfrom{\eq}  q(N)  \xlongto{\eta}  \Hom_R \bigl( q(\dual{S}_R), q(N) \tensor_S q(\dual{S}_R) \bigr)  \xlongfrom{\eq}  \phi^* \bigl( q(N) \tensor_S q(\dual{S}_R) \bigr) \tensor_R S \]
  where $\eta$ is the unit of the adjunction $\phi^\antishriek \dashv \phi_\antishriek$ and the last map is the one of \cref{lemma:scalar_ext_right_adjoint}.
\end{lemma}

\begin{proof}
  Using \cref{prop:Shulman}, we obtain descriptions of the unit and counit of the adjunction $\Lder \phi^\antishriek \dashv \Rder \phi_\antishriek$ in terms of the adjunction $\phi^\antishriek \dashv \phi_\antishriek$.
  Combining this with the description of the equivalence $\Rder \phi_\antishriek \eq \Lder \phi_!$ from (the proof of) \cref{cor:extension_left_adjoint} yields the desired statement.
\end{proof}

For our model of the pull--push equivalence of $f_*$, we require the following preliminaries.

\begin{lemma} \label{lemma:extend_dual}
  Let $G$ be a group, $\k \to R$ and $\k \to S$ two $G$-equivariant maps of cdgas, and $M$ and $N$ two $G$-equivariant $R$-modules such that $M$ is cofibrant and homotopically finite.
  Assume that $S$ is cofibrant as a $\k$-module.
  Then the following map is a quasi-isomorphism of $G$-equivariant $(R \tensor_\k S)$-modules
  \begin{align*}
    \Hom_R(M, N) \tensor_\k S  &\xlongto{\eq}  \Hom_{R \tensor_\k S}(M \tensor_\k S, N \tensor_\k S) \\
    \alpha \tensor s  &\longmapsto  \alpha \tensor s_*
  \end{align*}
  where $s_* \colon S \to S$ denotes multiplication by $s$.
\end{lemma}

\begin{proof}
  It follows from the definitions that the map is $G$-equivariant.
  To prove that it is a quasi-isomorphism, let $M' \defeq R \tensor V$ be the minimal model of $M$.
  We obtain a commutative diagram
  \[
  \begin{tikzcd}
    \Hom_R(M, N) \tensor_\k S \rar \dar[swap]{\eq} & \Hom_{R \tensor_\k S}(M \tensor_\k S, N \tensor_\k S) \dar{\eq} \\
    \Hom_R(M', N) \tensor_\k S \rar & \Hom_{R \tensor_\k S}(M' \tensor_\k S, N \tensor_\k S) \rar{\iso} & \Hom_R(M', N \tensor_\k S)
  \end{tikzcd}
  \]
  where the left-hand vertical map is a quasi-isomorphism by \cref{lemma:cofibrant_flat} since $S$ is cofibrant as a $\k$-module, and the right-hand vertical map is a quasi-isomorphism since both $M \tensor_\k S$ and $M' \tensor_\k S$ are cofibrant $(R \tensor_\k S)$-modules.
  The bottom horizontal composite is an isomorphism since $M'$ is a finite-dimensional quasi-free $R$-module.
  This implies the claim.
\end{proof}

\begin{observation} \label{obs:pr_DS}
  In the situation of \cref{lemma:extend_dual}, let $R \to T$ be a map of $G$-equivariant cdgas and assume that $T$ is cofibrant and homotopically finite as an $R$-module.
  Considering the case $M = T$ and $N = R$, there exists a dashed quasi-isomorphism of $G$-equivariant $(T \tensor_\k S)$-modules
  \[
  \begin{tikzcd}
    q_T (\dual T_R) \tensor_\k S \rar[dashed]{\eq} \dar[swap]{\eq} & q_{T \tensor_\k S} \bigl( \dual {(T \tensor_\k S)}_{R \tensor_\k S} \bigr) \dar[two heads]{\eq} \\
    \dual T_R \tensor_\k S \rar{\eq} & \dual {(T \tensor_\k S)}_{R \tensor_\k S}
  \end{tikzcd}
  \]
  such that the diagram commutes.
\end{observation}

\begin{lemma} \label{lemma:pull-copush_model}
  Let $G$ be a group, and let the following left-hand diagram be a $G$-equivariant pushout square of cdgas, all of which are cofibrant, homologically connected, and of finite homotopical type,
  \[
  \begin{tikzcd}
    R \tensor_\k S & \lar[swap]{\iota_S} S & X \times_B Y \rar{\pr_Y} \dar[swap]{\pr_X} & Y \dar{g} \\
    R \uar{\iota_R} & \lar[swap]{\phi} \k \uar[swap]{\psi} & X \rar{f} & B
  \end{tikzcd}
  \]
  that models the right-hand pullback square of fiberwise nilpotent animas over $\B G$.
  Assume that the fiber of $f$ is compact and connected, and that $R$ and $S$ are cofibrant as $\k$-modules.
  Then the pull--push equivalence $g^* f_* \eq (\pr_Y)_* (\pr_X)^*$ is modeled by the natural equivalence $\Lder \psi_! \Lder \phi^\antishriek \eq \Lder (\iota_S)^\antishriek \Lder (\iota_R)_!$ given by the map
  \[ \bigl( q_R (\blank) \tensor_R q_R (\dual{R}_\k) \bigr) \tensor_\k S  \xlongto{\eq}  q_R (\blank) \tensor_R q_{R \tensor_\k S} \dual{(R \tensor_\k S)}_S  \iso  \bigl( q_R (\blank) \tensor_\k S \bigr) \tensor_{R \tensor_\k S} q_{R \tensor_\k S} \dual{(R \tensor_\k S)}_S \]
  of \cref{obs:pr_DS}.
\end{lemma}

\begin{proof}
  The pull--push equivalence $g^* f_* \eq (\pr_Y)_* (\pr_X)^*$ is defined to be the mate of the canonical equivalence $(\pr_X)^* f^* \eq (\pr_Y)^* g^*$.
  By \cref{lemma:model_composition}, this latter equivalence is modeled by the equivalence $\alpha \colon \Lder (\iota_R)_! \Lder \phi_! \eq \Lder (\iota_S)_! \Lder \psi_!$ induced by the canonical natural isomorphism $a \colon (\iota_R)_! \phi_! \iso (\iota_S)_! \psi_!$.
  By \cref{cor:model_f_*}, our goal thus is to compute its mate $\Lder (\iota_S)^\antishriek \Lder (\iota_R)_! \to \Lder \psi_! \Lder \phi^\antishriek$, using the adjunctions $\Lder \phi^\antishriek \dashv \Rder \phi_\antishriek \eq \Lder \phi_!$ and $\Lder (\iota_S)^\antishriek \dashv \Rder (\iota_S)_\antishriek \eq \Lder (\iota_S)_!$ of \cref{cor:extension_left_adjoint} (note that the fiber of $\pr_Y$ is connected since it is a pullback of $f$).
  
  To this end, we define the two auxiliary functors
  \begin{align*}
    (\iota_S)^\dag &\defeq (\iota_S)^* \bigl( \blank \tensor_{R \tensor_\k S} ( q_R (\dual{R}_\k) \tensor_\k S ) \bigr) \colon \Modeq{G}{R \tensor_\k S} \longto \Modeq{G}{S} \\
    (\iota_S)_\dag &\defeq \Hom_S \bigl( q_R (\dual R_\k) \tensor_\k S, \blank \bigr) \colon \Modeq{G}{S} \longto \Modeq{G}{R \tensor_\k S}
  \end{align*}
  which form an adjunction $(\iota_S)^\dag \dashv (\iota_S)_\dag$ that is Quillen since $q_R (\dual R_\k) \tensor_\k S$ is cofibrant as an $S$-module (using \cref{lemma:restrict_cofibration}).
  Note that there is a natural quasi-isomorphism $(\iota_S)_! \to (\iota_S)_\dag$ given at $M \in \Modeq{G}{S}$ by the map
  \[ M \tensor_\k R  \xlongto{\eq}  \Hom_\k \bigl( q_R(\dual R_\k), M \bigr)  \iso  \Hom_S \bigl( q_R(\dual R_\k) \tensor_\k S, M \bigr) \]
  of \cref{lemma:scalar_ext_right_adjoint} (note that $R$ is homotopically finite as an $R$-module by \cref{lemma:compact_fibs}).
  As in (the proof of) \cref{cor:extension_left_adjoint}, this yields an equivalence $\Lder (\iota_S)_! \eq \Rder (\iota_S)_\dag$.
  Moreover note that $(\iota_S)_! \to (\iota_S)_\dag$ factors as $(\iota_S)_! \to (\iota_S)_\antishriek \to (\iota_S)_\dag$, where the first map is the one of \cref{lemma:scalar_ext_right_adjoint} and the second is given by precomposition with the map $q_R (\dual{R}_\k) \tensor_\k S  \to  q_{R \tensor_\k S} \dual{(R \tensor_\k S)}_S$ of \cref{obs:pr_DS}.
  This also implies that we have a factorization $\Lder (\iota_S)_! \to \Rder (\iota_S)_\antishriek \to \Rder (\iota_S)_\dag$ through the map of \cref{cor:extension_left_adjoint}.
  
  There is a natural equivalence $b \colon (\iota_R)_! \phi_\antishriek \to (\iota_S)_\dag \psi_!$ given at $M \in \Modeq{G}{\k}$ by the quasi-isomorphism
  \[ \Hom_\k \bigl( q_R (\dual R_\k), M \bigr) \tensor_\k S  \xlongto{\eq}  \Hom_S \bigl( q_R (\dual R_\k) \tensor_\k S, M \tensor_\k S \bigr) \]
  of \cref{lemma:extend_dual} (using that $\dual R_\k$ is homotopically finite by \cref{lemma:Hom_hf}); this induces a map $\beta \colon \Lder (\iota_R)_! \Rder \phi_\antishriek \to \Rder (\iota_S)_\dag \Lder \psi_!$.
  Chasing through the definitions shows that the following left-hand square commutes
  \[
  \begin{tikzcd}
    (\iota_R)_! \phi_! \ar{rr} \dar{\iso}[swap]{a} & & (\iota_R)_! \phi_\antishriek \dar{b}[swap]{\eq} & \Lder (\iota_R)_! \Lder \phi_! \ar{rr}{\eq} \dar{\eq}[swap]{\alpha} & & \Lder (\iota_R)_! \Rder \phi_\antishriek \dar{\beta} \\
    (\iota_S)_! \psi_! \rar{\eq} & (\iota_S)_\antishriek \psi_! \rar{\eq} & (\iota_S)_\dag \psi_! & \Lder (\iota_S)_! \Lder \psi_! \rar{\eq} & \Rder (\iota_S)_\antishriek \Lder \psi_! \rar{\eq} & \Rder (\iota_S)_\dag \Lder \psi_!
  \end{tikzcd}
  \]
  which implies (via a diagram chase) that the right-hand square commutes as well.
  The desired mate of $\alpha$ is thus equivalent to the mate $\Lder (\iota_S)^\antishriek \Lder (\iota_R)_! \to \Lder \psi_! \Lder \phi^\antishriek$ of the composite
  \begin{gather*}
    \Lder (\iota_R)_! \Rder \phi_\antishriek  \xlongto{\beta}  \Rder (\iota_S)_\dag \Lder \psi_!  \eq  \Rder (\iota_S)_\antishriek \Lder \psi_!
    \intertext{which is given by the composite}
    \Lder (\iota_S)^\antishriek \Lder (\iota_R)_!  \eq  \Lder (\iota_S)^\dag \Lder (\iota_R)_!  \xlongto{\gamma}  \Lder \psi_! \Lder \phi^\antishriek
  \end{gather*}
  where $\gamma$ is the mate of $\beta$ and the equivalence $\Lder (\iota_S)^\antishriek \eq \Lder (\iota_S)^\dag$ is adjoint to the equivalence $\Rder (\iota_S)^\antishriek \eq \Rder (\iota_S)^\dag$.
  By \cref{prop:Shulman}, the map $\gamma$ is induced by the mate $c \colon (\iota_S)^\dag (\iota_R)_! \to \psi_! \phi^\antishriek$ of $b$.
  This map $c$ is, at $M \in \Modeq{G}{R}$, given by the canonical isomorphism
  \[ (\iota_S)^* \bigl( (M \tensor_\k S) \tensor_{R \tensor_\k S} ( q_R (\dual{R}_\k) \tensor_\k S ) \bigr)  \iso  \phi^* \bigl( M \tensor_R q_R (\dual{R}_\k) \bigr) \tensor_\k S \]
  of $G$-equivariant $S$-modules.
  Similarly, the equivalence $\Lder (\iota_S)^\antishriek \eq \Lder (\iota_S)^\dag$ is induced by the map $(\iota_S)^\dag \to (\iota_S)^\antishriek$ adjoint to $(\iota_S)_\antishriek \to (\iota_S)_\dag$; it is given at $M \in \Modeq{G}{R \tensor_\k S}$ by the map
  \[ M \tensor_{R \tensor_\k S} \bigl( q_R (\dual{R}_\k) \tensor_\k S \bigr)  \longto  M \tensor_{R \tensor_\k S} q_{R \tensor_\k S} \dual{(R \tensor_\k S)}_S \]
  of \cref{obs:pr_DS}.
  Combining everything, we obtain the desired statement.
\end{proof}

\subsubsection*{The dualizing spectrum}

We are now ready to deduce a rational model for the dualizing spectrum (see \cref{def:DS}).
Afterwards we provide rational models for two related equivalences, which we will need later.

\begin{lemma} \label{lemma:model_DS}
  Let $G$ be a group and $\phi \colon \k \to S$ a cofibration of $G$-equivariant homologically connected cofibrant cdgas of finite homotopical type that models a map $f \colon X \to Y$ of fiberwise nilpotent animas over $\B G$.
  Assume that the fiber of $f$ is compact and simply connected.
  Then the dualizing spectrum $\DS f \in \Sp[X]$ is modeled by the $G$-equivariant $S$-module $\dual S_\k$.
\end{lemma}

\begin{proof}
  By definition, we have $\DS f  \eq  (\pr_1)_* \Delta_! (\SS_X)$ where $\Delta \colon X \to X \times_Y X$ is the diagonal and $\pr_1 \colon X \times_Y X \to X$ the first projection.
  By \cref{obs:pullback_model_eq}, the left-hand pullback square of animas over $\B G$
  \[
  \begin{tikzcd}
    X \times_Y X \rar{\pr_2} \dar[swap]{\pr_1} & X \dar{f}  &  S \tensor_\k S & \lar[swap]{\iota_2} S \\
    X \rar{f} & Y  &  S \uar{\iota_1} & \lar[swap]{\phi} \k \uar[swap]{\phi}
  \end{tikzcd}
  \]
  is modeled by the right-hand pushout square of $G$-equivariant cdgas (which is a homotopy pushout square since $\phi$ is a cofibration).
  This also implies that $\Delta$ is modeled by the multiplication map $\mu \colon S \tensor_\k S \to S$.
  Furthermore note that the fiber of $\pr_1$ is equivalent to the fiber of $f$ and hence compact and simply connected, and that the fiber of the diagonal map $\Delta \colon X \to X \times_Y X$ is the (based) loop space of the fiber of $f$ and hence connected.
  
  Hence \cref{lemma:model_f_!,cor:model_f_*} imply that $\Delta_!$ is modeled by $\Rder \mu^* = \mu^*$ and that $(\pr_1)_*$ is modeled by $\Lder (\iota_1)^\antishriek$.
  Hence $\DS f$ is modeled by the $G$-equivariant $S$-module
  \[ \iota_1^* \bigl( \mu^*(S) \tensor_{S \tensor_\k S} q \bigl( \dual{(S \tensor_\k S)}_{S \tensor_\k \k} \bigr) \bigr)  \eq  \iota_1^* \bigl( \mu^*(S) \tensor_{S \tensor_\k S} \bigl( S \tensor_\k q (\dual S_\k) \bigr) \bigr)  \iso  q (\dual S_\k)  \eq  \dual{S}_\k \]
  where we use \cref{obs:pr_DS} for the first equivalence (using that $S$ is homotopically finite as a $\k$-module by \cref{lemma:compact_fibs}).
\end{proof}

\begin{lemma} \label{lemma:model_f_*_DS}
  In the situation of \cref{lemma:model_DS}, the equivalence $f_!(\blank \tensor_X \DS{f}) \eq f_*(\blank)$ of \cref{lemma:f_*} is modeled by the identity of $\phi^* ( q(\blank) \tensor_S q (\dual{S}_\k) )$.
\end{lemma}

\begin{proof}
  The equivalence $f_!(\blank \tensor_X \DS{f}) \eq f_*(\blank)$ is defined to be the adjoint of the map \eqref{eq:f_*}.
  Via a diagram chase involving \cref{lemma:f_*_adjunction_model,lemma:projection_model,lemma:pull-push_model}, one checks that this map is modeled by the unit of the adjunction $\Lder \phi^\antishriek \dashv \Rder \phi_\antishriek \eq \Lder \phi_!$.
  Since the unit is adjoint to the identity, this implies the claim.
\end{proof}

\begin{lemma} \label{lemma:model_DS_product}
  Let $G$ be a group, and let the following upper row be cofibrations of $G$-equivariant homologically connected cofibrant cdgas of finite homotopical type
  \[
  \begin{tikzcd}[column sep = 15, row sep = 5]
    T & \lar[swap]{\phi} R & \lar \k \rar{\psi} & S \\
    X \rar{f} & Y \rar & B & \lar[swap]{g} Z
  \end{tikzcd}
  \]
  that model the lower row of maps of fiberwise nilpotent animas over $\B G$.
  Assume that the fiber of $f$ is compact and simply connected, and that the fiber of $g$ is connected.
  Note that the left-hand pushout square of $G$-equivariant cdgas
  \[
  \begin{tikzcd}
    T \tensor_\k S & \lar[swap]{\phi \tensor \id} R \tensor_\k S  &  X \times_B Z \rar{f \times \id} \dar[swap]{\pr_1} & Y \times_B Z \dar{\pr_1} \\
    T \uar & \lar[swap]{\phi} R \uar  &  X \rar{f} & Y
  \end{tikzcd}
  \]
  models the right-hand pullback square of animas over $\B G$.
  Then the equivalence $\pr_1^*(\DS f) \eq \DS {f \times \id}$ of \cref{lemma:DS_pullback} is modeled by the quasi-isomorphism
  \[ q_T (\dual T_R) \tensor_\k S  \xlongto{\eq}  q_{T \tensor_\k S} \bigl( \dual {(T \tensor_\k S)}_{R \tensor_\k S} \bigr) \]
  of \cref{obs:pr_DS}.
\end{lemma}

\begin{proof}
  First note that, by \cref{obs:pullback_model_eq}, the left-hand square indeed models the right-hand square since the fibers of $f$ and $g$ are connected (the latter we need to see that $T \tensor_\k S$ and $R \tensor_\k S$ indeed model $X \times_B Z$ and $Y \times_B Z$, respectively).
  Furthermore note that, since $f \times \id$ is a pullback of $f$, it also has compact simply connected fibers.
  Then the claim follows from a diagram chase using \cref{lemma:pull-copush_model,lemma:pull-push_model}.
\end{proof}

\subsection{Modeling fiberwise THH}

In this subsection, we provide (equivariant) rational models for fiberwise THH and various related maps.
We begin with the classical observation that the Hochschild complex of a cdga $R$ is the left derived functor of scalar extension along the multiplication map of $R$.
(Recall the definition of the two-sided bar construction and the Hochschild complex from \cref{def:bar,def:HH}.)

\begin{notation}
  Let $\k \to R$ be a map of cdgas.
  We write $\mu_R \colon R \tensor_\k R \to R$ for the multiplication map, and $R^e \defeq \mu_R^*(R)$ for $R$ considered as an $(R \tensor_\k R)$-module (though we sometimes omit the superscript if there is no risk of confusion).
\end{notation}

\begin{observation} \label{obs:HH_derived}
  Let $G$ be a group, $\iota \colon \k \to R$ a map of $G$-equivariant cdgas, and $M$ a $G$-equivariant $(R \tensor_\k R)$-module.
  Then the two-sided bar construction $\B_\k(R, R, R)$ is a $G$-equivariant $(R \tensor_\k R)$-module, and hence the Hochschild complex $\HH_\k(R, M) \defeq \B_\k(R, R, R) \tensor_{R \tensor_\k R} M$ is a $G$-equivariant $\k$-module.
  The $(R \tensor_\k R)$-module map $\epsilon \colon \B_\k(R, R, R) \to R^e$ is $G$-equivariant and induces a natural transformation $\HH_\k(R, \blank) \to (\mu_R)_!$.
  In particular we obtain a natural $G$-equivariant map $\epsilon \colon \HH_\k(R, \mu_R^*(M)) \to M$.
  
  When $R$ is cofibrant as a $\k$-module, then $\B_\k(R, R, R)$ is cofibrant as an $(R \tensor_\k R)$-module and $\epsilon \colon \B_\k(R, R, R) \to R$ is a quasi-isomorphism by \cref{lemma:bar_complex}.
  In particular, using \cref{lemma:cofibrant_flat}, we obtain the following two natural $G$-equivariant quasi-isomorphisms of $\k$-modules
  \[ \HH_\k(R, M)  \xlongfrom{\eq}  \B_\k(R, R, R) \tensor_{R \tensor_\k R} q(M)  \xlongto{\eq}  R^e \tensor_{R \tensor_\k R} q(M)  =  \Lder (\mu_R)_! (M) \]
  which identify $\HH_\k(R, \blank)$ with $\iota^* \Lder (\mu_R)_!$.
  Under this equivalence, the counit of the adjunction $\Lder (\mu_R)_! \dashv \Rder \mu_R^*$ is identified with the map $\epsilon \colon \HH_\k(R, \mu_R^*(M)) \to M$ (using \cref{prop:Shulman}).
\end{observation}

\begin{proposition} \label{prop:THH_model}
  Let $G$ be a group and $\iota \colon \k \to R$ a cofibration of $G$-equivariant homologically connected cofibrant cdgas of finite homotopical type that models a map $p \colon X \to B$ of fiberwise nilpotent animas over $\B G$.
  Assume that the fiber of $p$ is simply connected.
  Then the $B$-spectrum $\THH_B(X)$ is nilpotent and modeled by the $G$-equivariant $\k$-module $\HH_\k(R)$.
\end{proposition}

\begin{proof}
  By \cref{lemma:transfer_diag}, we have $\THH_B(X) \eq p_! \Delta^* \Delta_! p^*(\SS_B)$ where $\Delta \colon X \to X \times_B X$ is the diagonal.
  Since $\iota$ is a cofibration, this diagonal is modeled by the multiplication map $\mu_R \colon R \tensor_\k R \to R$.
  Moreover the fiber of $\Delta$ is the (based) loop space of the fiber of $p$ and hence connected.
  Thus, by \cref{prop:BMeq,lemma:model_f_!}, the $B$-spectrum $\THH_B(X)$ is nilpotent and modeled by $\Rder \iota^* \Lder (\mu_R)_! \Rder \mu_R^* \Lder \iota_! (\k) \eq \iota^* \Lder (\mu_R)_! (R^e)$, which is equivalent to $\HH_\k(R, R)$ by \cref{obs:HH_derived}.
\end{proof}

\begin{lemma} \label{lemma:assembly_model}
  In the situation of \cref{prop:THH_model}, the fiberwise assembly map $\SS_B[X] \to \THH_B(X)$ of \cref{def:THH_assembly} is modeled by the map $\epsilon \colon \HH_\k(R) \to R$ of $G$-equivariant $\k$-modules.
\end{lemma}

\begin{proof}
  This follows from \cref{lemma:fiberwise_THH_assembly,obs:HH_derived}.
\end{proof}

\begin{lemma} \label{lemma:eval_model}
  In the situation of \cref{prop:THH_model}, the following composite of the equivalence of \cref{lemma:THH_loops} and evaluation at the basepoint of $\Sphere 1$
  \begin{equation} \label{eq:THH_eval}
    \THH_B(X) \eq \SS_B[\L_B X]  \xlongto{e}  \SS_B[X]
  \end{equation}
  is modeled by the inclusion $R \to \HH_\k(R)$.
\end{lemma}

\begin{proof}
  Recall from \cref{lemma:THH_loops} that the identification $\THH_B(X) \eq \SS_B[\L_B X]$ is given by the composite
  \[ \THH_B(X)  \eq  p_! \Delta^* \Delta_! (\SS_X)  \eq  p_! e_! e^* (\SS_X)  \eq  \SS_B[\L_B X] \]
  and note that under the final identification the map $e \colon \SS_B[\L_B X] \to \SS_B[X] = p_! (\SS_X)$ corresponds to the counit of the adjunction $e_! \dashv e^*$ by \cref{lemma:susp_map}.
  Using that mates compose and the commutative diagram
  \[
  \begin{tikzcd}
    \L_B X \rar{e} \dar[swap]{e} & X \dar{\Delta} \rar{p} & B \dar{\id} \\
    X \rar{\Delta} & X \times_B X \rar & B
  \end{tikzcd}
  \]
  we see that the map \eqref{eq:THH_eval} is equivalent to the map
  \begin{equation} \label{eq:eval_model}
    \THH_B(X)  \eq  p_! \Delta^* \Delta_! (\SS_X)  \eq  p_! \Delta^* \Delta_! \Delta^* (\SS_{X \times_B X})  \xlongto{\epsilon}  p_! \Delta^* (\SS_{X \times_B X})  \eq  p_! (\SS_X)
  \end{equation}
  given by the counit of the adjunction $\Delta_! \dashv \Delta^*$.
  
  By \cref{prop:BMeq}, the equivalence $\Delta^* (\SS_{X \times_B X}) \eq \SS_X$ is modeled by the canonical equivalence $\Lder (\mu_R)_! (R \tensor_\k R) \eq (\mu_R)_! (R \tensor_\k R) \iso R$.
  By \cref{prop:Shulman}, the counit $\Delta_! \Delta^* (\SS_{X \times_B X}) \to \SS_{X \times_B X}$ is modeled by the unit $\eta \colon R \tensor_\k R \to \mu_R^* (\mu_R)_! (R \tensor_\k R) \iso R^e$, which is given by $\mu_R$.
  Hence the map \eqref{eq:eval_model} is modeled by the upper composite in the commutative diagram
  \[
  \begin{tikzcd}
    \iota^* \Lder (\mu_R)_! (\mu_R)^* (R) & \lar[swap]{\iso} \iota^* \Lder (\mu_R)_! (\mu_R)^* (\mu_R)_! (R \tensor_\k R) & \lar[swap]{\eta} \iota^* \Lder (\mu_R)_! (R \tensor_\k R) \rar{\eq} & \iota^*(R) \\
    \HH_\k(R, R^e) \uar{\eq} & & \ar{ll} \HH_\k(R, R \tensor_\k R) \uar[swap]{\eq} \urar[bend right, start anchor = east]
  \end{tikzcd}
  \]
  in $\Underlying(\Modeq{G}{\k})$.
  The right-hand diagonal map is given by projection onto Hochschild degree $0$ and then multiplying.
  Noting that it has a $G$-equivariant section given by $r \mapsto r \tensor 1$ completes the proof.
\end{proof}

\subsection{Modeling the fiberwise THH transfer}

We are finally ready to prove the main result of the paper: that the fiberwise THH transfer is modeled by the Hochschild homology transfer.
We begin by recalling the Hochschild homology transfer; see e.g.\ Keller \cite[Theorem~5.1]{Kel21}.

\begin{construction} \label{con:HH_transfer}
  Let $G$ be a group, let $\k \to R$ and $\k \to S$ be maps of $G$-equivariant cdgas such that $R$ and $S$ are cofibrant as $\k$-modules, and let $M$ be an $(S \tensor_\k R)$-module that is cofibrant and homotopically finite as an $R$-module.
  Then the \emph{Hochschild homology transfer} $\transfer{M} \colon \HH_\k(S) \to \HH_\k(R)$ is the map of $\Underlying(\Modeq{G}{\k})$ defined to be zig-zag
  \[
  \begin{tikzcd}[row sep = 15]
    \HH_\k(S, S) \rar{\nu_*} & \HH_\k \bigl( S, \Hom_R(M, M) \bigr) & \lar[swap]{\eq} \HH_\k \bigl( S, M \tensor_R \B_\k(R, R, R) \tensor_R \dual{M}_R \bigr) \dar{\iso} \\
    \HH_\k(R, R) & \lar[swap]{\ev_*} \HH_\k \bigl( R, \dual{M}_R \tensor_S M \bigr) & \lar \HH_\k \bigl( R, \dual{M}_R \tensor_S \B_\k(S, S, S) \tensor_S M \bigr)
  \end{tikzcd}
  \]
  where $\nu$ is the map $S \to \Hom_R(M, M)$ of $(S \tensor_\k S)$-modules determined by $\nu(1) = \id$, the vertical isomorphism is given by a cyclic permutation, and the quasi-isomorphism is induced by the composite
  \[ M \tensor_R \B_\k(R, R, R) \tensor_R \dual{M}_R  \xlongto{\eq}  M \tensor_R \dual{M}_R  \xlongto{\eq}  \Hom_R(M, M) \]
  where the first map is a quasi-isomorphism by \cref{lemma:bar_complex} and the second by \cref{lemma:Hom_qiso} below.
  In the case that $\k \to S$ factors through a map $\phi \colon R \to S$, and $M$ is equal to $S$, we also write $\phi^*$ for $\transfer{S}$.
\end{construction}

\begin{lemma} \label{lemma:Hom_qiso}
  Let $G$ be a group, $R$ a $G$-equivariant cdga, and $M$ and $N$ two $G$-equivariant cofibrant $R$-modules such that $M$ is homotopically finite.
  Then the map
  \begin{align*}
    N \tensor_R \dual{M}_R  &\longto  \Hom_R(M, N) \\
    n \tensor \phi  &\longmapsto  n \cdot \phi(\blank)
  \end{align*}
  is a quasi-isomorphism of $G$-equivariant $R$-modules.
\end{lemma}

\begin{proof}
  It follows from the definitions that the map is $G$-equivariant.
  To prove that it is a quasi-isomorphism, let $M' \defeq R \tensor V$ be the minimal model of $M$.
  We obtain a commutative diagram
  \[
  \begin{tikzcd}
    \Hom_R(M, R) \tensor_R N \rar \dar[swap]{\eq} & \Hom_R(M, N) \dar{\eq} \\
    \Hom_R(M', R) \tensor_R N \rar{\iso} & \Hom_R(M', N)
  \end{tikzcd}
  \]
  where the left-hand vertical map is a quasi-isomorphism by \cref{lemma:cofibrant_flat} since $N$ is cofibrant.
  The bottom horizontal composite is an isomorphism since $M'$ is a finite-dimensional quasi-free $R$-module.
  This implies the claim.
\end{proof}

Note that by (the same argument as in) \cref{rem:eq_models}, any two maps of connected spaces $X \to Y \to B$ with simply connected fibers admit $\pi_1(B)$-equivariant cdga models, as required by the following theorem (as long as their universal coverings are of finite rational type).

\begin{theorem} \label{thm:transfer_model_cdga}
  Let $G$ be a group, and let the following two left-hand maps be cofibrations of $G$-equivariant homologically connected cofibrant cdgas of finite homotopical type
  \[ S \xlongfrom{\phi} R \xlongfrom{\iota_R} \k  \qquad  X \xlongto{f} Y \xlongto{p_Y} B \]
  that model the two right-hand maps of fiberwise nilpotent animas over $\B G$.
  Assume that the fiber of $p_Y$ is simply connected and that the fiber of $f$ is simply connected and compact.
  Then the fiberwise THH transfer $f^*  \colon  \THH_B(Y)  \to  \THH_B(X)$ is modeled by the Hochschild homology transfer $\transfer{S} \colon \HH_\k(S) \to \HH_\k(R)$ as a map of $\Underlying(\Modeq{G}{\k})$.
\end{theorem}

\begin{proof}
  By \cref{lemma:transfer_diag}, the transfer $f^* \colon \THH_B(Y) \to \THH_B(X)$ is given by evaluating the following pasting at $\SS_B$
  \begin{equation} \label{eq:model_transfer_diag}
  \begin{tikzcd}[sep = 17]
    \Sp[B] \rar{Y^*} \dar[swap]{X^*} &[-10] \Sp[Y] \rar{\Delta_!} & \Sp[Y \times_B Y] \dar[swap]{({\id} \times f)^*} \rar[equals] &[-10] \Sp[Y \times_B Y] \ar{dd}{(f \times \id)^*} \ar[bend left = 15]{ddrr}{\id}[swap, name=U]{} & &[-10] \\
    \Sp[X] \dar[swap]{\Delta_!} \ar[phantom]{rr}{\text{\eqref{eq:transfer_diag_1}}} & & \Sp[Y \times_B X] \dar[swap]{\tensor \pr_2^*(\DS{f})} & & & \\
    \Sp[X \times_B X] \ar{rr}{(f \times \id)_*} \ar[bend right = 15]{ddrr}[name=C]{}[swap]{\id} & & \Sp[Y \times_B X] \ar{dd}[swap]{(f \times \id)^*} \ar[to=C, Rightarrow, shorten = 1em, "\epsilon"'] & \Sp[X \times_B Y] \dar{({\id} \times f)^*} \ar{rr}[swap]{(f \times \id)_*} \ar[from=U, Rightarrow, shorten = 1em, "\eta"'] & & \Sp[Y \times_B Y] \dar{\Delta^*} \\
    & & & \Sp[X \times_B X] \dar{\tensor \pr_2^*(\DS{f})} \ar[phantom]{rr}{\text{\eqref{eq:transfer_diag_2}}} & & \Sp[Y] \dar{Y_!} \\
    & & \Sp[X \times_B X] \rar[equals] & \Sp[X \times_B X] \rar{\Delta^*} & \Sp[X] \rar{X_!} & \Sp[B]
  \end{tikzcd}
  \end{equation}
  where $\eta$ and $\epsilon$ are the respective (co)unit.
  We will now proceed to model all natural transformations occurring in this diagram, starting from the bottom left.
  
  In the following, all tensor products are over $\k$ unless otherwise specified.
  We will also write $\iota_S \defeq \phi \after \iota_R$.
  Note that the $G$-equivariant cdgas $R \tensor R$, $R \tensor S$, $S \tensor R$, and $S \tensor S$ are cofibrant, and that they respectively model $Y \times_B Y$, $Y \times_B X$, $X \times_B Y$, and $X \times_B X$ as fiberwise nilpotent animas over $\B G$ by \cref{obs:pullback_model_eq}.
  We will use throughout that since $R$ and $S$ are cofibrant $\k$-modules and $S$ is a cofibrant $R$-module, the functors $\phi_!$, $(\iota_R)_!$, and $(\iota_S)_!$ preserve weak equivalences, and hence are equivalent to their left derived functors; since $q (\dual{S}_R)$ is cofibrant, this is also true for $\phi^\antishriek$.
  The same argument applies to $(\phi \tensor \id)_!$ and similar variants.
  
  By \cref{lemma:f_*_adjunction_model}, the counit
  \[ \Delta_! X^*  \xlongfrom{\epsilon}  (f \times \id_X)^* (f \times \id_X)_* \Delta_! X^* \]
  is modeled by the zig-zag
  \begin{equation} \label{eq:model_diag_counit}
    \mu_S^* (\iota_S)_! (\blank) \xlongto{\eta}  (\phi \tensor \id_S)_\antishriek (\phi \tensor \id_S)^\antishriek \mu_S^* (\iota_S)_! (\blank)  \xlongfrom{\eq}  (\phi \tensor \id_S)_! (\phi \tensor \id_S)^\antishriek \mu_S^* (\iota_S)_! (\blank)
  \end{equation}
  where $\eta$ is the unit and the quasi-isomorphism is the one of \cref{lemma:scalar_ext_right_adjoint}.
  
  The following left-hand equivalences of \eqref{eq:transfer_diag_1} are modeled by the right-hand (quasi-)\-iso\-mor\-phisms
  \begin{equation} \label{eq:model_diag_1}
  \begin{aligned}
    &\mathrel{\phantom{\eq}} (f \times \id_X)_* \Delta_! X^* (\blank) & &\mathrel{\phantom{=}} (\phi \tensor \id_S)^\antishriek \mu_S^* (\iota_S)_! (\blank) \\
    &\eq (f \times \id_X)_! \bigl( \Delta_! X^* (\blank) \tensor \DS{f \times \id} \bigr) & &= (\phi \tensor \id_S)^* \bigl( \mu_S^* (\iota_S)_! (\blank)  \tensor_{S \tensor S}  q ( \dual{(S \tensor S)}_{R \tensor S} ) \bigr) \\
    &\eq (f \times \id_X)_! \bigl( \Delta_! X^* (\blank) \tensor \pr_1^* \DS{f} \bigr) & &\xfrom{\eq} (\phi \tensor \id_S)^* \bigl( \mu_S^* (\iota_S)_! (\blank) \tensor_{S \tensor S} ( q (\dual S_R) \tensor S) \bigr) \\
    &\eq (f \times \id_X)_! \Delta_! \bigl( X^* (\blank) \tensor \Delta^* \pr_1^*(\DS f) \bigr) & &\iso (\phi \tensor \id_S)^* \mu_S^* \bigl( (\iota_S)_! (\blank) \tensor_S (\mu_S)_! ( q (\dual S_R) \tensor S) \bigr) \\
    &\eq (f, \id_X)_! \bigl( X^* (\blank) \tensor \DS f \bigr) & &\iso \mu_{R,S}^* \bigl( (\iota_S)_! (\blank) \tensor_S q (\dual S_R) \bigr) \\
    &\eq (f, \id_X)_! \bigl( f^* Y^* (\blank) \tensor (f, \id_X)^* \pr_2^*(\DS f) \bigr) & &\iso \mu_{R,S}^* \bigl( \phi_! (\iota_R)_! (\blank) \tensor_S (\mu_{R, S})_!(R \tensor q (\dual S_R)) \bigr) \\
    &\eq (f, \id_X)_! f^* Y^* (\blank) \tensor \pr_2^*(\DS f) & &\iso \mu_{R,S}^* \phi_! (\iota_R)_! (\blank) \tensor_{R \tensor S} \bigl( R \tensor q (\dual S_R) \bigr) \\
    &\eq ({\id_Y} \times f)^* \Delta_! Y^*(\blank) \tensor \pr_2^*(\DS f) & &\iso ({\id_R} \tensor \phi)_! \mu_R^* (\iota_R)_! (\blank) \tensor_{R \tensor S} \bigl( R \tensor q (\dual S_R) \bigr)
  \end{aligned}
  \end{equation}
  where $\mu_{R,S} \colon R \tensor S \to S$ is the multiplication map, and we use, in order, \cref{lemma:model_f_*_DS}, \cref{lemma:model_DS_product}, \cref{lemma:projection_model}, \cref{lemma:model_composition} and its adjoint, \cref{lemma:model_composition} again, \cref{lemma:projection_model} again, and finally \cref{lemma:pull-push_model}.
  
  Composing \eqref{eq:model_diag_counit} with the first (quasi-)isomorphisms of \eqref{eq:model_diag_1} yields the upper and right-hand part of the commutative diagram
  \[
  \begin{tikzcd}[column sep = 17]
    \mu_S^* (\iota_S)_! (\blank) \rar{\eta} & (\phi \tensor \id_S)_\antishriek (\phi \tensor \id_S)^\antishriek \mu_S^* (\iota_S)_! (\blank) & \lar[swap]{\eq} (\phi \tensor \id_S)_! (\phi \tensor \id_S)^\antishriek \mu_S^* (\iota_S)_! (\blank) \\
    & (\phi \tensor \id_S)_\antishriek \mu_{R,S}^* \bigl( (\iota_S)_! (\blank) \tensor_S q (\dual S_R) \bigr) \uar{\eq} & \lar[swap]{\eq} (\phi \tensor \id_S)_! \mu_{R,S}^* \bigl( (\iota_S)_! (\blank) \tensor_S q (\dual S_R) \bigr) \uar[swap]{\eq}
  \end{tikzcd}
  \]
  whose left-hand and bottom part, when evaluated on $\k \in \Modeq{G}{\k}$, yields the top and right-hand sides of the following commutative diagram of $G$-equivariant $(S \tensor S)$-modules
  \[
  \begin{tikzcd}
    S^e \ar{ddd}[swap]{\alpha} \ar{rr}{\eta} &[-70] &[-130] \dlar[start anchor = -172][swap]{\eq} \Hom_{R \tensor S} \bigl( q (\dual {(S \tensor S)}_{R \tensor S}), S^e  \tensor_{S \tensor S}  q (\dual{(S \tensor S)}_{R \tensor S}) \bigr) \\
    & \Hom_{R \tensor S} \bigl( q (\dual S_R) \tensor S, S^e  \tensor_{S \tensor S}  q (\dual{(S \tensor S)}_{R \tensor S}) \bigr) & \\
    & \Hom_{R \tensor S} \bigl( q (\dual S_R) \tensor S, \mu_{R,S}^* q (\dual S_R) \bigr) \uar{\eq} & \\
    \Hom_R \bigl( q (\dual S_R), q (\dual S_R) \bigr) \ar[start anchor = north east]{ur}[swap]{\iso} & & \ar[start anchor = 172]{ul}{\eq} \Hom_{R \tensor S} \bigl( q (\dual {(S \tensor S)}_{R \tensor S}), \mu_{R,S}^* q (\dual S_R) \bigr) \ar{uuu}[swap]{\eq} \\
    S \tensor_R q(\dual S_R) \uar{\eq} \ar{rr}{\iso} & & \mu_{R,S}^* q (\dual S_R) \tensor_{R \tensor S} (S \tensor S) \uar[swap]{\eq}
  \end{tikzcd}
  \]
  where $\alpha$ is the map determined by $1 \mapsto \id$, and the two bottom vertical quasi-isomorphisms are obtained from \cref{lemma:scalar_ext_right_adjoint}.
  Hence the left-hand half of the pasting \eqref{eq:model_transfer_diag}, evaluated on $\SS_B$, is modeled by the left-hand side of the commutative diagram
  \begin{equation} \label{eq:model_diag_left}
  \begin{tikzcd}
    & \dlar[bend right = 10][swap]{\alpha} S^e \drar[bend left = 10]{\beta} & \\
    \Hom_R \bigl( q (\dual S_R), q (\dual S_R) \bigr) \rar{\eq} & \Hom_R \bigl( q (\dual S_R), \dual S_R \bigr) \iso \Hom_R \bigl( S, \dual {(q (\dual S_R))}_R \bigr) & \lar[swap]{\eq} \Hom_R(S, S) \\
    S \tensor_R q (\dual S_R) \uar{\eq} \ar{rr}{\eq} & & S \tensor_R \dual S_R \uar
  \end{tikzcd}
  \end{equation}
  and thus it is also modeled by the right-hand side, where $\beta$ is again determined by $1 \mapsto \id$.
  
  By \cref{lemma:model_composition,lemma:pullback_monoidal_model}, the following upper equivalence making the middle square of \eqref{eq:model_transfer_diag} commute
  \begin{equation} \label{eq:model_diag_middle}
  \begin{aligned}
    (f \times \id_X)^* \bigl( (\id_Y \times f)^* (\blank) \tensor \pr_2^* (\DS{f}) \bigr)  &\eq  (\id_X \times f)^* (f \times \id_Y)^* (\blank) \tensor \pr_2^* (\DS{f}) \\
    (\phi \tensor \id_S)_! \bigl( ({\id_R} \tensor \phi)_! (\blank) \tensor_{R \tensor S} (R \tensor q (\dual S_R)) \bigr)  &\iso  ({\id_S} \tensor \phi)_! (\phi \tensor \id_R)_! (\blank) \tensor_{S \tensor S} \bigl( S \tensor q (\dual S_R) \bigr)
  \end{aligned}
  \end{equation}
  is modeled by the canonical lower isomorphism of functors $\Modeq{G}{R \tensor R} \to \Modeq{G}{S \tensor S}$.
  
  On a cofibrant object of $\Modeq{G}{S \tensor R}$, the left-hand equivalences of \eqref{eq:transfer_diag_2} are modeled by the right-hand (quasi-)isomorphisms
  \begin{equation} \label{eq:model_diag_2}
  \begin{aligned}
    &\mathrel{\phantom{\eq}} X_! \Delta^* \bigl( ({\id_X} \times f)^* (\blank) \tensor \pr_2^*(\DS f) \bigr) & &\mathrel{\phantom{\iso}} \iota_S^* (\mu_S)_! \bigl( ({\id_S} \tensor \phi)_! (\blank) \tensor_{S \tensor S} (S \tensor q (\dual S_R)) \bigr) \\
    &\eq X_! \bigl( \Delta^* ({\id_X} \times f)^* (\blank) \tensor \Delta^* \pr_2^*(\DS f) \bigr) & &\iso \iota_S^* \bigl( (\mu_S)_! ({\id_S} \tensor \phi)_! (\blank) \tensor_{S} (\mu_S)_! (S \tensor q (\dual S_R)) \bigr) \\
    &\eq Y_! f_! \bigl( (\id_X, f)^* (\blank) \tensor (\id_X, f)^* \pr_1^*(\DS f) \bigr) & &\iso \iota_R^* \phi^* \bigl( (\mu_{S,R})_! (\blank) \tensor_{S} (\mu_{S,R})_! (q (\dual S_R) \tensor R) \bigr) \\
    &\eq Y_! f_! (\id_X, f)^* (\blank \tensor \pr_1^*(\DS f)) & &\iso \iota_R^* \phi^* (\mu_{S,R})_! \bigl( \blank \tensor_{S \tensor R} (q (\dual S_R) \tensor R) \bigr) \\
    &\eq Y_! \Delta^* (f \times \id_Y)_! (\blank \tensor \pr_1^*(\DS f)) & &\iso \iota_R^* (\mu_R)_! (\phi \tensor \id_R)^* \bigl( \blank \tensor_{S \tensor R} (q (\dual S_R) \tensor R) \bigr) \\
    &\eq Y_! \Delta^* (f \times \id_Y)_! (\blank \tensor \DS {f \times \id}) & &\xto{\eq} \iota_R^* (\mu_R)_! (\phi \tensor \id_R)^* \bigl( \blank \tensor_{S \tensor R} q (\dual {(S \tensor R)}_{R \tensor R}) \bigr) \\
    &\eq Y_! \Delta^* (f \times \id_Y)_* (\blank) & &= \iota_R^* (\mu_R)_! (\phi \tensor \id_R)^\antishriek (\blank)
  \end{aligned}
  \end{equation}
  where $\mu_{S,R} \colon S \tensor R \to S$ is the multiplication map, and we use, in order, \cref{lemma:pullback_monoidal_model}, \cref{lemma:model_composition} and its adjoint, \cref{lemma:pullback_monoidal_model} again, \cref{lemma:pull-push_model}, \cref{lemma:model_DS_product}, and finally \cref{lemma:model_f_*_DS}.
  
  By \cref{lemma:f_*_adjunction_model}, the unit
  \begin{align*}
    (f \times \id_Y)_* (f \times \id_Y)^*  \xlongfrom{\eta}  \id
    \shortintertext{is modeled by the map}
    (\phi \tensor \id_R)^\antishriek \Lder (\phi \tensor \id_R)_!  \xlongto{\epsilon}  q
  \end{align*}
  given by the left-hand composite in the commutative diagram
  \begin{equation} \label{eq:model_diag_a}
    \begin{tikzcd}
      q(\blank) \tensor_{R \tensor R} (S \tensor R) \tensor_{S \tensor R} q \bigl( \dual{(S \tensor R)}_{R \tensor R} \bigr) \dar & \lar[swap]{\eq} q(\blank) \tensor_{R \tensor R} (S \tensor R) \tensor_{S \tensor R} \bigl( q (\dual{S}_R) \tensor R \bigr) \dar \\
      q(\blank) \tensor_{R \tensor R} (S \tensor R) \tensor_{S \tensor R} \dual{(S \tensor R)}_{R \tensor R} \dar[swap]{\ev} & \lar[swap]{\iso} q(\blank) \tensor_{R \tensor R} (S \tensor R) \tensor_{S \tensor R} ( \dual{S}_R \tensor R ) \dar{\ev} \\
      q(\blank) & \lar[equal] q(\blank)
    \end{tikzcd}
  \end{equation}
  of endofunctors of $\Modeq{G}{R \tensor R}$.
  Hence applying the final equivalence of \eqref{eq:model_diag_2} to $\Lder (\phi \tensor \id_R)_! (\blank)$ and composing with $\epsilon$ yields $\iota_R^* (\mu_R)_!$ applied to the right-hand composite of \eqref{eq:model_diag_a}.
  
  Composing the last isomorphisms of \eqref{eq:model_diag_1} with the isomorphism of \eqref{eq:model_diag_middle}, and evaluating on $\k \in \Modeq{G}{\k}$ yields the following isomorphism of $G$-equivariant $(S \tensor S)$-modules
  \[ S \tensor_R q (\dual S_R)  \iso  S \tensor_R R^e \tensor_R q (\dual S_R) \]
  where we consider $R^e$ as an $R$-$R$-bimodule.
  Writing $\B_\k(A)$ for the two-sided bar construction $\B_\k(A, A, A)$, we can replace $R^e$ by $\B_\k(R)$ by \cref{lemma:bar_complex}.
  Evaluating the first isomorphisms of \eqref{eq:model_diag_2} on the result, and composing this with the right-hand composite of \eqref{eq:model_diag_a}, we obtain the topmost two maps of the commutative diagram of $\k$-modules
  \begin{equation} \label{eq:HH_rotate}
  \begin{tikzcd}[column sep = 17]
    & \dlar[bend right = 20, start anchor = west][swap]{\iso} \iota_R^* (\mu_R)_! \bigl( q (\dual S_R) \tensor_S S \tensor_R \B_\k(R) \bigr) \rar{\ev} & \iota_R^* (\mu_R)_! \bigl( \B_\k(R) \bigr) \\
    \iota_S^* (\mu_S)_! \bigl( S \tensor_R \B_\k(R) \tensor_R q (\dual S_R) \bigr) & \HH_\k \bigl(R, q(\dual S_R) \tensor_S S \bigr) \uar[swap]{\iso} \rar{\ev} & \HH_\k (R, R) \uar[swap]{\iso} \\
    \HH_\k \bigl( S, S \tensor_R \B_\k(R) \tensor_R q (\dual S_R) \bigr) \uar{\eq} \dar[swap]{\eq} & \lar[swap]{\iso} \HH_\k \bigl(R, q(\dual S_R) \tensor_S \B_\k(S) \tensor_S S \bigr) \uar[swap]{\eq} \dar{\eq} & \HH_\k \bigl( R, \dual S_R \tensor_S S \bigr) \uar[swap]{\ev} \\
    \HH_\k \bigl( S, S \tensor_R \B_\k(R) \tensor_R \dual S_R \bigr) \dar[swap]{\eq} & \lar[swap]{\iso} \HH_\k \bigl( R, \dual S_R \tensor_S \B_\k(S) \tensor_S S \bigr) \urar[bend right = 20, start anchor = east][swap]{\eq} & \\
    \HH_\k \bigl( S, S \tensor_R \dual S_R \bigr)
  \end{tikzcd}
  \end{equation}
  where we use \cref{obs:HH_derived} to identify $\iota_S^* (\mu_S)_! (\blank)$ with $\HH_\k(S, \blank)$ for the cofibrant $(S \tensor S)$-module $S \tensor_R \B_\k(R) \tensor_R q (\dual S_R)$.
  Combining this diagram with \eqref{eq:model_diag_left} yields the desired result.
\end{proof}

The following proposition provides an alternative description of the Hochschild homology transfer, which is sometimes useful.

\begin{proposition} \label{prop:HH_transfer_alt}
  Let $G$ be a group, and let $\k \to R$ and $\phi \colon R \to S$ be cofibrations of $G$-equivariant cdgas such that $S$ is homotopically finite as an $R$-module.
  Then the Hochschild homology transfer $\transfer{S}$ of \cref{con:HH_transfer} is homotopic to the composite
  \[ \HH_\k(S)  \xlongto{\gamma_*}  \HH_\k \bigl( S, \Hom_R(S, S) \bigr)  \xlongfrom{\eq}  \HH_\k \bigl( R, \dual S_R \bigr)  \xlongto{(\ev_1)_*}  \HH_\k(R) \]
  in $\Underlying(\Modeq{G}{\k})$.
  The quasi-isomorphism is induced by $\phi$ and the map $\phi_* \colon \Hom_R(S, R) \to \Hom_R(S, S)$.
  The map $\ev_1 \colon \Hom_R(S, R) \to R$ is given by evaluation at $1 \in S$.
\end{proposition}

\begin{proof}
  By \eqref{eq:HH_rotate}, the map $\HH_\k ( S, S \tensor_R \dual S_R ) \to \HH_\k ( R )$ in $\Underlying(\Modeq{G}{\k})$ is homotopic to the upper composite in the commutative diagram
  \[
  \begin{tikzcd}
    & \dlar[bend right = 20, start anchor = west][swap]{\iso} \iota_R^* (\mu_R)_! \bigl( q (\dual S_R) \tensor_R \B_\k(R) \bigr) \rar{\ev_1} \dar{\iso} & \iota_R^* (\mu_R)_! \bigl( \B_\k(R) \bigr) \dar[equal] \\
    \iota_S^* (\mu_S)_! \bigl( S \tensor_R \B_\k(R) \tensor_R q (\dual S_R) \bigr) & \lar[swap]{\iso} \iota_R^* (\mu_R)_! \bigl( \B_\k(R) \tensor_R q (\dual S_R) \bigr) \rar{\ev_1} & \iota_R^* (\mu_R)_! \bigl( \B_\k(R) \bigr) \\
    \HH_\k \bigl( S, S \tensor_R \B_\k(R) \tensor_R q (\dual S_R) \bigr) \dar[swap]{\eq} \uar{\eq} & \lar \HH_\k \bigl(R, \B_\k(R) \tensor_R q (\dual S_R) \bigr) \uar[swap]{\eq} \dar{\eq} \rar{\ev_1} & \HH_\k \bigl( R, \B_\k(R) \bigr) \dar{\eq} \uar[swap]{\eq} \\
    \HH_\k \bigl( S, S \tensor_R \dual S_R \bigr) & \lar \HH_\k \bigl(R, \dual S_R \bigr) \rar{\ev_1} & \HH_\k ( R )
  \end{tikzcd}
  \]
  where we use \cref{obs:HH_derived} to identify $\iota_S^* (\mu_S)_! (M)$ with $\HH_\k(S, M)$ for a cofibrant $(S \tensor S)$-module $M$, and similarly for $R$.
  The left-hand horizontal maps are all induced by maps of $(R \tensor R)$-modules of the form $x \mapsto 1 \tensor x$.
\end{proof}

To conclude this section, we include an example showing in a simple case how to use our rational model for the THH transfer for a hands-on calculation.
It recovers an example of Lind--Malkiewich \cite[Corollary~9.4]{LM} rationally.

\begin{example}[Lind--Malkiewich] \label{ex:cdga}
  We consider the map $f \colon \B \Sphere 1 \to \B \Sphere 3$ induced by the canonical map of topological groups $\Sphere 1 \to \Sphere 3$.
  The homotopy fiber of $f$ is given by the homotopy orbits of $\Sphere 1$ acting on $\Sphere 3$; the action is free, so the homotopy orbits are given by the strict orbits, which is $\Sphere 2$.
  Rationally both $\B \Sphere 1$ and $\B \Sphere 3$ are Eilenberg--MacLane spaces, and using the Serre spectral sequence for $f$, we see that a generator of $\Coho * (\B \Sphere 3; \QQ)$ pulls back to the square of a generator of $\Coho * (\B \Sphere 1; \QQ)$.
  Hence $f$ is modeled by the map of cdgas $\QQ[x] \to \QQ[y]$ given by $x \mapsto y^2$, where $\deg x = 4$, $\deg y = 2$, and we equip both algebras with the trivial differential.
  Writing $R \defeq \QQ[x]$ and $S \defeq \QQ[y]$, we have $S = \quot{R[y]}{(y^2 - x)}$ as an $R$-algebra; as an $R$-module it is freely generated by $1$ and $y$.
  
  By \cref{thm:transfer_model_cdga,prop:HH_transfer_alt}, the $\THH$ transfer of $f$ is modeled by the zig-zag
  \begin{equation} \label{eq:ex_transfer}
    \HH_\QQ(S, S)  \longto  \HH_\QQ \bigl( S, \Hom_R(S, S) \bigr)  \xlongfrom{\eq}  \HH_\QQ \bigl( R, \dual S_R \bigr)  \xlongto{(\ev_1)_*}  \HH_\QQ(R, R)
  \end{equation}
  of maps of cochain complexes.
  For a graded vector space $V$, there is a quasi-isomorphism
  \begin{align*}
    \SA V \tensor \SA \shift V  &\xlongto{\eq}  \HH_\QQ \bigl( \SA V, \SA V \bigr) \\
    a \tensor (\shift v_1 \wedge \dots \wedge \shift v_n)  &\longmapsto  \sum_{\sigma \in \Sigma_n} \sgn(\sigma) a \tensor \shift v_{\sigma^{-1}(1)} \tensor \dots \tensor \shift v_{\sigma^{-1}(n)}
  \end{align*}
  where the domain is equipped with the trivial differential (see e.g.\ \cite[Theorem~3.2.2]{Lod}).
  Hence a basis for the homology of $\HH_\QQ(S, S)$ is given by the elements of the form $y^i$ and $y^i \tensor \shift y$ for $i \ge 0$, and similarly for $\HH_\QQ(R, R)$.
  
  Note that $\dual S_R$ is freely generated as an $R$-module by the basis $(\dual 1, \dual y)$ dual to the basis $(1, y)$ of $S$.
  In this basis the (right) $S$-module structure is given by $\dual 1 \cdot y = x \dual y$ and $\dual y \cdot y = \dual 1$.
  Identifying $\Hom_R(S, S) \iso S \tensor_R \dual S_R$, the first map of \eqref{eq:ex_transfer} is given by
  \[ y^i  \longmapsto  y^i \tensor \dual 1 + y^{i+1} \tensor \dual y  \qquad \text{and} \qquad  y^i \tensor \shift y  \longmapsto  (y^i \tensor \dual 1 + y^{i+1} \tensor \dual y) \tensor \shift y \]
  and we want to exhibit these images as being homologous to elements in the image of the second map.
  To this end we observe
  \begin{align*}
    d \bigl( (y^i \tensor \dual y) \tensor y \bigr) &= y^i \tensor \dual 1 - y^{i+1} \tensor \dual y \\
    d \bigl( (y^i \otimes \dual y) \otimes \shift y \otimes \shift y \bigr) &= (y^i \otimes \dual 1) \otimes \shift y - (y^i \otimes \dual y) \otimes \shift y^2 + (y^{i+1} \otimes \dual y) \otimes \shift y
  \end{align*}
  to see that the image of $y^i$ is homologous to $2 y^i \tensor \dual 1$ and $2 y^{i+1} \tensor \dual y$, and that the image of $y^i \tensor \shift y$ is homologous to $(y^i \otimes \dual y) \otimes \shift y^2$.
  If $i$ is odd, we instead use that
  \begin{align*}
    d \bigl( (y^{i - 1} \otimes \dual 1) \otimes \shift y \otimes \shift y \bigr)  &=  (y^{i - 1} \tensor x \dual y) \tensor \shift y - (y^{i-1} \tensor \dual 1) \tensor \shift y^2 + (y^i \otimes \dual 1) \otimes \shift y \\
    &=  (y^{i + 1} \tensor \dual y) \tensor \shift y - (y^{i-1} \tensor \dual 1) \tensor \shift y^2 + (y^i \otimes \dual 1) \otimes \shift y
  \end{align*}
  to see that the image of $y^i \tensor \shift y$ is homologous to $(y^{i - 1} \otimes \dual 1) \otimes \shift y^2$.
  
  Taken together, we see that on homology the composite \eqref{eq:ex_transfer} induces the map determined by
  \[ [y^i]  \longmapsto  \begin{cases*} [2 x^{\frac i 2}], & if $i$ even \\ 0, & if $i$ odd \end{cases*}  \qquad \text{and} \qquad  [y^i \tensor \shift y]  \longmapsto  \begin{cases*} 0, & if $i$ even \\ [x^{\frac {i - 1} 2} \tensor \shift x], & if $i$ odd \end{cases*} \]
  which recovers \cite[Corollary~9.4]{LM} rationally.
\end{example}

\begin{appendices}

\crefalias{section}{appsec}

\section{Groupoidal limits of model categories} \label{sec:sections}

In this appendix, we show that a groupoidal limit of model categories is a model for the limit of the underlying $\infty$-categories.
We deduce this from work of Harpaz \cite{Har}, who showed this for simplicial model categories, and Dugger \cite{Dug}, who showed that (under some conditions) any model category is Quillen equivalent to a simplicial model category.

\subsection{Limits of \texorpdfstring{$\infty$}{infinity}-categories}

We begin by recalling that the limit of a diagram $F \colon \cat C \to \Catinf$ can be computed as the cocartesian sections of the associated cocartesian fibration over $\cat C$ (see \cite[Corollary~3.3.3.2]{LurHTT}).
We will require that this equivalence is natural in $F$.
As we could not find a reference for this fact in the literature, we provide a proof.

\begin{notation}
  Let $\cat C$ be a simplicial set.
  We denote by $\Fibmark{\cat C}$ the category of marked simplicial sets over $\cat C$ equipped with the cocartesian model structure (see \cite[Remark~3.1.3.9]{LurHTT}) and write $\Cocart{\cat C} \defeq \Underlying(\Fibmark{\cat C})$.
  We omit the subscripts when $\cat C = *$.
  For $\cat E \in \Fibmark{\cat C}$, we furthermore write $\Mapcc{\cat C}(\cat C, \cat E) \subseteq \Map_{\cat C}(\cat C, \cat E)$ for the simplicial subset spanned by those maps $\cat C \to \cat E$ that land in the marked edges of $\cat E$.
  By \cite[Remark~3.1.4.5]{LurHTT}, the functor $\Mapcc{\cat C}(\cat C, \blank) \colon \Fibmark{\cat C} \to \sSet$ is right Quillen when we equip $\sSet$ with the Joyal model structure.
  We write $\Funcc{\cat C}(\cat C, \blank) \colon \Cocart{\cat C} \to \Catinf$ for its right derived functor.
\end{notation}

Recall that $\Underlying(\sSetmark) \eq \Underlying(\sSet) = \Catinf$ by \cite[Proposition~3.1.5.3]{LurHTT}.

\begin{notation}
  For an $\infty$-category $\cat C$, we denote by $\Unstr \colon \Fun(\cat C, \Catinf) \to \Cocart{\cat C}$ the functor induced by (the opposite of) the unstraightening functor of \cite[Theorem~3.2.0.1]{LurHTT}.
\end{notation}

Recall that $\Unstr$ is natural in $\cat C$ by \cite[Corollary~A.32]{GHN}.
We now prove that the limit of a functor valued in $\infty$-categories can be naturally computed as the cocartesian sections of its unstraightening.

\begin{lemma} \label{lemma:limit_sections}
  Let $\cat C$ be an $\infty$-category and $F \colon \cat C \to \Catinf$ a functor.
  Then there is a natural equivalence
  \[ \Funcc{\cat C} \bigl( \cat C, \Unstr(F) \bigr)  \eq  \lim_{\cat C} F \]
  of functors $\Fun(\cat C, \Catinf) \to \Catinf$.
  Furthermore, for a functor $g \colon \cat C' \to \cat C$, the diagram
  \[
  \begin{tikzcd}
    \Funcc{\cat C} \bigl( \cat C, \Unstr(F) \bigr) \rar{g^*} \dar[swap]{\eq} & \Funcc{\cat C'} \bigl( \cat C', g^* \Unstr(F) \bigr) \rar{\eq} & \Funcc{\cat C'} \bigl( \cat C', \Unstr(g^* F) \bigr) \dar{\eq} \\
    \lim_{\cat C} F \ar{rr} & & \lim_{\cat C'} g^* F
  \end{tikzcd}
  \]
  commutes in the $\infty$-category of functors $\Fun(\cat C, \Catinf) \to \Catinf$.
\end{lemma}

\begin{proof}
  Let $\iota \colon \cat C \to \lcone{\cat C}$ be the inclusion and $j \colon \set{\lconept} \to \lcone{\cat C}$ the inclusion of the cone point.
  Then the diagram
  \[
  \begin{tikzcd}
    \Fun(\cat C, \Catinf) \dar{\eq}[swap]{\Unstr} & \lar[swap]{\iota^*} \Fun(\lcone{\cat C}, \Catinf) \dar{\eq}[swap]{\Unstr} \rar{j^*} & \Catinf \dar[equal] \\
    \Cocart{\cat C} & \lar[swap]{\iota^*} \Cocart{\lcone{\cat C}} \rar{j^*} & \Catinf
  \end{tikzcd}
  \]
  commutes and is natural in $\cat C$ by naturality of unstraightening.
  For a functor $F' \colon \lcone{\cat C} \to \Catinf$, we thus have maps
  \[ F'(\lconept)  \eq  \restrict {\Unstr(F')} \lconept  \xlongfrom{j^*}  \Funcc{\lcone{\cat C}} \bigl( \lcone{\cat C}, \Unstr(F') \bigr)  \xlongto{\iota^*}  \Funcc{\cat C} \bigl( \cat C, \restrict {\Unstr(F')} {\cat C} \bigr)  \eq  \Funcc{\cat C} \bigl( \cat C, \Unstr(\restrict {F'} {\cat C}) \bigr) \]
  that are natural in $F'$ and $\cat C$.
  The map $j^*$ is an equivalence by \kerodon{02TC}, so we obtain a natural map
  \[ F'(\lconept)  \longto  \Funcc{\cat C} \bigl( \cat C, \Unstr(\restrict {F'} {\cat C}) \bigr) \]
  which is an equivalence if $F'$ is a limit diagram, since then $\iota^*$ is an equivalence by \cite[Proposition~3.3.3.1]{LurHTT}.
  Note that in that case $\lim_{\cat C} \restrict {F'} {\cat C} = F'(\lconept)$ by definition.
  
  Now consider the right Kan extension functor $\iota_* \colon \Fun(\cat C, \Catinf) \to \Fun(\lcone{\cat C}, \Catinf)$.
  It takes values in limit diagrams by definition.
  Since $\iota$ is fully faithful, the counit $\epsilon$ of the adjunction $\iota^* \dashv \iota_*$ yields an equivalence $\restrict {\iota_*(F)} {\cat C} \eq F$ that is natural in $F$.
  We thus have equivalences
  \[ \lim_{\cat C} F  \eq  \lim_{\cat C} \restrict {\iota_*(F)} {\cat C}  \eq  \Funcc{\cat C} \bigl( \cat C, \Unstr(\restrict {\iota_*(F)} {\cat C}) \bigr)  \eq  \Funcc{\cat C} \bigl( \cat C, \Unstr(F) \bigr) \]
  that are natural in $F$.
  
  Given a functor $g \colon \cat C' \to \cat C$, the mate of the equivalence $(\iota')^* (\lcone{g})^* \eq g^* \iota^*$ is a natural transformation $\alpha \colon (\lcone{g})^* \iota_* \to \iota'_* g^*$ that is given by the induced map $\lim_{\cat C} F \to \lim_{\cat C'} g^* F$ at the cone point.
  We thus obtain the following commutative diagram
  \[
  \begin{tikzcd}
    \lim_{\cat C} F \dar[equal] \ar{rr} & & \lim_{\cat C'} g^* F \dar[equal] \\
    \bigl( \iota_*(F) \bigr) (\lconept) \rar[equal] \dar[swap]{\eq} & \bigl( (\lcone{g})^* \iota_* (F) \bigr) (\lconept) \dar \rar{\alpha} & \bigl( \iota'_* g^* (F) \bigr) (\lconept) \dar{\eq} \\
    \Funcc{\cat C} \bigl( \cat C, \Unstr(F) \bigr) \rar & \Funcc{\cat C'} \bigl( \cat C', \Unstr(g^* F) \bigr) \rar[equal] & \Funcc{\cat C'} \bigl( \cat C', \Unstr(g^* F) \bigr)
  \end{tikzcd}
  \]
  where, for the lower right-hand square, we use that the diagram
  \[
  \begin{tikzcd}
    (\iota')^* (\lcone{g})^* \iota_* \rar{\alpha} \dar{\eq} & (\iota')^* \iota'_* g^* \dar{\epsilon}[swap]{\eq} \\
    g^* \iota^* \iota_* \rar{\eq}[swap]{\epsilon} & g^*
  \end{tikzcd}
  \]
  commutes by \cref{lemma:mates}.
\end{proof}

The following lemma yields an explicit description of $\Unstr(F)$ when $F \colon \cat C \to \Catinf$ is represented by a functor to marked simplicial sets.

\begin{lemma} \label{lemma:limit_sections_nerve}
  Let $\cat C$ be a category and $F \colon \cat C \to \sSetmark$ a functor taking values in fibrant objects.
  Denoting by $\Nervemark[F](\cat C) \in \Fibmark{\cat C}$ the marked relative nerve of \cite[Definition~3.2.5.12]{LurHTT}, there is a natural equivalence of functors $\Fun(\cat C, \sSetmark) \to \Cocart{\cat C}$
  \[ \Nervemark[F](\cat C)  \eq  \Unstr(\upsilon \after F) \]
  where $\upsilon \colon \sSetmark \to \Underlying(\sSetmark) \eq \Catinf$ is the canonical functor.
  Furthermore, for a functor $g \colon \cat C' \to \cat C$, the diagram
  \[
  \begin{tikzcd}
    g^* \Nervemark[F](\cat C) \rar{\eq} \dar[swap]{\iso} & g^* \Unstr(\upsilon \after F) \dar{\eq} \\
    \Nervemark[g^* F](\cat C') \rar{\eq} & \Unstr(\upsilon \after g^* F)
  \end{tikzcd}
  \]
  commutes in the $\infty$-category of functors $\Fun(\cat C, \sSetmark) \to \Cocart{\cat C'}$.
\end{lemma}

\begin{proof}
  By \cite[Corollary~3.2.5.20]{LurHTT}, there is an equivalence $\Nervemark[F](\cat C) \eq \Unstr(\upsilon \after F)$ in $\Cocart{\cat C}$ that is natural in $F$.
  It is furthermore compatible with restriction along $g$ by construction.
\end{proof}

\subsection{The model category of sections}

We are now ready to prove the main result of this appendix.
We begin by defining the category of sections of a functor $F$ taking values in categories.
It is a model for the lax limit of $F$.

\begin{definition} \label{def:Sect}
  For a category $\cat C$ and a pseudofunctor $F \colon \cat C \to \twoCat$, we write $\Sect{\cat C}(F)$ for the category such that
  \begin{itemize}
    \item
    an object $S$ of $\Sect{\cat C}(F)$ consists of an object $S(C) \in F(C)$ for every $C \in \cat C$ and a morphism $S(f) \colon F(f)(S(C)) \to S(C')$ of $F(C')$ for every morphism $f \colon C \to C'$ of $\cat C$,
    \item
    a morphism $\phi \colon S \to S'$ of $\Sect{\cat C}(F)$ consists of a morphism $\phi_C \colon S(C) \to S'(C)$ of $F(C)$ such that $\phi_{C'} \circ S(f) = S'(f) \circ F(f)(\phi_C)$ for every morphism $f \colon C \to C'$ of $\cat C$,
  \end{itemize}
  and identities and composition are defined in the obvious way.
\end{definition}

Note that $\Sect{\cat C}(F)$ is isomorphic to the category of sections of the projection functor $\gc F \to \cat C$ of the Grothendieck construction of $F$.

\begin{notation}
For a strict $2$-category $\cat C$, we denote by $\truncat{1} \cat C$ its underlying category, obtained by forgetting the $2$-morphisms.
This construction is canonically functorial with respect to strict pseudofunctors.
\end{notation}

\begin{proposition} \label{prop:sections}
  Let $\cat J$ be a groupoid and $M \colon \cat J \to \truncat{1} \ModCatL$ a functor taking values in left proper, combinatorial model categories.
  Then there is an equivalence of $\infty$-categories
  \[ \sigma  \colon  \cUnderlying \bigl( \Sect{\cat J}(M) \bigr)  \xlongto{\eq}  \lim_{\cat J} ({\cUnderlying} \after M) \]
  where we equip $\Sect{\cat J}(M)$ with the injective model structure, i.e.\ weak equivalences and cofibrations are defined pointwise.
  The map $\sigma$ is natural in the full subcategory of $\Fun(\cat J, \truncat{1} \ModCatC)$ spanned by the functors that land in $\ModCatL$.
  Furthermore, given a functor $f \colon \cat I \to \cat J$ of groupoids, the diagram
  \[
  \begin{tikzcd}
    \cUnderlying \bigl( \Sect{\cat J}(M) \bigr) \rar{\sigma} \dar & \lim_{\cat J} ({\cUnderlying} \after M) \dar \\
    \cUnderlying \bigl( \Sect{\cat I}(M \after f) \bigr) \rar{\sigma} & \lim_{\cat I} ({\cUnderlying} \after M \after f)
  \end{tikzcd}
  \]
  commutes.
\end{proposition}

\begin{proof}
  First note that, since $M$ is pointwise combinatorial, the injective model structure on $\Sect{\cat J}(M)$ exists by a result of Barwick \cite[Theorem~2.30]{Bar} and is left proper by \cite[Proposition~2.31]{Bar} since $M$ is pointwise left proper.
  Note that it is clear that a map $M \to M'$ in $\Fun(\cat J, \truncat{1} \ModCatC)$ induces a map $\Sect{\cat J}(M) \to \Sect{\cat J}(M')$ that is again in $\ModCatC$.
  
  We write $\gc M$ for the Grothendieck construction of $M$, and say that a morphism $(\alpha, a)$ in $\gc M$ is a \emph{weak equivalence} if $a \colon M(\alpha)(X) \to Y$ is a weak equivalence.
  Furthermore writing $\Nervemark$ for the nerve with the weak equivalences marked, we obtain a canonical map of marked simplicial sets
  \[ \Nervemark \bigl( \cofibobj{\Sect{\cat J}(M)} \bigr)  \longto  \Mapcc{\cat J} \bigl( \cat J, \Nervemark(\gc \cofibobj{M}) \bigr) \]
  where on the right-hand side the pointwise marked edges are marked (and we use that $\cat J$ is a groupoid, so that the structure maps of any object of $\Sect{\cat J}(M)$ consist of isomorphisms).
  We furthermore have $\Nervemark(\gc \cofibobj{M}) \iso \Nervemark[{\Nervemark} \after \cofibobj{M}](\cat J)$, where the latter denotes the marked relative nerve of \cite[Definition~3.2.5.12]{LurHTT}.
  Writing $r$ for a fibrant replacement functor of $\sSetmark$, we thus obtain a map
  \[ \nu  \colon  \Nervemark \bigl( \cofibobj{\Sect{\cat J}(M)} \bigr)  \longto  \Mapcc{\cat J} \bigl( \cat J, \Nervemark[r \after {\Nervemark} \after \cofibobj{M}](\cat J) \bigr) \]
  of marked simplicial sets.
  We note that the marked relative nerve is a right Quillen functor by \cite[Proposition~3.2.5.18]{LurHTT}, so that the target of $\nu$ is fibrant (using that $\Fibmark{\cat J}$ is a $\sSetmark$-enriched model category by \cite[Corollary~3.1.4.3]{LurHTT}).
  In particular $\nu$ factors through a fibrant replacement of the domain, whose underlying simplicial set is a quasi-category representing $\cUnderlying(\Sect{\cat J}(M))$ by definition.
  Similarly, by \cref{lemma:limit_sections,lemma:limit_sections_nerve}, the underlying simplicial set of the target of $\nu$ represents $\lim_{\cat J} (\cUnderlying \after M)$.
  Hence we obtain a natural map
  \[ \sigma  \colon  \cUnderlying \bigl( \Sect{\cat J}(M) \bigr)  \longto  \lim_{\cat J} ({\cUnderlying} \after M) \]
  as desired, which is an equivalence if $\nu$ is a weak equivalence of marked simplicial sets.
  Note that $\nu$ is compatible with restriction along functors $f \colon \cat I \to \cat J$, and hence $\sigma$ is as well (by choosing the factorization through a fibrant replacement of the domain of $\nu$ in some appropriate functor category).
  
  We will now prove that $\nu$ is a weak equivalence, beginning with the case that $M$ takes values in \emph{simplicial} model categories.
  To this end, we write $\Nervesimp(\cat C) \in \sSet$ for the homotopy coherent nerve of a simplicial category $\cat C$, and $\ModCatL[\sSet]$ for the strict $(2,1)$-category of simplicial model categories, simplicial left Quillen functors, and natural isomorphisms.
  Harpaz \cite[Corollary~4.4]{Har} has shown that, given a pseudofunctor $S \colon \cat J \to \ModCatL[\sSet]$ taking values in combinatorial model categories, the following canonical map is an equivalence of quasi-categories
  \[ \Nervesimp \bigl( \bifibobj{\Sect{\cat J}(S)} \bigr)  \xlongto{\eq}  \Mapcc{\cat J} \bigl( \cat J, \markcart{\Nervesimp(\gcbifib S)} \bigr) \]
  where $\gcbifib S$ denotes the full simplicial subcategory of the simplicial Grothendieck construction $\gc S$ spanned by the objects that are fibrant and cofibrant, and $\markcart{X}$ denotes $X$ with the cocartesian edges marked.
  Here we again implicitly use that $\cat J$ is a groupoid, so that the structure maps of any element of $\Sect{\cat J}(S)$ consist of isomorphisms.
  
  We write $\Nervesimpmark$ for the homotopy coherent nerve with the weak equivalences marked.
  By (the dual of) \cite[Lemma~2.1]{Har} we have $\Nervesimpmark(\gcbifib S) = \markcart{\Nervesimp(\gcbifib S)}$, so that there is an induced map $\iota \colon \markcart{\Nervesimp(\gcbifib S)} \to \Nervesimpmark(\gc \cofibobj{S})$.
  From now on assuming that $S$ is a strict functor, there is, by a result of Beardsley--Wong \cite[Theorem~2.13]{BW}, a natural isomorphism $\Nervesimpmark(\gc \cofibobj{S}) \iso \Nervemark[\Nervesimpmark \after \cofibobj{S}](\cat J)$ over $\Nerve(\cat J)$.
  Using \cite[Propositions~3.1.3.5 and 3.1.4.1]{LurHTT}, we see that the composite
  \[ \markcart{\Nervesimp(\gcbifib S)}  \xlongto{\iota}  \Nervesimpmark(\gc \cofibobj{S})  \iso  \Nervemark[{\Nervesimpmark} \after \cofibobj{S}](\cat J)  \longto  \Nervemark[r \after {\Nervesimpmark} \after \cofibobj{S}](\cat J) \]
  is a cocartesian equivalence: over an object $j \in \cat J$ it is given by
  \[ \Nervesimpmark \bigl( \bifibobj{S(j)} \bigr)  \longto  \Nervesimpmark \bigl( \cofibobj{S(j)} \bigr)  \xlongto{\eq}  r \Nervesimpmark \bigl( \cofibobj{S(j)} \bigr) \]
  where the first map is a weak equivalence of marked simplicial sets since it has a simplicial homotopy inverse (with respect to the simplicial enrichment given by $\Mapmarkueq$, see \cite[Corollary~3.1.4.4]{LurHTT}) given by a simplicial fibrant replacement functor.
  
  There is a canonical natural transformation ${\Nervemark} \after \cofibobj{S} \to {\Nervesimpmark} \after \cofibobj{S}$ that is pointwise a weak equivalence of marked simplicial sets: for any simplicial model category $\cat S$, the maps $\Nervemark(\bifibobj{\cat S}) \to \Nervemark(\cofibobj{\cat S})$ and $\Nervesimpmark(\bifibobj{\cat S}) \to \Nervesimpmark(\cofibobj{\cat S})$ are simplicial homotopy equivalences (again with respect to $\Mapmarkueq$), and the canonical map $\Nervemark(\bifibobj{\cat S}) \to \Nervesimpmark(\bifibobj{\cat S})$ is a weak equivalence by \cite[Proposition~1.3.4.7]{LurHA}.
  Factoring the composite map ${\Nervemark} \after \cofibobj{S} \to {\Nervesimpmark} \after \cofibobj{S} \to r \after {\Nervesimpmark} \after \cofibobj{S}$ as a trivial cofibration followed by a trivial fibration, we obtain the upper row in the diagram
  \[
  \begin{tikzcd}
    {\Nervemark} \after \cofibobj{S} \rar[tail]{\eq} \dar[swap]{\eq} & \dlar[dashed]{\eq} F \rar[two heads]{\eq} & r \after {\Nervesimpmark} \after \cofibobj{S} \\
    r \after {\Nervemark} \after \cofibobj{S}
  \end{tikzcd}
  \]
  where a dashed weak equivalence as indicated exists since $r \after {\Nervemark} \after \cofibobj{S}$ is pointwise fibrant.
  Applying the marked relative nerve followed by $\Mapcc{\cat J} ( \cat J, \blank )$, we obtain a commutative diagram of marked simplicial sets
  \[
  \begin{tikzcd}
    \Nervesimpmark \bigl( \bifibobj{\Sect{\cat J}(S)} \bigr) \rar{\eq} \dar & \Mapcc{\cat J} \bigl( \cat J, \markcart{\Nervesimp(\gcbifib S)} \bigr) \dar \drar[bend left = 20, start anchor = east]{\eq} \\
    \Nervesimpmark \bigl( \cofibobj{\Sect{\cat J}(S)} \bigr) \rar & \Mapcc{\cat J} \bigl( \cat J, \Nervesimpmark(\gc \cofibobj{S}) \bigr) \rar & \Mapcc{\cat J} \bigl( \cat J, \Nervemark[r \after {\Nervesimpmark} \after \cofibobj{S}](\cat J) \bigr) \\
    \Nervemark \bigl( \cofibobj{\Sect{\cat J}(S)} \bigr) \rar \uar \drar[bend right = 20, end anchor = west][swap]{\nu} & \Mapcc{\cat J} \bigl( \cat J, \Nervemark(\gc \cofibobj{S}) \bigr) \rar \uar \dar & \dlar[bend left = 20, end anchor = east]{\eq} \Mapcc{\cat J} \bigl( \cat J, \Nervemark[F](\cat J) \bigr) \uar[swap]{\eq} \\
     & \Mapcc{\cat J} \bigl( \cat J, \Nervemark[r \after {\Nervemark} \after \cofibobj{S}](\cat J) \bigr)
  \end{tikzcd}
  \]
  such that the indicated maps are weak equivalences.
  By the same argument as above, the two left-hand vertical maps are weak equivalences.
  Hence $\nu$ is a weak equivalence as claimed.
  
  We now complete the proof by deducing that $\sigma$ is a weak equivalence in the non-simplicial case as well.
  To this end we use that Dugger \cite[Theorem~1.2]{Dug} has constructed a strict functor $\simp$ from the full subcategory of $\ModCatL$ spanned by the left proper, combinatorial model categories to $\ModCatL[\sSet]$ such that $\simp \cat N$ is naturally Quillen equivalent to $\cat N$.
  It sends a model category $\cat N$ to its category of simplicial objects $\simp \cat N \defeq \Fun(\opcat{\Delta}, \cat N)$, equipped with the ``Reedy hocolim model structure'' whose cofibrations are the Reedy cofibrations and whose weak equivalences are those natural transformations that induce a weak equivalence on homotopy colimits.
  (That this construction is indeed functorial follows from \cite[Proposition~15.4.1 and Theorem~19.4.5]{Hir}.)
  Moreover it follows from \cite[Theorem~15.6.27 and Theorem~15.3.4~(2)]{Hir}, \cite[Corollary~1.54]{AR}, and \cite[Theorem~4.7]{Bar} that $\simp \cat N$ is again left proper and combinatorial.
  The natural Quillen equivalence between $\simp \cat N$ and $\cat N$ is given by the left Quillen functor $\cat N \to \simp \cat N$ induced by pulling back along $\opcat{\Delta} \to *$.
  Thus we obtain a commutative diagram
  \[
  \begin{tikzcd}
    \cUnderlying \bigl( \Sect{\cat J}(M) \bigr) \rar{\sigma} \dar[swap]{\eq} & \lim_{\cat J} ({\cUnderlying} \after M) \dar{\eq} \\
    \cUnderlying \bigl( \Sect{\cat J}(\simp M) \bigr) \rar{\sigma} & \lim_{\cat J} ({\cUnderlying} \after \simp M)
  \end{tikzcd}
  \]
  where the lower horizontal map is an equivalence by the discussion above.
  The two vertical maps are also equivalences, using that $\Sect{\cat J}(\blank)$ preserves Quillen equivalences.
  This completes the proof.
\end{proof}

The following provides a monoidal version of \cref{prop:sections}.

\begin{corollary} \label{cor:sections_monoidal}
  Let $\cat J$ be a groupoid and $M \colon \cat J \to \truncat{1} \ModCatmon$ a functor taking values in left proper, combinatorial model categories.
  Then the equivalence of \cref{prop:sections} can be lifted to a natural equivalence of functors $\Fun(\cat J, \truncat{1} \ModCatmon) \to \Catinfmon$
  \[ \monUnderlying \bigl( \Sect{\cat J}(M) \bigr)  \eq  \lim_{\cat J} ({\monUnderlying} \after M) \]
  where we equip $\Sect{\cat J}(M)$ with the pointwise symmetric monoidal structure.
\end{corollary}

\begin{proof}
  First note that it is clear that $\Sect{\cat J}(M)$ is a symmetric monoidal model category and that as such $\Sect{\cat J}$ is functorial in $M$.
  Now choose a strictification of the pseudofunctor $\ModCatmon \times \Finpt \to \ModCatC$ given by $(\cat M, n) \mapsto \cat M^{\times n}$.
  This induces a strict functor
  \[ \Fun(\cat J, \truncat{1} \ModCatmon)  \longto  \Fun \bigl( \Finpt, \Fun(\cat J, \truncat{1} \ModCatC) \bigr) \]
  and we will write $M^\tensor \colon \Finpt \to \Fun(\cat J, \truncat{1} \ModCatC)$ for the image of $M$.
  Then there is, by \cref{prop:sections} and definition of $\monUnderlying$, an equivalence of functors $\Finpt \to \Catinf$
  \[ \monUnderlying \bigl( \Sect{\cat J}(M) \bigr)  \eq  {\cUnderlying} \after {\Sect{\cat J}} \after M^\tensor  \eq  {\lim_{\cat J}} \after (\cUnderlying)_* \after M^\tensor  \eq  \lim_{\cat J} \monUnderlying(M) \]
  that is natural in $M$ (for the first equivalence we use that $\Sect{\cat J}$ preserves products).
\end{proof}

\end{appendices}

\section*{References}

\printbibliography[keyword=this,heading=subbibliography,title={This series}]
\printbibliography[notkeyword=this,heading=subbibliography,title={Other}]

@article{AGH,
  author = {Abellán, Fernando and Gagna, Andrea and Haugseng, Rune},
  title = {Straightening for lax transformations and adjunctions of $(\infty, 2)$-categories},
  volume = {31},
  DOI = {10.1007/s00029-025-01084-z},
  eid = {85},
  journal = {Selecta Mathematica},
  year = {2025}
}

@book{AR,
  author = {Adámek, Jiří and Rosicky, Jiří},
  title = {Locally Presentable and Accessible Categories},
  series = {London Mathematical Society Lecture Note Series},
  number = {189},
  DOI = {10.1017/cbo9780511600579},
  publisher = {Cambridge University Press},
  year = {1994}
}

@article{ABG,
  title = {Parametrized spectra, multiplicative Thom spectra and the twisted Umkehr map},
	author = {Ando, Matthew and Blumberg, Andrew J. and Gepner, David},
	doi = {10.2140/gt.2018.22.3761},
	journal = {Geometry \& Topology},
	number = {7},
	pages = {3761--3825},
	volume = {22},
	year = {2018},
}

@article {BMR,
  AUTHOR = {Barthel, Tobias and May, J. P. and Riehl, Emily},
  TITLE = {Six model structures for DG-modules over DGAs: model category theory in homological action},
  JOURNAL = {New York Journal of Mathematics},
  VOLUME = {20},
  YEAR = {2014},
  PAGES = {1077--1159},
  URL = {https://nyjm.albany.edu/j/2014/20_1077.html}
}

@article{Bar,
  author = {Barwick, Clark},
  title = {On left and right model categories and left and right Bousfield localizations},
  volume = {12},
  DOI = {10.4310/hha.2010.v12.n2.a9},
  number = {2},
  journal = {Homology, Homotopy and Applications},
  year = {2010},
  pages = {245--320}
}

@article {BW,
  AUTHOR = {Beardsley, Jonathan and Wong, Liang Ze},
  TITLE = {The operadic nerve, relative nerve and the Grothendieck construction},
  JOURNAL = {Theory and Applications of Categories},
  VOLUME = {34},
  YEAR = {2019},
  EID = {13},
  PAGES = {349--374},
  URL = {http://www.tac.mta.ca/tac/volumes/34/13/34-13abs.html}
}

@article{Ben,
  author = {Ben-Moshe, Shay},
  title = {Categorical ambidexterity},
  year = {2026},
  journal = {New York Journal of Mathematics},
  volume = {32},
  pages = {371--390},
  url = {https://nyjm.albany.edu/j/2026/32-17.html}
}

@unpublished{Ber,
  author = {Berglund, Alexander},
  title = {Poincaré duality fibrations and graph complexes},
  note = {In progress}
}

@Article{BM,
  Author = {Berglund, Alexander and Madsen, Ib},
  Title = {Rational homotopy theory of automorphisms of manifolds},
  Journal = {Acta Mathematica},
  Volume = {224},
  Number = {1},
  Pages = {67--185},
  Year = {2020},
  doi = {10.4310/ACTA.2020.v224.n1.a2}
}

@online{BS24,
  author = {Berglund, Alexander and Stoll, Robin},
  title = {Equivariant algebraic models for relative self-equivalences and block diffeomorphisms},
  year = {2025},
  note = {Preprint},
  eprinttype = {arxiv},
  eprint = {2501.01865v1}
}

@article{BZ,
  author = {Berglund, Alexander and Zeman, Tomáš},
  title = {Algebraic models for classifying spaces of fibrations},
  volume = {29},
  DOI = {10.2140/gt.2025.29.3567},
  number = {7},
  journal = {Geometry \& Topology},
  year = {2025},
  pages = {3567--3634}
}

@article{BG,
  author = {Bousfield, A. K. and Gugenheim, V. K. A. M.},
  title = {On PL De Rham theory and rational homotopy type},
  volume = {8},
  number = {179},
  journal = {Memoirs of the American Mathematical Society},
  year = {1976}
}

@book{BK,
  author = {Bousfield, A. K. and Kan, D. M.},
  title = {Homotopy Limits, Completions and Localizations},
  series = {Lecture Notes in Mathematics},
  number = {304},
  publisher = {Springer},
  year = {1972},
  DOI = {10.1007/978-3-540-38117-4}
}

@article{Bra21,
  author = {Braunack-Mayer, Vincent},
  title = {Combinatorial parametrised spectra},
  volume = {21},
  DOI = {10.2140/agt.2021.21.801},
  number = {2},
  journal = {Algebraic \& Geometric Topology},
  year = {2021},
  pages = {801--891}
}

@online{Bra,
  author = {Braunack-Mayer, Vincent},
  title = {Strict algebraic models for rational parametrised spectra~II},
  year = {2020},
  note = {Preprint},
  eprinttype = {arxiv},
  eprint = {2011.06307v1}
}

@article{BS,
  author = {Brown, Jr., Edgar H. and Szczarba, Robert H.},
  title = {Rational and real homotopy theory with arbitrary fundamental groups},
  volume = {71},
  DOI = {10.1215/s0012-7094-93-07111-6},
  number = {1},
  journal = {Duke Mathematical Journal},
  year = {1993},
  pages = {299--316}
}

@article{CCRY,
  author = {Carmeli, Shachar and Cnossen, Bastiaan and Ramzi, Maxime and Yanovski, Lior},
  title = {Characters and transfer maps via categorified traces},
  journal = {Forum of Mathematics Sigma},
  volume = {13},
  eid = {e93},
  DOI = {10.1017/fms.2025.23},
  year = {2025}
}

@book{Cis,
  author = {Cisinski, Denis-Charles},
  title = {Higher Categories and Homotopical Algebra},
  DOI = {10.1017/9781108588737},
  publisher = {Cambridge University Press},
  year = {2019},
  series = {Cambridge Studies in Advanced Mathematics},
  number = {180}
}

@online{Cno,
  author = {Bastiaan Cnossen},
  title = {Twisted ambidexterity in equivariant homotopy theory},
  year = {2023},
  note = {Preprint},
  eprinttype = {arxiv},
  eprint = {2303.00736v2}
}

@online{DM,
  author = {Di Liberti, Ivan and Meadows, Nicholas},
  title = {Classifying Infinity Topoi via Weighted Limits},
  year = {2025},
  note = {Preprint},
  eprinttype = {arxiv},
  eprint = {2512.15613v2}
}

@article{Dug,
  author = {Dugger, Daniel},
  title = {Replacing model categories with simplicial ones},
  volume = {353},
  DOI = {10.1090/s0002-9947-01-02661-7},
  number = {12},
  journal = {Transactions of the American Mathematical Society},
  year = {2001},
  pages = {5003--5027}
}

@article{DWW,
  author = {Dwyer, W. and Weiss, M. and Williams, B.},
  title = {A parametrized index theorem for the algebraic $K$-theory Euler class},
  volume = {190},
  DOI = {10.1007/bf02393236},
  number = {1},
  journal = {Acta Mathematica},
  year = {2003},
  pages = {1--104}
}

@article{FMT,
  author = {Félix, Yves and Murillo, Aniceto and Tanré, Daniel},
  title = {Fibrewise stable rational homotopy},
  volume = {3},
  DOI = {10.1112/jtopol/jtq023},
  number = {4},
  journal = {Journal of Topology},
  year = {2010},
  pages = {743--758}
}

@book{GR,
  author = {Gaitsgory, Dennis and Rozenblyum, Nick},
  title = {A Study in Derived Algebraic Geometry},
  subtitle = {Volume~I: Correspondences and Duality},
  DOI = {10.1090/surv/221.1},
  series = {Mathematical Surveys and Monographs},
  number = {221},
  publisher = {American Mathematical Society},
  year = {2017},
}

@incollection{Gep,
  author = {Gepner, David},
  title = {An introduction to higher categorical algebra},
  booktitle = {Handbook of Homotopy Theory},
  editor = {Miller, Haynes},
  chapter = {13},
  pages = {487--548},
  DOI = {10.1201/9781351251624},
  publisher = {CRC Press},
  year = {2020}
}

@article{GHK,
  author = {Gepner, David and Haugseng, Rune and Kock, Joachim},
  title = {$\infty$-Operads as Analytic Monads},
  volume = {2022},
  DOI = {10.1093/imrn/rnaa332},
  number = {16},
  journal = {International Mathematics Research Notices},
  year = {2021},
  pages = {12516--12624}
}

@article{GHN,
  author = {Gepner, David and Haugseng, Rune and Nikolaus, Thomas},
  title = {Lax Colimits and Free Fibrations in $\infty$-Categories},
  volume = {22},
  DOI = {10.4171/dm/593},
  journal = {Documenta Mathematica},
  year = {2017},
  pages = {1225--1266}
}

@article{Goo,
  author = {Goodwillie, Thomas G.},
  title = {Relative algebraic $K$-theory and cyclic homology},
  volume = {124},
  DOI = {10.2307/1971283},
  number = {2},
  journal = {Annals of Mathematics},
  year = {1986},
  pages = {347--402}
}

@article{GHT,
  author = {Gómez-Tato, Antonio and Halperin, Stephen and Tanré, Daniel},
  title = {Rational homotopy theory for non-simply connected spaces},
  volume = {352},
  DOI = {10.1090/s0002-9947-99-02463-0},
  number = {4},
  journal = {Transactions of the American Mathematical Society},
  year = {2000},
  pages = {1493--1525}
}

@article {Har,
  AUTHOR = {Harpaz, Yonatan},
  TITLE = {Lax limits of model categories},
  JOURNAL = {Theory and Applications of Categories},
  VOLUME = {35},
  YEAR = {2020},
  NUMBER = {25},
  PAGES = {959--978},
  URL = {http://www.tac.mta.ca/tac/volumes/35/25/35-25abs.html}
}

@online{Har1,
  author = {Harpaz, Yonatan},
  title = {The Cobordism Hypothesis in Dimension 1},
  year = {2012},
  note = {Preprint},
  eprinttype = {arxiv},
  eprint = {1210.0229v1}
}

@online{HNS,
  author = {Harpaz, Yonatan and Nikolaus, Thomas and Saunier, Victor},
  title = {Trace methods for stable categories I: The linear approximation of algebraic K-theory},
  year = {2024},
  note = {Preprint},
  eprinttype = {arxiv},
  eprint = {2411.04743v2}
}

@article{Hau,
  author = {Haugseng, Rune},
  title = {On lax transformations, adjunctions, and monads in $(\infty, 2)$-categories},
  volume = {5},
  DOI = {10.21136/hs.2021.07},
  number = {1},
  journal = {Higher Structures},
  year = {2021},
  pages = {244--281}
}

@inproceedings{Hes,
  author = {Hess, Kathryn},
  title = {Rational homotopy theory: a brief introduction},
  booktitle = {Interactions between Homotopy Theory and Algebra},
  editor = {Avramov, Luchezar L. and Christensen, J. Daniel and Dwyer, William G. and Mandell, Michael A. and Shipley, Brooke E.},
  series = {Contemporary Mathematics},
  number = {436},
  publisher = {American Mathematical Society},
  year = {2007},
  pages = {175--202},
  DOI = {10.1090/conm/436/08409}
}

@article{Hil,
  author = {Hilton, Peter},
  title = {On G-spaces},
  volume = {7},
  number = {1},
  journal = {Boletim da Sociedade Brasileira de Matemática},
  year = {1976},
  pages = {65--73},
  DOI = {10.1007/bf02584848}
}

@inproceedings{Hil75,
  author = {Hilton, Peter},
  title = {Nilpotent actions on nilpotent groups},
  booktitle = {Algebra and Logic},
  series = {Lecture Notes in Mathematics},
  number = {450},
  editor = {Crossley, John N.},
  publisher = {Springer},
  year = {1975},
  pages = {174--196},
  DOI = {10.1007/bfb0062856}
}

@book{HMR,
  title = {Localization of Nilpotent Groups and Spaces},
  author = {Hilton, Peter and Mislin, Guido and Roitberg, Joe},
  series = {North-Holland Mathematics Studies},
  number = {15},
  year = {1975},
  publisher = {Elsevier}
}

@article{HRS,
  author = {Hilton, Peter and Roitberg, Joseph and Singer, David},
  title = {On $G$-spaces, Serre classes and $G$-nilpotency},
  journal = {Mathematical Proceedings of the Cambridge Philosophical Society},
  volume = {84},
  number = {3},
  year = {1978},
  pages = {443--454},
  DOI = {10.1017/s0305004100055274}
}

@article{Hin,
  author = {Hinich, Vladimir},
  title = {Dwyer--Kan localization revisited},
  volume = {18},
  DOI = {10.4310/hha.2016.v18.n1.a3},
  number = {1},
  journal = {Homology, Homotopy and Applications},
  year = {2016},
  pages = {27--48}
}

@book{Hir,
  title = {Model Categories and Their Localizations},
  author = {Hirschhorn, Philip S.},
  series = {Mathematical Surveys and Monographs},
  number = {99},
  year = {2003},
  publisher = {American Mathematical Society},
  doi = {10.1090/surv/099}
}

@online{HL,
  author = {Hopkins, Michael and Lurie, Jacob},
  title = {Ambidexterity in $K(n)$-Local Stable Homotopy Theory},
  date = {2013-12-19},
  note = {Unpublished manuscript},
  url = {https://people.math.harvard.edu/~lurie/papers/Ambidexterity.pdf}
}

@book{Hov,
  author = {Hovey, Mark},
  title = {Model Categories},
  year = {2007},
  series = {Mathematical Surveys and Monographs},
  number = {63},
  publisher = {American Mathematical Society},
  doi = {10.1090/surv/063}
}

@article{HSS,
  author = {Hoyois, Marc and Scherotzke, Sarah and Sibilla, Nicolò},
  title = {Higher traces, noncommutative motives, and the categorified Chern character},
  volume = {309},
  DOI = {10.1016/j.aim.2017.01.008},
  journal = {Advances in Mathematics},
  year = {2017},
  pages = {97--154}
}

@article{Kel21,
  author = {Keller, Bernhard},
  title = {Hochschild (Co)homology and Derived Categories},
  volume = {47},
  DOI = {10.1007/s41980-021-00556-0},
  number = {1 supplement},
  journal = {Bulletin of the Iranian Mathematical Society},
  year = {2021},
  pages = {57--83}
}

@article{Kel98,
  author = {Keller, Bernhard},
  title = {Invariance and localization for cyclic homology of DG algebras},
  volume = {123},
  DOI = {10.1016/s0022-4049(96)00085-0},
  journal = {Journal of Pure and Applied Algebra},
  year = {1998},
  pages = {223--273}
}

@article{Kel99,
  author = {Keller, Bernhard},
  title = {On the cyclic homology of exact categories},
  volume = {136},
  DOI = {10.1016/s0022-4049(97)00152-7},
  number = {1},
  journal = {Journal of Pure and Applied Algebra},
  year = {1999},
  pages = {1--56}
}

@inproceedings{KS,
  author = {Kelly, G. M.  and  Street, Ross},
  title = {Review of the elements of 2-categories},
  DOI = {10.1007/bfb0063101},
  booktitle = {Category Seminar},
  booktitleaddon = {Sydney, NSW, Australia, 1972/73},
  editor = {Kelly, Gregory M.},
  series = {Lecture Notes in Mathematics},
  number = {420},
  publisher = {Springer},
  year = {1974},
  pages = {75--103}
}

@article{Kle,
  author = {Klein, John R.},
  title = {The dualizing spectrum of a topological group},
  volume = {319},
  DOI = {10.1007/pl00004441},
  number = {3},
  journal = {Mathematische Annalen},
  year = {2001},
  pages = {421--456}
}

@article{KW,
  author = {Klein, John R and Williams, E Bruce},
  title = {Homotopical intersection theory I},
  volume = {11},
  DOI = {10.2140/gt.2007.11.939},
  number = {2},
  journal = {Geometry \& Topology},
  year = {2007},
  pages = {939--977}
}

@article{LP,
  author = {Lack, Stephen and Paoli, Simona},
  title = {2-nerves for bicategories},
  volume = {38},
  DOI = {10.1007/s10977-007-9013-2},
  number = {2},
  journal = {K-Theory},
  year = {2008},
  pages = {153--175}
}

@book{Lan,
  author = {Land, Markus},
  title = {Introduction to Infinity-Categories},
  DOI = {10.1007/978-3-030-61524-6},
  series = {Compact Textbooks in Mathematics},
  publisher = {Birkhäuser},
  year = {2021}
}

@article{LM,
  author = {Lind, John A. and Malkiewich, Cary},
  title = {The transfer map of free loop spaces},
  volume = {371},
  DOI = {10.1090/tran/7497},
  number = {4},
  journal = {Transactions of the American Mathematical Society},
  year = {2018},
  pages = {2503--2552}
}

@book{Lod,
  author = {Loday, Jean-Louis},
  title = {Cyclic Homology},
  edition = {2},
  DOI = {10.1007/978-3-662-11389-9},
  series = {Grundlehren der mathematischen Wissenschaften},
  number = {301},
  publisher = {Springer},
  year = {1998}
}

@inbook{Lod15,
  author = {Loday, Jean-Louis},
  title = {Free loop space and homology},
  booktitle = {Free Loop Spaces in Geometry and Topology},
  editor = {Latschev, Janko and Oancea, Alexandru},
  publisher = {European Mathematical Society},
  series = {IRMA Lectures in Mathematics and Theoretical Physics},
  number = {24},
  year = {2015},
  pages = {137--156},
  DOI = {10.4171/153/4},
}

@online{LR25,
  author       = {Loubaton, Félix and Ruit, Jaco},
  title        = {On the squares functor and the Gaits\-gory–Rozen\-blyum conjectures},
  year         = {2025},
  note         = {Preprint},
  eprinttype   = {arxiv},
  eprint       = {2507.07807v2}
}

@online {LurHA,
  author = "Lurie, Jacob",
  title = "Higher Algebra",
  url = "http://www.math.harvard.edu/~lurie/papers/HA.pdf",
  note = "Unpublished manuscript",
  date = "2017-09-18"
}

@book {LurHTT,
  author = "Lurie, Jacob",
  title = "Higher Topos Theory",
  publisher = "Princeton University Press",
  series = "Annals of Mathematics Studies",
  number = "170",
  year = "2009",
  eprinttype = "arxiv",
  eprint = "math/0608040"
}

@article{Macp,
  author = {Macpherson, Andrew W},
  title = {A bivariant Yoneda lemma and $(\infty, 2)$-categories of correspondences},
  volume = {22},
  DOI = {10.2140/agt.2022.22.2689},
  number = {6},
  journal = {Algebraic \& Geometric Topology},
  year = {2022},
  pages = {2689--2774}
}

@book{MP,
  author = {May, J. P. and Ponto, K.},
  title = {More Concise Algebraic Topology},
  subtitle = {Localization, Completion, and Model Categories},
  publisher = {University of Chicago Press},
  series = {Chicago Lectures in Mathematics},
  year = {2011},
  DOI = {10.7208/chicago/9780226511795.001.0001}
}

@article{McC,
  author = {McCarthy, Randy},
  title = {The cyclic homology of an exact category},
  volume = {93},
  DOI = {10.1016/0022-4049(94)90091-4},
  number = {3},
  journal = {Journal of Pure and Applied Algebra},
  year = {1994},
  pages = {251--296}
}

@online{NS,
  author = {Naef, Florian and Safronov, Pavel},
  title = {Simple homotopy invariance of the loop coproduct},
  note = {Preprint},
  eprinttype = {arxiv},
  eprint = {2406.19326v1},
  year = {2024}
}

@article{Pon,
  author = {Ponto, Kate},
  title = {Fixed Point Theory and Trace for Bicategories},
  DOI = {10.24033/ast.815},
  journal = {Astérisque},
  year = {2018},
  volume = {333}
}

@article{Pow,
  author = {Power, A. J.},
  title = {A general coherence result},
  volume = {57},
  DOI = {10.1016/0022-4049(89)90113-8},
  number = {2},
  journal = {Journal of Pure and Applied Algebra},
  year = {1989},
  pages = {165--173}
}

@article{SS,
  author = {Schwede, Stefan and Shipley, Brooke E.},
  title = {Algebras and Modules in Monoidal Model Categories},
  volume = {80},
  DOI = {10.1112/s002461150001220x},
  number = {2},
  journal = {Proceedings of the London Mathematical Society},
  year = {2000},
  pages = {491--511}
}

@thesis {Sci,
  author = {Sciarappa, Luke},
  title = {Model categories in equivariant rational homotopy theory},
  type = {SPUR Final Paper},
  institution = {Massachusetts Institute of Technology},
  year = {2017},
  url = {https://math.mit.edu/research/undergraduate/spur/documents/2017Sciarappa.pdf}
}

@article{Shi,
  author = {Shipley, Brooke E.},
  title = {Convergence of the homology spectral sequence of a cosimplicial space},
  volume = {118},
  DOI = {10.1353/ajm.1996.0004},
  number = {1},
  journal = {American Journal of Mathematics},
  year = {1996},
  pages = {179--207}
}

@article {Shu,
  author = {Shulman, Michael},
  title = {Comparing composites of left and right derived functors},
  journal = {New York Journal of Mathematics},
  volume = {17},
  year = {2011},
  pages = {75--125},
  url = {http://nyjm.albany.edu/j/2011/17_75.html}
}

@article{Sul,
  author = {Dennis Sullivan},
  title = {Infinitesimal computations in topology},
  journal = {Publications ma\-thé\-ma\-tiques de l'IHÉS},
  year = {1977},
  volume = {47},
  number = {1},
  pages = {269--331},
  doi = {10.1007/bf02684341}
}

@online{I,
  shorthand = {Part~I},
  author = {Naef, Florian and Stoll, Robin},
  title = {A rational model for the fiberwise THH transfer I: Sullivan algebras},
  note = {Preprint},
  year = {2026},
  eprinttype = {arxiv},
  eprint = {2604.02516v2}
}

@online{II,
  shorthand = {Part~II},
  author = {Naef, Florian and Stoll, Robin},
  title = {A rational model for the fiberwise THH transfer II: $A_\infty$-algebras},
  note = {Preprint},
  year = {2026},
  eprinttype = {arxiv},
  eprint = {2604.24709v1}
}

@online{Transfer,
  author = {Naef, Florian and Stoll, Robin},
  title = {The Becker-Gottlieb transfer is not functorial: An example using $C_\infty$-algebras},
  year = {2026},
  note = {Preprint},
  eprinttype = {arxiv},
  eprint = {2610.18718v2}
}

\end{document}